\documentclass[preprint,12pt,a4paper]{elsarticle}
\usepackage[utf8]{inputenc}
\usepackage[english]{babel}
\usepackage{esvect}
\usepackage{tikz-cd}
\usetikzlibrary{shapes,arrows}
\tikzstyle{block}=[draw opacity=0.7,line width=1.4cm]
\usepackage{graphicx}
\usepackage{a4wide}
\usepackage{amsthm}
\usepackage{mathrsfs}
\usepackage[hidelinks]{hyperref}
\usepackage{amssymb}
\usepackage{arydshln}
\usepackage{enumitem}
\usepackage{mathtools}
\usepackage{accents}
\usepackage{etoolbox}
\usepackage{mathtools}
\usepackage{soul}

\theoremstyle{definition}
\newtheorem{definition}{Definition}[section]

\newtheorem{example}[definition]{Example}
\newtheorem{invariant}{Invariant}
\newtheorem*{remark*}{Remark}
\newtheorem*{claim*}{Claim}
\theoremstyle{remark}
\newtheorem{remark}[definition]{Remark}

\theoremstyle{plain}
\newtheorem{theorem}[definition]{Theorem}
\newtheorem{corollary}[definition]{Corollary}
\newtheorem{prop}[definition]{Proposition}
\newtheorem{claim}[definition]{Claim}
\newtheorem{observation}[definition]{Observation}
\newtheorem{lemma}[definition]{Lemma} 
\newtheorem{conjecture}[definition]{Conjecture}
\newtheorem{question}[definition]{Question}
\newtheorem{problem}[definition]{Problem}
\newtheorem{fact}[definition]{Fact}

\def\todo#1{}
\def\marginpar#1{}
\def\nbrel#1#2{R\ifstrempty{#1}{}{_{#1}}\ifstrempty{#2}{}{^{#2}}}
\def\nbpred#1#2{p\ifstrempty{#1}{}{_{#1}}\ifstrempty{#2}{}{^{#2}}}
\def\rel#1#2{\nbrel{\ifstrempty{#1}{}{\str{#1}}}{#2}}
\def\pred#1#2{\nbpred{\ifstrempty{#1}{}{\str{#1}}}{#2}}
\def\str#1{\mathbf {#1}}
\def\Alt{\mathop{\mathrm{Alt}}\nolimits}
\def\Sym{\mathop{\mathrm{Sym}}\nolimits}

\def\K{\mathcal{K}}

\def\Fraisse{Fra\"{\i}ss\' e}
\def\age{\mathop{\mathrm{Age}}\nolimits}
\def\Aut{\mathop{\mathrm{Aut}}\nolimits}
\def\Emb{\mathop{\mathrm{Emb}}\nolimits}

\def\Linorder{\mathcal {LO}}

\def\func#1#2{\nbfunc{\ifstrempty{#1}{}{\str{#1}}}{#2}}
\def\nbfunc#1#2{F\ifstrempty{#1}{}{_{#1}}\ifstrempty{#2}{}{^{#2}}}
\def\Age{\age}
\def\dom{\mathrm{Dom}}
\def\range{\mathrm{Range}}
\newcommand{\ie}{i.\,e.}

\newcommand{\eg}{e.\,g.}

\def\powerset#1{\mathscr{P}(#1)}
\DeclareMathOperator{\aritya}{\text{a}}
\def\arity#1{\aritya(#1)}

\def\BA{\mathcal{BA}}

\begin{document}
\let\WriteBookmarks\relax
\def\floatpagepagefraction{1}
\def\textpagefraction{.001}

\title{Twenty years of Ne\v set\v ril's classification programme of Ramsey classes}
\author[1]{Jan Hubička}
\ead{hubicka@kam.mff.cuni.cz}
\ead[url]{http:/www.ucw.cz/~hubicka}
\affiliation[1]{organization={Department of Applied Mathematics, Charles University},
	addressline={Malostranské náměstí~25},
	city={Praha~1},
	postcode={118~00},
	country={Czech Republic}}
\author[2]{Matěj Konečný}
\affiliation[2]{organization={Institute of Algebra, TU Dresden}, city={Dresden}, country={Germany}}
\ead{matej.konecny@tu-dresden.de}

\begin{abstract}
	In the 1970s, structural Ramsey theory emerged as a new branch of combinatorics. This development came with the isolation of the concepts of the $\str{A}$-Ramsey property and Ramsey class. Following the influential Nešetřil--Rödl theorem, several Ramsey classes have been identified. In the 1980s, Nešetřil, inspired by a seminar of Lachlan, discovered a crucial connection between Ramsey classes and \Fraisse{} classes, and, in his 1989 paper, connected the classification programme of homogeneous structures to structural Ramsey theory. In 2005, Kechris, Pestov, and Todor\v{c}evi\'{c} revitalized the field by connecting Ramsey classes to topological dynamics. This breakthrough motivated Nešetřil to propose a program for classifying Ramsey classes. We review the progress made on this program in the past two decades, list open problems, and discuss recent extensions to new areas, namely the extension property for partial automorphisms (EPPA), and big Ramsey structures.
\end{abstract}

\maketitle

\tableofcontents

\section{Introduction}
Ne\v set\v ril and R\"odl can be considered to be the founders of the area of
structural Ramsey theory.  This survey focuses on Nešetřil's classification
programme of Ramsey classes which was proposed by Ne\v set\v ril 20 years
ago~\cite{Nevsetril2005} and is an essential part of the immensely fruitful area of structural Ramsey theory which has
a number of beautiful results, as well as applications and connections to areas such as model theory, topological dynamics, or computer science.  We start
with a historical perspective introducing the key notions and then discuss the
present state, open problems, and recent extensions of the programme.  Appendix~\ref{appendix:partite} then gives a self-contained
presentation of the partite construction in the full strength used today for identifying Ramsey classes.

The historical part is organized roughly chronologically, since we believe that all the notions and results
feel natural when viewed in this context. However, we do make occasional exceptions to this rule to make
the presentation smoother.
The survey is intended to be self-contained, though some prior knowledge of
Ramsey theory might help.

\paragraph{Note} This document is the postprint version of the survey published in Computer Science Review~\cite{hubicka2025twenty}, updated on October 31, 2025, with several added references. The reference list, however, is still far from being exhaustive.

\section{1970s: From Ramsey theorem to its structural forms}
\label{sec:1970s}
Following the standard set-theoretic convention, we identify every integer $n$ with the set $\{0,\allowbreak 1,\allowbreak \ldots,\allowbreak  n-1\}$. Given a set $S$ and an integer $k$
we denote by $\binom{S}{k}$ the set of all subsets of $S$ with precisely $k$ elements.
Given a set $S$ and an integer $r$, an \emph{$r$-colouring of $S$} is a function $\chi\colon S\to r$.
Given integers $N$, $n$, $k$, and $r$ the \emph{Erd\H os--Rado partition arrow},
written as $N\longrightarrow (n)^k_r$, is a shortcut for the following statement:
\begin{quote}
	For every $r$-colouring $\chi\colon\binom{N}{k}\to r$ there exists a monochromatic set $S\in \binom{N}{n}$.
\end{quote}
Here, \emph{monochromatic} means that $\chi$ restricted to ${\binom{S}{k}}$ is a constant function.
With this notation, the finite Ramsey theorem takes the following compact form:
\begin{theorem}[Finite Ramsey Theorem, 1930~\cite{Ramsey1930}]
	$$(\forall n,k,r\in \mathbb N)(\exists N\in \mathbb N)N\longrightarrow (n)^k_r.$$
\end{theorem}
Note that the Ramsey theorem is equivalent to its special case with $r=2$: By grouping colours and iteratively using the $r=2$ version, the general version follows easily. To simplify notation, we will often use the $r=2$ variant instead of $\forall r\in \mathbb N$.

\subsection{Structural partition arrow}

Structural Ramsey theory is concerned with structural generalizations of the
partition arrow.  While its key notions can be presented in the very general
context of category theory (see \eg{}~\cite{Graham1972,Leeb}), the hard core of structural Ramsey theory lies in studying specific examples of
structures, such as graphs, ordered graphs, hypergraphs or metric spaces.
Rather than introducing the key definitions in full generality, we find it
convenient to restrict our attention to structures and embeddings:

\medskip

A \emph{language} is a collection $L$ of relation and function symbols, each having an associated {\em arity} denoted by $\arity{S}$ for a symbol $S\in L$.
In this survey, the language will often be fixed and/or understood from the context.
We adopt the standard model-theoretic notion of structures (see \eg~\cite{Hodges1993}) with one slight generalization, our functions are set-valued (see Remark~\ref{rem:set_valued} for some discussion).
Formally, denoting by $\powerset{A}$ the set of all subsets of $A$, an \emph{$L$-structure} $\str{A}$ is a structure with {\em vertex set} $A$, functions $\func{A}{}\colon A^{\arity{\func{}{}}}\to \powerset{A}$ for every function $\func{}{}\in L$ and relations $\rel{A}{}\subseteq A^{\arity{\rel{}{}}}$ for every relation $\rel{}{}\in L$.
Notice that the domain of a function consists of tuples while the range consists of sets. We will use bold letters such as $\str{A},\str{B},\str{C}$ to denote structures
and the corresponding normal letters $A$, $B$, $C$ to denote their underlying sets.

If the set $A$ is finite, we call $\str A$ a \emph{finite structure}. We consider only structures with finitely or countably infinitely many vertices.
If $L$ consists of relation symbols only, we call $L$ a {\em relational language} and say that an $L$-structure is  a {\em relational structure}.
A function $\func{}{}$ such that $\arity{\func{}{}}=1$ is a {\em unary function}.

A \emph{homomorphism} $f\colon \str{A}\to \str{B}$ is a map $f\colon A\to B$
such that  for every relation $\rel{}{}\in L$ we have $$(x_1,x_2,\ldots, x_{\arity{\rel{}{}}})\in \rel{A}{}\implies (f(x_1),f(x_2),\ldots,f(x_{\arity{\rel{}{}}}))\in \rel{B}{},$$ and for every function $\func{}{}\in L$ we have $$f[\func{A}{}(x_1,x_2,\allowbreak \ldots, x_{\arity{\func{}{}}})] = \func{B}{}(f(x_1),f(x_2),\ldots,f(x_{\arity{\func{}{}}})),$$ where for a subset $A'\subseteq A$, we denote by $f[A']$ the set $\{f(x): x\in A'\}$.

If $f$ is injective then $f$ is called a \emph{monomorphism}. A monomorphism $f$ is an \emph{embedding} if its inverse is also a monomorphism. A bijective embedding is an \emph{isomorphism}. An isomorphism from a structure to itself is called an \emph{automorphism}. If $f$ is an inclusion then $\str{A}$ is a \emph{substructure} of $\str{B}$, and we write $\str A\subseteq \str B$.

\begin{remark}\label{rem:set_valued}
	The generalization of functions to set-valued ones comes from~\cite{Evans3} and is motivated by applications (see Section~\ref{sec:orientations}): We use set-valued functions to restrict the notion of a substructure to algebraically closed substructures only (see \eg{}~\cite{Hodges1993}) by explicitly representing the algebraic closure. This would not always be possible with standard functions without removing some automorphisms.

	Notice that this generalization is ``conservative'': If we restrict to structures where the range of each function consists of singleton sets, we obtain the standard model-theoretic structures with the standard notions of maps between them (formally after identifying each singleton set with its element). It is easy to verify that the general results presented in this paper imply the same statements for the standard definition of structures.

	Moreover, in almost all examples in this paper the ranges of functions consist of singletons and the empty set. This corresponds to working with partial functions in the standard setting.
\end{remark}

\medskip

Given structures $\str{A}$ and $\str{B}$, we denote by $\Emb(\str A,\str B)$ the set of
all embeddings from $\str{A}$ to $\str{B}$.
Given structures $\str{A}$, $\str{B}$,
$\str{C}$, and an integer $r$, the \emph{structural Erd\H os--Rado partition arrow},
written as $\str{C}\longrightarrow (\str{B})^\str{A}_r$, is a shortcut for the following statement:
\begin{quote}
	For every colouring $\chi\colon\Emb(\str{A},\str{C})\to r$, there exists a
	\emph{monochromatic} embedding $f\in \Emb(\str{B},\str{C})$. That is, $\chi$ restricted
	to $\{f\circ e:e\in \Emb(\str{A},\str{B})\}$ is a constant function.
\end{quote}

This is a natural generalization of the original partition arrow:
Denoting by $\Linorder$ the class of all finite linear orders, the finite Ramsey theorem can
be equivalently stated using the structural form of the partition arrow:
$$(\forall \str{A},\str{B}\in \Linorder)(\exists \str{C}\in \Linorder)\str{C}\longrightarrow (\str{B})^\str{A}_2.$$

\begin{remark}
	The partition arrow was originally defined for copies instead of embeddings. However, it turns out that in order to obtain Ramsey theorems, one typically needs structures to be linearly ordered for both versions of the partition arrow, and since both these notions coincide for rigid structures, there is a posteriori little difference in practice. (See Observation~\ref{obs:ramsey_rigid}, Proposition~\ref{prop:ordernecessary}, and the adjacent discussions for more details.)
\end{remark}

\subsection{Induced Ramsey theorems: colouring vertices or edges}
One of the earliest structural results was obtained by Folkman~\cite{folkman1970} who showed that for every finite graph $\str{G}$ there exists a graph $\str{H}$ such
that $\str{H}\longrightarrow(\str{G})^\bullet_2$ where $\bullet$ denotes the one-vertex graph. He also showed that $\str{H}$ can be constructed such that the sizes of the largest cliques of $\str{H}$ and $\str{G}$ are the same.
This result is not hard. Komj\'ath and R\"odl~\cite{komjath1986} proved that it follows from the lexicographic product of graphs, and Nešetřil and R\"odl~\cite{Nevsetvril1976b} proved a stronger form of it as a consequence of the well-known theorem of Erd\H os and Hajnal on the
existence of hypergraphs of arbitrarily large chromatic number and girth~\cite{erdos1959graph,erdos1966chromatic}, see \eg{} Ne\v set\v ril's survey~\cite{nevsetril2013combinatorial} and his first paper~\cite{nesetril1966}.

More complicated results were concerned with colouring edges. The first edge-colouring results were independently obtained
in mid 1970s by  Deuber~\cite{Deuber1975}, Erd\H os, Hajnal, and P\'osa~\cite{erdos1975strong}, and R\"odl~\cite{rodl1973} (published in English in 1976~\cite{rodl1976generalization}). A simple proof was also given by Ne\v set\v ril and R\"odl in 1978~\cite{nevsetvril1978simple}.
This theorem became known as the \emph{Induced Ramsey theorem}, since in the case of graphs the monochromatic copy is an induced subgraph.
The corresponding \emph{induced Ramsey numbers} (and their variants) are actively studied~\cite{fox2008induced,Conlon2015}.

Answering a question of Galvin~\cite{Erdos1970}, in 1975, Nešetřil and R\"odl (using a smart representation of triangle-free graphs by $n$-tuples) showed the
following:
\begin{theorem}[Nešetřil--R\"odl, 1975~\cite{nevsetril1975ramsey}]
	\label{thm:NR3f}
	For every finite triangle-free graph $\str{G}$ there exists a finite
	triangle-free graph $\str{H}$ such that for every
	$2$-colouring of the edges of $\str{H}$ there exists a monochromatic copy of $\str{G}$.
\end{theorem}
Notice that we cannot state the Ramsey property of $\str{H}$ in Theorem~\ref{thm:NR3f}
as $\str{H}\longrightarrow(\str{G})^{\str{E}}_2$ with $\str{E}$ being the graph consisting of two vertices connected by an edge.  In fact, with our definition of the partition arrow
all coloured substructures have to be \emph{rigid} (that is, have no non-identity automorphisms):
\begin{observation}
	\label{obs:ramsey_rigid}
	Let $\str A$ and $\str B$ be structures with $\Emb(\str A,\str B)$ non-empty. If there is $\str C$ such that $\str C\longrightarrow (\str B)^\str A_2$ then $\str A$ is rigid.
\end{observation}
\begin{proof}
	Assume for a contradiction that we have such $\str C$, but $\str A$ is not rigid. This means that for every $\str A'\subseteq \str C$ isomorphic to $\str A$ there are at least two embeddings $f\neq f'\colon \str A\to \str C$ with $f[A] = f'[A] = A'$. Consequently, we can colour each of them by a different colour, contradicting that $\str C\longrightarrow (\str B)^\str A_2$.
\end{proof}
We will solve this problem by adding an order of vertices to the structures. At this point, this may seem like an overkill
since the structural partition arrow can also be defined with respect to copies rather then embeddings -- and it is often
done so in the literature. However, for more fundamental reasons, one must actually add a linear order in order to colour more complicated structures (see Proposition~\ref{prop:ordernecessary}) so this step would become necessary later anyway.

Assume that $L$ contains a binary relation $<$ and $\str A$ is an $L$-structure. We say that $\str{A}$ is \emph{ordered} if $(A,<_\str{A})$ is a linear order. For example, \emph{ordered graphs} can now be seen as ordered $L$-structures in the language $L=\{<,E\}$ where $E$ is a binary relation symbol representing edges. Using a more involved proof, Theorem~\ref{thm:NR3f} was one year later generalized to the following:
\begin{theorem}[Nešetřil--R\"odl, 1976~\cite{Nevsetvril1976}]
	\label{thm:NRedge}
	Let $\str{E}$ be the ordered graph consisting of two vertices connected
	by an edge.  Given an integer $k>2$, denote by $\mathcal K_k$ the class of
	all finite ordered graphs with no clique of size $k$. Then:
	$$(\forall k>2)(\forall \str{G}\in \mathcal K_k)(\exists \str{H}\in \mathcal K_k):\str{H}\longrightarrow (\str{G})^\str{E}_2.$$
\end{theorem}
\subsection{Ne\v set\v ril--R\"odl and Abramson--Harrington theorems}
\label{sec:ah}
The first non-trivial structural theorem for colouring arbitrary ordered substructures was given in 1977 by Ne\v set\v ril and R\"odl \cite{Nevsetvril1977} and, independently, by Abramson and Harrington~\cite{Abramson1978}.
\begin{theorem}[Unrestricted Nešetřil--R\"odl Theorem~\cite{Nevsetvril1977} or Abramson--Harrington Theorem~\cite{Abramson1978}]
	\label{thm:unNR}
	Let $L$ be a relational language containing a binary relation $<$, and let $\K$ be the class of all finite ordered $L$-structures. Then
	$$(\forall \str{A},\str{B}\in \K)(\exists \str{C}\in \K)\str{C}\longrightarrow (\str{B})^\str{A}_2.$$
\end{theorem}
The Ne\v set\v ril--R\"odl theorem, in its complete form, expands upon Theorems~\ref{thm:NR3f} and~\ref{thm:NRedge} by incorporating the ability to forbid cliques.
Before stating this theorem, we first introduce the key notion of a Ramsey class.
\begin{definition}[Ramsey Class]
	Let $\mathcal K$ be a class of finite structures and let $\str{A}$ be a structure.  We say that $\mathcal K$ is \emph{$\str{A}$-Ramsey}
	if for every $\str{B}\in \mathcal K$ there exists $\str{C}\in \mathcal K$ such that $\str{C}\longrightarrow(\str{B})^\str{A}_2$.
	We say that $\K$ is \emph{Ramsey} (or that it has the \emph{Ramsey property}) if it is $\str{A}$-Ramsey for every $\str{A}\in \mathcal K$.
\end{definition}
It the other words, $\K$ is Ramsey if and only if it satisfies the structural form of the Ramsey theorem:
$$(\forall \str{A},\str{B}\in \K)(\exists \str{C}\in \K)\str{C}\longrightarrow (\str{B})^\str{A}_2.$$

Given languages $L\subseteq L^+$ and an $L^+$-structure $\str{A}$, the \emph{$L$-reduct} of
$\str{A}$ is the $L$-structure $\str{A}\vert _L$ created from $\str{A}$ by forgetting all relations
and functions from $L^+\setminus L$. Note that model theory often works with the more general concept of first-order reducts, where an $L'$-structure $\str B$ is a \emph{first-order reduct} of an $L$-structure $\str A$
if it is the $L'$-reduct of the expansion of $\str A$ by all first-order definable relations (in a suitable language), see~\eg{}~\cite{Hodges1993}.

Generalizing the notion of a graph clique, we say that a relational $L$-structure $\str{F}$ is \emph{irreducible} if for every $u,v\in F$ there
exists a relation $R\in L$ and a tuple $\vec{x}\in \rel{F}{}$ containing both $u$ and
$v$ (or, in other words, if the Gaifman graph of $\str F$ is a clique, where the \emph{Gaifman graph} of $\str F$
is the graph with vertex set $F$ where $uv$ is an edge if and only if exists a relation $R\in L$ and a tuple
$\vec{x}\in \rel{F}{}$ containing both $u$ and $v$).
In Definition~\ref{def:irreducible} we will generalize the notion of irreducibility for languages containing functions.

Given a set $\mathcal F$ of structures, we call a structure $\str{A}$
\emph{$\mathcal F$-free} if there is no $\str{F}\in \mathcal F$ with an embedding
$\str{F}\to\str{A}$. The key strength of the Nešetřil--R\"odl theorem is its ability to construct Ramsey structures with given restrictions.
\begin{theorem}[Nešetřil--R\"odl Theorem~\cite{Nevsetvril1977,nevsetvril1977partitions,Nevsetvril1983}]
	\label{thm:NR}
	Let $L$ be a relational language containing a binary relation $<$, and let $\mathcal F$ be a set of finite $L$-structures.
	Assume that every $\str{F}\in \mathcal F$ has an irreducible $(L\setminus \{<\})$-reduct.
	Then the class of all finite ordered $\mathcal F$-free structures is a Ramsey class.
\end{theorem}
This celebrated result (which we formulate in a mild strengthening for forbidding ordered cliques) generalizes all results discussed so far. For example, all classes $\mathcal K_k$ from Theorem~\ref{thm:NRedge} are Ramsey classes.

The proof technique called the (Nešetřil--R\"odl) \emph{partite construction} (or \emph{partite method}), which was developed in a series of papers~\cite{Nevsetvril1976,Nevsetvril1977,Nevsetvril1979,Nevsetvril1981,Nevsetvril1982,Nevsetvril1983,Nevsetvril1984,Nevsetvril1984a,Nevsetvril1987,Nevsetvril1989,Nevsetvril1990,Nevsetvril2007}, is still the main tool of the structural Ramsey theory today, see for example~\cite{bhat2016ramsey,Hubicka2016,reiher2024graphs,reiher2023girth,reiher2024colouring}.
Thanks to many revisions and simplifications, the partite construction can now be presented in a relatively compact form, which we give in Appendix~\ref{appendix:partite}.

Theorem~\ref{thm:unNR} has multiple known proofs. In addition to the partite construction, these include the original proof of Abramson and Harrington~\cite{Abramson1978}, as well as various applications of dual Ramsey results (see~\cite[Theorem 12.13]{PromelBook} for the case of graphs, and see~\cite[Theorem 6.19]{Balko2023Sucessor} or~\cite[Section 7.2]{braunfeld2023big} for the general case). On the other hand, all published proofs of Theorem~\ref{thm:NR} use the partite construction.
Recently, new alternative proofs of the Ramsey property for triangle-free graphs~\cite{dobrinen2017universal,Hubicka2020CS} have been found (and generalized to finite binary languages~\cite{zucker2020,Balko2023Sucessor}).
These new proofs extend to infinite structures, but this comes at a cost of
additional complexity, and the partite construction thus remains the most effective tool for proving structural Ramsey results about finite structures to date.

\subsection{Ordering property}

Nešetřil and R\"odl also proved the following property for unordered variants of all classes $\mathcal K$ that Theorem~\ref{thm:NR} is concerned with:
\begin{definition}[Ordering Property~\cite{nevsetril1975type,Nesetvril1978}]
	\label{def:ordering}
	Let $\mathcal K$ be a class of hypergraphs. We say that $\mathcal K$ has the
	\emph{ordering property} if for every $\str{A}\in \mathcal K$  there exists
	$\str{B}\in \mathcal K$ such that for every linear ordering $(B,\leq_B)$ of vertices
	of $\str{B}$ and every linear ordering $(A,\leq_A)$ of vertices of $\str{A}$
	there exists a \emph{monotone} embedding $f\colon \str{A}\to\str{B}$:
	for every pair of vertices $u,v\in A$ satisfying $u\leq_A v$ we also have $f(u)\leq_B f(v)$.
\end{definition}
This immediately implies that structures in the Nešetřil--R\"odl theorem need to be ordered, as the order can be used
to colour embeddings. We will discuss this in greater generality in Section~\ref{sec:precompact}.
\section{1980s: Seeking additional Ramsey classes}
Ramsey classes are very special and hard to identify. Only a handful additional examples
have been known in the 1980s. An interesting example is the following:
\begin{theorem}
	\label{thm:posets}
	Let $L$ be the language consisting of two binary symbols $<$ and $\ll$, and let $\mathcal K$ be the
	class of all finite $L$-structures $\str{A}$ where $(A,\ll)$ is a partial order,
	and $(A,{<})$ is a linear extension of $(A,\ll)$.  Then $\mathcal K$ is a Ramsey class.
\end{theorem}
\begin{remark}
	This theorem has an interesting history. It was announced by Ne\v set\v ril and
	R\"odl in 1984~\cite{Nevsetvril1984} as a result following from the partite construction. One year later, using a different method,  Paoli, Trotter, and Walker proved a weaker
	form of this result~\cite[Lemma~15 and Theorem~16]{Trotter1985}\footnote{This proof is sometimes considered faulty, since the paper states as Theorem 2 an infinite form of the product Ramsey theorem, which is known to hold only in its finite form. However, when Theorem 2 is applied to prove Lemma~15 and Theorem~16, the following remark saves the day: ``To simplify the presentation of an argument we take $\underset{\sim}{Z}$ to be an infinite poset. Of course, we can actually choose $\underset{\sim}{Z}$ as $\underline{p}^k$ where $p$ is a sufficiently large integer.''}. The basic idea of
	the proof, based on a combination of the product Ramsey theorem and the dual
	Ramsey theorem, was adapted by Fouch{\'e}~\cite{Fouche1997} to determine small
	Ramsey degrees of partial orders.  This result directly implies
	Theorem~\ref{thm:posets} as shown by Soki\'c~\cite[Theorem
		7(6)]{sokic2012ramsey}.  A self-contained presentation of this strategy was given
	by Solecki and Zhao~\cite{solecki2017ramsey}, generalizing
	Theorem~\ref{thm:posets} to multiple linear extensions.
	A different approach, based on the Graham--Rothschild
	theorem alone, was found by Ma{\v{s}}ulovi{\'c} in 2018~\cite{masulovic2016pre} and later generalized to multiple partial orders and linear extensions jointly with Dragani{\'c}~\cite{draganic2019ramsey}.
	In the same year, the original proof
	using partite construction (which is also presented here in
	Section~\ref{aaplications}) was published by Ne\v set\v ril and R\"odl~\cite{nevsetvril2018ramsey}.  Generalizations to multiple partial orders and multiple linear extensions (possibly satisfying simple axioms) can be also obtained from
	Theorem~\ref{thm:posets} and a product argument (see \eg{} Bodirsky's paper~\cite[Theorem 1.5 and Proposition 1.6]{bodirsky2014new} for a structural view of the product argument, and Arman and R\"odl's paper~\cite{arman2016note} for an independent presentation).
\end{remark}

The class of all finite partial orders with linear extensions feels very close
to the class of all linear orders, so it seems that Theorem~\ref{thm:posets}
should be easier than Theorem~\ref{thm:NR}. In a certain sense this is true,
and there is a particularly direct proof using the Graham--Rothschild theorem.
On the other hand, this Ramsey class has been left somewhat under-appreciated. Only
recently it was noticed that it implies other theorems, such as the Ramsey property
of the class of all ordered metric spaces with rational distances~\cite{masulovic2016pre}, or the class
of all ordered triangle-free graphs~\cite{Hubicka2020CS}. It also played
an important role in the development of infinitary extensions of structural Ramsey theory~\cite{Hubicka2020CS,Balko2023Sucessor}.

Since the partite construction has a dual form, additional progress was made on
dual Ramsey classes, originating from the Graham--Rothschild theorem (or the
dual Ramsey theorem), see for example~\cite{promel1985induced,frankl1987}.

\subsection{Observation from Lachlan's seminar}\label{sec:lachlan}
While participating in Lachlan's seminar at the Simon Fraser University, Nešetřil discovered a crucial connection between Ramsey classes and \Fraisse{} theory. Let us review
the key definitions.

\begin{definition}[Amalgamation~\cite{Fraisse1953}]
	\label{def:amalgamation}
	Let $\str{A}$, $\str{B}_1$, and $\str{B}_2$ be structures, and let $\alpha_1\colon\str{A} \to \str B_1$ and $\alpha_2\colon \str{A}\to \str B_2$ be embeddings. A structure $\str{C}$
	with embeddings $\beta_1\colon\str{B}_1 \to \str{C}$ and
	$\beta_2\colon\str{B}_2\to\str{C}$ such that $\beta_1\circ\alpha_1 =
		\beta_2\circ\alpha_2$ is called an \emph{amalgam} of $\str{B}_1$ and $\str{B}_2$ over $\str{A}$ with respect to $\alpha_1$ and $\alpha_2$. See Figure~\ref{amalgamfig}.

	We call the amalgam $\str C$ \emph{strong} if $\beta_1[B_1] \cap \beta_2[B_2] = \beta_1\circ\alpha_1[A]$. It is \emph{free} if it is strong, $C = \beta_1[B_1] \cup \beta_2[B_2]$ and moreover whenever a tuple $\vec{x} \in C^n$ is in some relation or some function of $\str C$ then either $\vec{x} \in \beta_1[B_1]^n$ or $\vec{x}\in \beta_2[B_2]^n$ (here, abusing notation, we say that $\vec{x}$ is in some function of $\str C$ if $\vec{x} = (x_0,\ldots,x_{n-2},x_{n-1})$ and there is an $(n-1)$-ary function $\func{}{}\in L$ such that $x_{n-1}\in\func{C}{}(x_0,\ldots,x_{n-2})$).
\end{definition}
\begin{figure}
	\centering
	\includegraphics{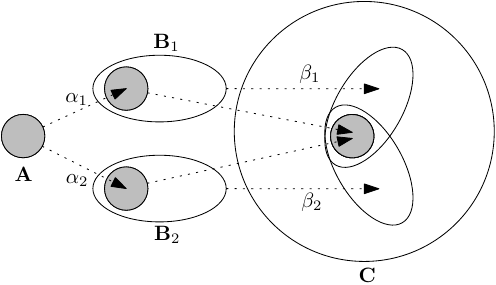}
	\caption{An amalgam of $\str{B}_1$ and $\str{B}_2$ over $\str{A}$.}
	\label{amalgamfig}
\end{figure}
We will often call $\str{C}$ simply an \emph{amalgam} of $\str{B}_1$ and $\str{B}_2$ over $\str{A}$
(as in most cases $\alpha_1$ and $\alpha_2$ can be chosen to be the inclusions).

\begin{definition}[Amalgamation Class~\cite{Fraisse1953}]
	\label{defn:amalg}
	A class $\K$ of finite structures is called an \emph{amalgamation class} if it is closed for isomorphisms and satisfies the following three conditions:
	\begin{enumerate}
		\item \emph{Hereditary property:} For every $\str{A}\in \K$ and every structure $\str{B}$ with an embedding $f\colon \str B \to \str{A}$ we have $\str{B}\in \K$;
		\item \emph{Joint embedding property:} For every $\str{A}, \str{B}\in \K$ there exists $\str{C}\in \K$ with embeddings $f\colon \str A\to \str{C}$ and $g\colon \str B\to \str C$;
		\item \emph{Amalgamation property:}
		      For $\str{A},\str{B}_1,\str{B}_2\in \K$ and embeddings $\alpha_1\colon\str{A}\to\str{B}_1$ and $\alpha_2\colon\str{A}\to\str{B}_2$, there is $\str{C}\in \K$ which is an amalgam of $\str{B}_1$ and $\str{B}_2$ over $\str{A}$ with respect to $\alpha_1$ and $\alpha_2$.
	\end{enumerate}
	If the $\str{C}$ in the amalgamation property can always be chosen as the free amalgam then $\K$ is a \emph{free amalgamation class}.
	Analogously, if $\str{C}$ can be always chosen as a strong amalgam then $\K$ is a \emph{strong amalgamation class}.
\end{definition}
Note that if the language of $\K$ contains no constants (nullary functions) then, by the hereditary property, the empty structure is a member of $\K$ and in this case the amalgamation property for $\str A$ being the empty structure is exactly the joint embedding property.

A structure $\str M$ is \emph{(ultra)homogeneous} if every isomorphism between finite substructures of $\str M$ extends to an automorphism of $\str M$. Given a structure $\str M$, its \emph{age} (denoted by $\Age(\str M)$) is the class of all finite structures which embed to $\str M$. A structure $\str M$ is \emph{locally finite} if for every finite $A\subseteq M$ there is a finite substructure $\str A'\subseteq \str M$ with $A\subseteq A'$. (Note that every relational structure is locally finite, this only becomes non-trivial for languages with functions.) The following is one of the cornerstones of model theory. (Technically, the references for the following theorem are a bit misleading as they use the standard definition of functions instead of ours. However, the results of \Fraisse{} theory remain true even in our setting, see \eg{}~\cite{Baudisch2014}.)

\begin{theorem}[\Fraisse{}~\cite{Fraisse1986}, See also~\cite{Baudisch2014}]
	\label{fraissethm}
	Let $\K$ be a class of finite structures with only countably many non-isomorphic structures.

	\begin{enumerate}[label=$(\alph*)$]
		\item\label{fraisse:a} $\K$ is the age of a countable
		      locally finite homogeneous structure $\str{H}$ if and only if $\K$ is an amalgamation
		      class.
		\item If the conditions of~\ref{fraisse:a} are satisfied then the structure $\str{H}$ is
		      unique up to isomorphism and is called the \emph{\Fraisse{} limit of $\K$}.
	\end{enumerate}
\end{theorem}
If the conditions of~\ref{fraisse:a} are satisfied, we call $\K$ a \emph{\Fraisse{} class}.
Note that \Fraisse{} theory (amalgamation classes, ages, \dots) is sometimes introduced for finitely generated structures instead of finite ones. However, we chose to only present the finite version as the Ramsey constructions which are central to this survey work with finite structures anyway.

We are now ready to show the crucial connection between Ramsey classes and homogeneous structures:
\begin{observation}[Nešetřil~\cite{Nevsetvril1989a,Nevsetril2005}]\label{obs:ramseyamalg}
	Let $\mathcal C$ be a Ramsey class of finite structures with the joint embedding property. Then $\mathcal C$ has the amalgamation property.
\end{observation}
\begin{proof}[Proof (Nešetřil--R\"odl~\cite{nevsetvril1977partitions})]
	We need to show that for every $\str A,\str B_1,\str B_2\in \mathcal C$ and embeddings $\alpha_1\colon  \str A\rightarrow \str B_1$ and $\alpha_2\colon  \str A\rightarrow \str B_2$ there is $\str C\in \mathcal C$ and embeddings $\beta_1\colon  \str B_1\rightarrow \str C$ and $\beta_2\colon  \str B_2\rightarrow \str C$ such that $\beta_1\circ \alpha_1 = \beta_2\circ \alpha_2$.

	Pick $\str B\in \mathcal C$ which embeds both $\str B_1$ and $\str B_2$, and take $\str C\in \mathcal C$ such that $\str C \longrightarrow (\str B)^\str A_2$. We will prove that $\str C$ is the amalgam we are looking for. Assume the contrary which means that there is no embedding $\alpha\colon  \str A\rightarrow \str C$ with the property that there are embeddings $\beta_1\colon  \str B_1\rightarrow \str C$ and $\beta_2\colon  \str B_2\rightarrow \str C$ such that $\beta_i\circ\alpha_i = \alpha$ for $i\in\{1,2\}$. Define a colouring $c\colon  \Emb(\str A,\str C) \rightarrow \{0, 1\}$ by letting
	$$c(\alpha) =
		\begin{cases}
			0 & \text{if there is } \beta \colon  \str B_1\rightarrow \str C\text{ such that }\alpha = \beta\circ\alpha_1 \\
			1 & \text{otherwise}.
		\end{cases}$$

	For an illustration, see Figure~\ref{fig:ramseyamalg}.
	\begin{figure}
		\centering
		\includegraphics{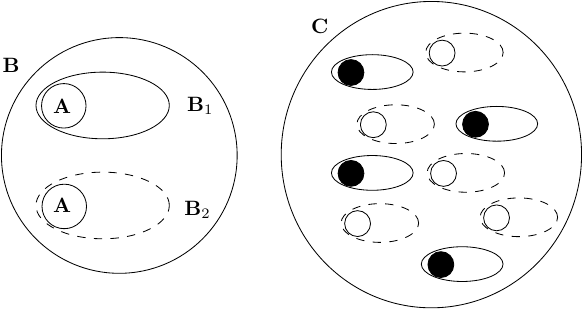}
		\caption{An illustration of the proof of Observation~\ref{obs:ramseyamalg}. Copies of $\str A$ from $\str B_1$ are coloured black, copies of $\str A$ from $\str B_2$ are coloured white.}
		\label{fig:ramseyamalg}
	\end{figure}

	As $\str C \longrightarrow (\str B)^\str A_2$, there is an embedding $\beta \colon  \str B \rightarrow \str C$ such that $c$ is constant on $\Emb(\str A,\allowbreak \beta(\str B))$. But there are at least two embeddings of $\str A$ into $\str B$ --- one is given by $\alpha_1$ and the other is given by $\alpha_2$. And $\alpha_1$ can be extended to an embedding of $\str B_1$, while $\alpha_2$ can be extended to an embedding of $\str B_2$, hence they have different colours, which is a contradiction.
\end{proof}

\begin{remark}
	The first appearance of a version of the proof of Observation~\ref{obs:ramseyamalg} which we present here is in~\cite{nevsetvril1977partitions} where Nešetřil and Rödl prove that, in modern terms, if $\mathcal C$ is a monotone class of finite relational structures closed under inverse homomorphisms and disjoint unions and the class of all linearly ordered members of $\mathcal C$ is Ramsey then $\mathcal C$ is a free amalgamation class. The intimate connection between amalgamation and \Fraisse{} theory on one side, and structural Ramsey theory on the other side, was, however, only noticed by Nešetřil at Lachlan's seminar in 1985. This then led to~\cite{Nevsetvril1989a}, where a more refined result is proved.
\end{remark}

The joint embedding property is a natural assumption. Observation~\ref{obs:ramseyamalg} thus shows us that when looking for Ramsey classes, we can restrict ourselves to ages of homogeneous structures.

\subsection{Ramsey classes from homogeneous graphs}
Ne\v set\v ril's starting point was the classification of countable homogeneous graphs:
\begin{theorem}[Lachlan--Woodrow~\cite{Lachlan1980}]
	\label{thm:LW}
	Let $\str{G}$ be a countably infinite homogeneous graph. Then $\str{G}$ or its complement is isomorphic to one of the following:
	\begin{enumerate}
		\item The Rado graph $\str{R}$: the \Fraisse{} limit of the class of all finite graphs. (Also called the countable random graph.)
		\item The \Fraisse{} limit $\str{R}_k$, for some $k>2$, of the class of all finite graph omitting the clique of size $k$.
		\item The disjoint union of $n$ cliques of size $k$ for $n,k\in \mathbb N\cup \{\omega\}$ where at least one of $n$ and $k$ is equal to $\omega$.
	\end{enumerate}
\end{theorem}
This theorem gives a complete list of the potential  graph candidates for Ramseyness.
By Observation~\ref{obs:ramsey_rigid}, none of them is a Ramsey class
due to the lack of rigidity.
However, this problem can be solved by adding
an order, just as in the statement of the Ne\v set\v ril--R\"odl theorem:
\begin{definition}
	\label{defn:freeorder}
	Let $L$ be a language not containing symbol $<$ and let $\mathcal K$ be a class
	of $L$-structures.  Let the language $L^{<}$ extend $L$ by a single binary
	relation $<$.
	Then the class of \emph{free orderings} of $\mathcal K$ is the class $\mathcal K^<$
	consisting of all ordered $L^<$-structures $\str{A}$ satisfying $\str{A}\vert _L\in \K$.
\end{definition}

For a long time, the true role of orderings in the Ramsey property was not well understood. For example, Ne\v set\v ril in 1989~\cite{Nevsetvril1989a} only considered
free orderings of the amalgamation classes given by Theorem~\ref{thm:LW}.
In this
situation, the first case (the class of all ordered graphs) is a Ramsey class
by the unrestricted Nešetřil--R\"odl theorem (Theorem~\ref{thm:unNR}), and the second case
(the class of all ordered graphs with no clique of size $k$) is a Ramsey class
by the Nešetřil--R\"odl theorem (Theorem~\ref{thm:NR}). However, in the last case one only obtains a Ramsey class when putting $k=1$ or $n=1$ (see Observation~\ref{obs:equiv_non_ramsey}), even though the Ramsey property for copies follows easily from a standard product argument (which was known to Nešetřil). Only in 2005 did the seminal paper of Kechris, Pestov, and Todor\v{c}evi\'{c}~\cite{Kechris2005} introduce the correct set-up for Ramsey expansions (see Section~\ref{sec:precompact}) which, in the case of disjoint unions of cliques, allows one to  that restrict themselves to \emph{convex} orderings (that is, orderings where every clique forms an interval). Then (possibly by further expanding by unary relations, see Example~\ref{ex:2komega}) one does indeed obtain Ramsey classes for the last case of Theorem~\ref{thm:LW} by a simple application of the product argument.

\section{1990s: Siesta time}
Ne\v set\v ril's 1989 paper~\cite{Nevsetvril1989a} showed that all hereditary
Ramsey classes of freely ordered graphs had already been known earlier and concluded
with a conjecture that this is also the case for hereditary classes of freely ordered
hypergraphs.  As a consequence of this, it seemed that the sources of potential candidates
for interesting Ramsey classes had run out, and in the 1990s the search for them has slowed
down.  However, new candidates started emerging in the context of model theory, for which the 1990s
meant a significant progress.

\subsection{Classification Programme of Homogeneous Structures}
The celebrated (Cherlin and Lachlan's) \emph{classification programme of homogeneous structures}
aims to provide full catalogues of countable
homogeneous structures of a given type. The most important cases for structural Ramsey theory where the classification
is complete are:
\begin{enumerate}
	\item\label{cat1} Schmerl's catalogue of homogeneous partial orders~\cite{Schmerl1979},
	\item\label{cat2} Lachlan and Woodrow's catalogue of homogeneous (simple) graphs~\cite{Lachlan1980},
	\item\label{cat3} Lachlan's catalogue of homogeneous tournaments~\cite{lachlan1984countable},
	\item Cherlin's catalogue of homogeneous digraphs~\cite{Cherlin1998} (this generalizes catalogues \ref{cat1}, \ref{cat2} and \ref{cat3}),
	\item Cherlin's catalogue of homogeneous ordered graphs~\cite{Cherlin2013} (the first classification result motivated by structural Ramsey theory),
	\item Braunfeld's catalogue of homogeneous finite-dimensional (generalized) permutation structures~\cite{SamPhD}, which has recently been shown to be complete by Braunfeld and Simon~\cite{sam2,braunfeld2018classification}, and
	\item Cherlin's catalogue of metrically homogeneous graphs~\cite{Cherlin2013}. A proof of its completeness has been announced by Cherlin.
\end{enumerate}
Several additional catalogues are known. An extensive list is given in  Cherlin's recent monograph~\cite{Cherlin2013} and the associated extended bibliography~\cite{cherlin2021homogeneity}. All these results give us a rich and systematic source of homogeneous structures and each such homogeneous structure potentially leads to a Ramsey class.
This, however, was fully exploited and understood only in the following decades.

\section{2000s: The revitalization}
The most important step towards revitalizing the structural Ramsey theory was undoubtedly the connection to topological dynamics
given by Kechris, Pestov, and Todor\v{c}evi\'{c} (announced in 2003). We will explain it shortly, but first we turn our attention to the main topic of our survey.

\subsection{Ne\v set\v ril's classification programme of Ramsey classes}
In 2005, Ne\v set\v ril proposed a project to
analyse known catalogues of homogeneous structures and initiated the classification programme of Ramsey
classes (which we refer to as \emph{Ne\v set\v ril's classification programme})~\cite{Nevsetril2005,Hubicka2005a}.
The overall idea is symbolized in~\cite{Nevsetril2005} as follows:
\begin{center}
	\begin{tikzpicture}[auto,
			box/.style ={rectangle, draw=black, thick, fill=white,
					text width=9em, text centered,
					minimum height=2em}]
		\tikzstyle{line} = [draw, thick, -latex',shorten >=2pt];
		\matrix [column sep=5mm,row sep=3mm] {
			\node [box] (Ramsey) {Ramsey\\ classes};
			 &  & \node [box] (amalg) {amalgamation classes};
			\\
			\\
			\node [box] (lift) {Ramsey structures};
			 &  & \node [box] (lim) {homogeneous structures};
			\\
		};
		\begin{scope}[every path/.style=line]
			\path (Ramsey)   -- (amalg);
			\path (amalg)   -- (lim);
			\path (lim)   -- (lift);
			\path (lift)   -- (Ramsey);
		\end{scope}
	\end{tikzpicture}
\end{center}
The individual arrows in the diagram can be understood as follows.
\begin{enumerate}
	\item \emph{Ramsey classes $\longrightarrow$ amalgamation classes}.  By Observation~\ref{obs:ramseyamalg}, every hereditary iso\-morphism-closed Ramsey class of finite structures with the joint embedding property is an amalgamation class.
	\item \emph{Amalgamation classes $\longrightarrow$ homogeneous structures}. By the \Fraisse{} theorem (Theorem~\ref{fraissethm}), every amalgamation class of structures with countably many mutually non-isomorphic structures has a \Fraisse{}~limit which is a homogeneous structure.

	      The additional assumption about the amalgamation class having only countably many mutually non-isomorphic structures is a relatively mild one. However, there are interesting and natural Ramsey classes not satisfying it, such as the class of all finite ordered metric spaces.
	\item \emph{Ramsey structures $\longrightarrow$ Ramsey classes}. We call a structure \emph{Ramsey} if its age is a Ramsey class and thus this connection follows by the definition.
\end{enumerate}
The difficult part of the diagram is thus the last arrow \emph{homogeneous structures $\longrightarrow$ Ramsey structures}.
We already know that only very special homogeneous structures are Ramsey, and for this reason
the precise formulation of this classification programme requires close attention.
It is a known fact that the automorphism group of every Ramsey structure fixes a linear order on
vertices. This follows at once from the connection with topological dynamics, see Proposition~\ref{prop:ordernecessary}.
This is a very strong hypothesis---such a linear order is unlikely to appear in practice unless one begins with a class of
ordered structures. So it usually happens that the
age $\K$ of a homogeneous structure $\str{H}$ is not Ramsey.
In many such cases, it can be ``expanded'' to a Ramsey class $\K^<$ as given by Definition~\ref{defn:freeorder}.
It is easy to check that $\K^<$ is an amalgamation class and
the \Fraisse{} limit of $\K^<$ thus exists and may lead to a Ramsey structure $\str{H}^<$.
This Ramsey structure can be thought of as an expansion of $\str{H}$ by an additional ``free'' or ``generic'' linear
order of vertices.

\begin{example}[Countable Random Graph]
	Consider the class $\mathcal G$ of all finite graphs and
	its~\Fraisse{} limit $\str{R}$ (cf. Theorem~\ref{thm:LW}).  Because an order is not fixed by $\Aut(\str{R})$, we can
	consider the class $\mathcal G^<$ of all linearly ordered finite graphs.
	By the Ne\v set\v ril--R\"odl theorem (Theorem \ref{thm:unNR}), $\mathcal G^<$ is Ramsey
	and thus its~\Fraisse{} limit $\str{R}^<$ (where the graph part is isomorphic to $\str R$, the order part is isomorphic to $(\mathbb Q, {<})$, and they are in ``generic'' position, \ie~they do not interact with each other non-trivially) is a Ramsey structure.
	In this sense, we completed the last arrow of Ne\v set\v ril's diagram.
\end{example}

\subsection{Precompact expansions and the expansion property}
\label{sec:precompact}
Initially, the classification problem was understood in terms
of adding linear orders when they were absent, and
then confining one’s attention to classes with the linear order
present. This point of view is still implicit in~\cite{Kechris2005}.

However, it turns out that it is necessary to consider more general expansions
of languages than can be afforded using a single linear order, such as the one in the following example:
\begin{example}[Equivalences with Two Classes]\label{ex:2komega}
	Let $\str H$ be the disjoint union of two copies of $K_\omega$ from the Lachlan--Woodrow list (Theorem~\ref{thm:LW}) and let $\mathcal C$ be its age. In other words, $\mathcal C$ is the class of all finite graphs which are the disjoint union of two cliques. We will see that $\mathcal C^<$ does not have the Ramsey property.

	Let $\str A\in \mathcal C^<$ be the (up to isomorphism unique) one-vertex structure and let $\str B\in \mathcal C^<$ be the two-vertex structure with no edge. Given an arbitrary $\str C\in \mathcal C^<$, let $C_0$ be the set of vertices of one clique of $\str C$ and let $C_1$ be the set of vertices of the other clique of $\str C$. Define a colouring $c\colon \Emb(\str A,\str C)\to 2$ by $c(f) = i$ if and only if $f[A] \subseteq C_i$ for $i\in \{0,1\}$. Clearly, every copy of $\str B$ in $\str C$ attains both colours, hence $\str C\not\longrightarrow (\str B)^\str A_2$.

	This suggests that in order to obtain a Ramsey class, one needs to distinguish vertices from the two cliques. And indeed, consider the language $L^+ = \{E,<,R\}$ where $E$ and $<$ are binary relations and $R$ is a unary relation. Let $\str H^+$ be the $L^+$-expansion of $\str H$ such that $(H,{<})$ is isomorphic to $(\mathbb Q,{<})$, and $R$ consists precisely of the vertices of one of the cliques.

	Observe that we can instead interchangeably work with the structure $\str H^\star$ in the language $L^\star = \{<,R\} \subseteq L^+$ which we obtain from $\str H^+$ by forgetting the relation $E$ -- it can be recovered using the unary relation. The Nešetřil--R\"odl theorem (Theorem~\ref{thm:NR}) tells us that $\str H^\star$, and hence also $\str H^+$, is a Ramsey structure.

	Below we will discuss the importance of finding the \textit{optimal} Ramsey expansion. Let us remark that $\str H^+$ is not the optimal expansion of $\str H$, in the optimal expansion one adds \textit{convex} orders only (that is, we demand that every vertex in $R$ is smaller than every vertex not in $R$), see Section~\ref{sec:equivalences}.
\end{example}

On the other hand, if one allows more general expansions of the structure---for
example, naming every point, that is, adding for every $x\in \str H$ the unary relation $U^x = \{x\}$---then the Ramsey property follows vacuously.
This means that a conceptual issue of identifying the
``correct'' Ramsey expansion needs to be resolved first before returning to
the technical issues of proving the existence of such an expansion, and
constructing it explicitly. The topological dynamical view of~\cite{Kechris2005} clarifies what kind of Ramsey expansion we should look for, and this can be expressed very
concretely in combinatorial terms, and hence nowadays, the conceptual issue may be considered to be satisfactorily resolved (at least provisionally).

Next, we review the resolution of this issue and justify the following:
the last arrow in Ne\v set\v ril's
diagram represents the search for an \emph{optimal expansion} of the
homogeneous structure to a larger language, in a precise sense,
and raises the question of the existence of such expansion.

Before entering into the technicalities associated with the conceptual issues,
we present another critical example in which the Ramsey property requires something
more than the addition of a linear order, and we describe the optimal solution
according to the modern point of view.

\begin{example}[Generic Local Order]
	\label{example:S2}
	An early example (given by Laflamme, Nguyen Van Th\'e, and Sauer~\cite{laflamme2010partition}) of a structure
	with a non-trivial optimal Ramsey expansion is the \emph{generic local order}.
	This is the homogeneous tournament $\str{S}(2)$ (which we view as an $L$-structure where $L$ consists of a binary relation symbol $E$) defined as follows.   Let $\mathbb
		T$ denote the unit circle in the complex plane.  Define an oriented graph
	structure on $\mathbb T$ by declaring that there is an arc from $x$ to $y$ if and only if
	$0<\mathop{\mathrm{arg}}\nolimits(y/x)< \pi$.  Call $\vv{\mathbb T}$ the resulting oriented graph.  The
	dense local order is then the substructure $\str{S}(2)$ of $\vv{\mathbb T}$ whose
	vertices are those points of $\mathbb T$ with rational argument.

	Another construction of $\str{S}(2)$ is to start from the order of the rationals
	(seen as a countable transitive tournament), randomly colour vertices with two
	colours so that both colour classes are dense and then reverse the direction of all arrows between vertices of different
	colours. In fact, this colouring is precisely the necessary Ramsey expansion.
	We thus consider the class of finite $L^+$-structures where
	$L^+$ consists of a binary relation $E$ (representing the directed edges) and a
	unary relation $R$ (representing one of the colour classes).
	The linear ordering of vertices is implicit, but can be defined based
	on the relations $E$ and $R$, putting $a< b$ if and only if either there
	is an edge from $a$ to $b$ and they belong to the same colour class, or there
	is an edge from $b$ to $a$ and they belong to different colour classes.
	The Ramsey property then follows exactly in the same way as in Example~\ref{ex:2komega}, and unlike
	in Example~\ref{ex:2komega}, this can be proved to be the optimal Ramsey expansion for $\str S(2)$.
\end{example}
The general notion of expansion (or, more particularly, \emph{relational expansion}) that we work with comes from model theory.

\begin{definition}[Expansion and Reduct]\label{defn:expansion}
	Let $L$ be a language and let $L^+\supseteq L$ extend it by relation symbols. By this we mean that $L^+\setminus L$ is a relational language, and every symbol $S\in L$ has the same arity in $L$ and $L^+$. We call $L^+$ a \emph{(relational) expansion} of $L$.

	For every $L^+$-structure $\str{A}$, there is a unique $L$-reduct $\str{A}\vert _L$ for $\str L$, that is, $A\vert _L=A$, $\nbrel{\str A\vert _L}{}=\rel{A}{}$ for every relation $\rel{}{}\in L$ and $\nbfunc{\str A\vert _L}{}=\func{A}{}$ for every function $\func{}{}\in L$. We call $\str{A}$ a \emph{(relational) expansion} of $\str{A}\vert _L$ and $\str{A}$ is called the \emph{$L$-reduct} of $\str{A}$.

	Given a language $L$, an expansion $L^+$ of $L$, and a class $\K$ of $L$-structures, a class $\K^+$ of $L^+$-structures is an \emph{expansion} of $\K$ if $\K$ is precisely the class of all $L$-reducts of structures in $\K^+$.
\end{definition}
Originally, expansions were called \emph{lifts} and reducts were called \emph{shadows} by Nešetřil and the first author, influenced by the computer science language of Gábor Kun and Nešetřil~\cite{Kun2007}. Over time, the standard model-theoretic terminology prevailed (though there are exceptions~\cite{Kun2025}).

As mentioned earlier, every $L$-structure $\str{H}$ can be turned into a Ramsey
structure $\str H^+$ by naming every point. This can be done by expanding the language $L$
by infinitely many unary relation symbols and putting every vertex of
$\str{H}^+$ into a unique relation. The Ramsey property then holds,
since there are no non-trivial embeddings between structures from the age of $\str H^+$. Clearly, additional
restrictions on the allowed expansions need to be made.

We now formulate a notion of ``optimal''
Ramsey expansion. We will first give this in purely combinatorial terms. In those terms, we seek a \emph{reasonable precompact
	Ramsey expansion with the expansion property} as defined
below (we will call it \emph{canonical} for short). In order to understand why the name canonical is suitable we will then need to invoke notions and
non-trivial results of topological dynamics.

\begin{definition}[Reasonable Expansion~\cite{Kechris2005}]
	\label{defn:reasonable}
	Let $\mathcal K^+$ be an expansion of a class of structures $\K$.
	We say that $\mathcal K^+$ is \emph{reasonable} if for
	every pair of structures $\str A, \str{B} \in \mathcal K$ with an embedding $f\colon \str A\to \str B$ and every expansion $\str A^+\in \mathcal K^+$ of $\str A$ there is an expansion $\str B^+\in \mathcal K^+$ of $\str B$ such that $f$ is an embedding $\str A^+\to \str B^+$.
\end{definition}
\begin{definition}[Precompact Expansion~\cite{The2013}]
	\label{defn:precompact}
	Let $\mathcal K^+$ be an expansion of a class of structures $\K$.
	We say that $\mathcal K^+$ is a \emph{precompact expansion} of $\mathcal K$ if for
	every structure $\str{A} \in \mathcal K$ there are only finitely many
	structures $\str{A}^+ \in \mathcal K^+$ such that $\str{A}^+$ is an expansion of
	$\str{A}$.
\end{definition}
Slightly abusing terminology, we say that an expansion $\str M^+$ of a homogeneous structure $\str M$ if \emph{precompact} if $\age(\str M^+)$ is a precompact expansion of $\age(\str M)$. (Note that $\age(\str M^+)$ is always a reasonable expansion of $\age(\str M)$.)
\begin{definition}[Expansion Property~\cite{The2013}]
	\label{defn:ordering}
	Let $\mathcal K^+$ be an expansion of $\K$. For $\str{A},\str{B}\in \K$ we say
	that $\str{B}$ has the \emph{expansion property} for $\str{A}$ if for every expansion $\str{B}^+\in \mathcal K^+$ of $\str{B}$ and for every expansion $\str{A}^+\in \mathcal K^+$ of $\str{A}$ there is an embedding $\str A^+\to\str{B}^+$.

	$\mathcal K^+$ has the \emph{expansion property} relative to $\K$ if for every $\str{A}\in \K$
	there is $\str{B}\in \K$ with the expansion property for $\str{A}$.
\end{definition}
Notice that the expansion property is a natural generalization of the ordering property (Definition~\ref{def:ordering}).
\subsection{Kechris--Pestov--Todor\v{c}evi\'{c} correspondence}

Intuitively, precompactness prevents us from considering wild expansions such as naming every point,
and the expansion property captures minimality of the expansion.
To further motivate these concepts, we review the key connections to topological
dynamics.

Given a structure $\str H$, we consider the automorphism group $\Aut(\str{H})$ as a Polish topological
group by giving it the topology of pointwise convergence.  Recall  that a
topological group $\Gamma$ is \emph{extremely amenable} if whenever $X$ is a
\emph{$\Gamma$-flow} (that is, a non-empty compact Hausdorff space on
which $\Gamma$ acts continuously), then there is a $\Gamma$-fixed point in $X$. A $\Gamma$-flow is \emph{minimal} if it admits no nontrivial closed $\Gamma$-invariant subset or, equivalently, if every orbit is dense.
Among all minimal $\Gamma$-flows, there exists a canonical one, known as the \emph{universal minimal $\Gamma$-flow}. See \cite{NVT14,zucker2016topological} for details.

In 1998, Pestov~\cite{Pestov1998free} used the Ramsey theorem to show
that the automorphism group of the order of the rationals is extremely amenable. Two years later, Glasner and
Weiss~\cite{glasner2002minimal} proved (again applying the Ramsey theorem) that the
space of all linear orderings on a countable set is the universal minimal flow
of the infinite permutation group.  In 2005, Kechris, Pestov, and Todor\v cevi\' c
introduced the general framework (which we refer to as the \emph{KPT-cor\-re\-spon\-dence}) connecting \Fraisse{} theory, Ramsey classes,
extremely amenable groups and metrizable minimal flows~\cite{Kechris2005}.
Subsequently, this framework was generalized to the notion of Ramsey expansions~\cite{The2013,NVT2009,Melleray2015,zucker2016topological} with main results as follows:

\begin{theorem}[{Kechris, Pestov, Todor\v cevi\' c \cite[Theorem 4.8]{Kechris2005}}]
	\label{thm:KPT}
	Let $\str{H}$ be locally finite\footnote{While in this survey, the age of a structure consists of finite structures, \cite{Kechris2005} consider the more general setting with finitely generated structures. For this reason, we need to add the extra assumption of local finiteness to the statement of Theorem~\ref{thm:KPT}.} homogeneous structure. Then $\Aut(\str{H})$ is extremely amenable if and only if $\Age(\str{H})$ is a Ramsey class.
\end{theorem}
Theorem~\ref{thm:KPT} is often formulated with the additional assumption that $\Age(\str{H})$ is rigid (\ie{} no structure in $\Age(\str{H})$ has non-trivial automorphisms).  This is however implied by our
definition of a Ramsey class which colours embeddings, see Observation~\ref{obs:ramsey_rigid}.

As we have mentioned earlier, it is an easy consequence of Theorem~\ref{thm:KPT} that the automorphism group of a Ramsey structure needs to fix a linear order. For finite relational languages this can actually be proved combinatorially (see \eg{}~\cite[Proposition 2.22]{Bodirsky2015}) and in a stronger form where the order will be definable.
\begin{prop}[Kechris--Pestov--Todor\v cevi\' c~\cite{Kechris2005}]\label{prop:ordernecessary}
	Let $\str M$ be a locally finite homogeneous structure whose age has the Ramsey property. Then $\Aut(\str M)$ fixes a linear order, that is, there exists a linear order $\ll$ on the vertices of $\str M$ such that for every $g\in \Aut(\str M)$ and every $x,y\in M$ it holds that $x\ll y$ if and only if $g(x)\ll g(y)$.
\end{prop}
\begin{proof}
	Let $\Linorder(M)$ be the space of all linear orders on $M$ (understood as a subspace of $2^{M^2}$). Note that $\Linorder(M)$ is Hausdorff and compact, and it is easy to see that $\Aut(\str M)$ acts continuously on it by its standard action: For $L\in \Linorder(M)$ and $g\in \Aut(\str M)$ we define $g\cdot L$ by $(x,y)\in g\cdot L$ if and only if $(g^{-1}(x), g^{-1}(y))\in L$. Therefore, as $\str M$ is Ramsey, $\Aut(\str M)$ is extremely amenable and thus this action has a fixed point, which is an order $L$ such that $g\cdot L = L$ for every $g\in\Aut(\str M)$.
\end{proof}

\medskip

Precompact Ramsey expansions with the expansion property relate to universal minimal flows as follows.
\begin{theorem}[{Melleray, Nguyen Van Th\'e, and Tsankov \cite[Corollary 1.3]{Melleray2015}}]
	\label{thm:metrizable}
	Let $\str{H}$ be a locally finite homogeneous structure. The following are equivalent:
	\begin{enumerate}
		\item The universal minimal flow of $\Aut(\str{H})$ is metrizable and has a comeagre orbit.
		\item The structure $\str{H}$ admits a precompact expansion $\str{H}^+$ whose age has the Ramsey property, and has the expansion property relative to $\Age(\str{H})$.
	\end{enumerate}
\end{theorem}
Let us note that by~\cite{zucker2016topological}, one can equivalently drop the comeagre orbit assumption in Theorem~\ref{thm:metrizable}, and by essentially the same proof as~\cite[Theorem~10.7]{Kechris2005}, one can equivalently drop the requirement that the expansion has the expansion property (see also~\cite{The2013universal}).

\begin{remark}\label{rem:unique_expansion}
	One can prove from Theorem~\ref{thm:metrizable} (or more precisely from a slightly more refined version where the universal minimal flow is constructed explicitly as the closure of the orbit of the expansion under the action of $\Aut(\str H)$) that for a given homogeneous structure $\str{H}$
	there is, up to quantifier-free $L_{\omega_1,\omega}$-interdefinability, at most one expansion $\str{H}^+$ such that
	$\Age(\str{H}^+)$ is a reasonable precompact Ramsey expansion of $\Age(\str{H})$ with the
	expansion property (relative to $\Age(\str{H})$), which justifies calling it the canonical Ramsey expansion. This was done in Section~9 of~\cite{Kechris2005} for expansions by a linear order only but the same proof also works in this more general case. The only written proof of this fact that we are aware of is by Hadek~\cite{Hadek2025} and is purely combinatorial (though it uses the language of category theory).

	(Here, an $L$-structure $\str M$ is \emph{first-order definable} in an $L'$-structure $\str M'$ if $\str M$ and $\str M'$ have the same vertex set, and for every symbol $S\in L$ there is a first-order $L'$-formula $\varphi_S$ such that $\bar{x} \in S^\str M$ if and only if $\str M'$ satisfies $\varphi_S(\bar{x})$ (one can consider $a$-ary functions as $(a+1)$-ary relations). The structures $\str M$ and $\str M'$ are \emph{first-order interdefinable} if $\str M$ is first-order definable in $\str M'$ and vice versa. $L_{\omega_1,\omega}$-(inter)definability is a generalization of first-order (inter)definability where one allows infinite disjunctions in formulas.)

	In fact, if $\str H$ is $\omega$-categorical, one actually gets uniqueness up to quantifier-free first-order interdefinability. On the other hand, infinitary formulas are necessary in the general case: Consider, for example, $\str H$ to be the \Fraisse{} limit of all finite graphs with edges labelled by countably many colours. Put $\str H^+ = (\str H, {<})$ to be the free linear ordering of $\str H$. By Theorem~\ref{thm:NR}, $\Age(\str H^+)$ is Ramsey and it is easy to verify that it has the expansion property relative to $\Age(\str H)$. However, given any set $S\subseteq \omega$ of labels, one can define $<_S$ by reversing $<$ on all edges with label from $S$. Clearly, the age of $(\str H, <_S)$ is also Ramsey and has the expansion property relative to $\Age(\str H)$, but one needs infinitary disjunctions to get $<$ from $<_S$ and vice versa.
\end{remark}

\medskip

The classification programme of Ramsey classes thus turns into two questions.
Given a locally finite homogeneous structure $\str{H}$, we ask the following:
\begin{enumerate}[label=Q\arabic*]
	\item \label{Q1} Is there a Ramsey structure $\str{H}^+$ which is a (relational) expansion of $\str{H}$
	      such that $\Age(\str{H}^+)$ is a precompact expansion of $\Age(\str{H})$? (If $\str H$ is already Ramsey then one can put $\str{H}^+=\str{H}$.)
\end{enumerate}
If the answer to \ref{Q1} is positive, we know that $\str{H}^+$ can be
chosen such that $\Age(\str{H})^+$ has the expansion property relative to
$\Age(\str{H})$.
We can moreover ask:
\begin{enumerate}[label=Q\arabic*,resume]
	\item \label{Q2}
	      If the answer to \ref{Q1} is positive, can we give an explicit description of $\str{H}^+$ which
	      additionally satisfies that the $\Age(\str{H}^+)$ has the expansion property with respect to $\Age(\str{H})$?
	      In other words, can we describe the canonical Ramsey expansion of $\str{H}$?
\end{enumerate}
\begin{remark}
	In addition to the universal minimal flow (Theorem~\ref{thm:metrizable}),
	by a counting argument given by Angel, Kechris, and Lyons~\cite{AKL14} (which goes back to Nešetřil and Rödl~\cite{Nesetvril1978} where the strong ordering property is established for classes of graphs without short cycles),
	knowledge of the answer to question~\ref{Q2} often gives amenability of
	$\Aut(\str{H})$, and under somewhat stronger assumptions also shows that
	$\Aut(\str{H})$ is uniquely ergodic. See
	\cite{sokic2015semilattices,PawliukSokic16,jahel2019unique} for an initial
	progress of the classification programme in this direction.
\end{remark}

\section{2010s: A plethora of Ramsey expansions}
The initial success in obtaining Ramsey expansions for classes from the classification programme of homogeneous structures (such as Ramsey expansions for all amalgamation classes from Cherlin's homogeneous digraph classification~\cite{Cherlin1998} given by Jasi{\'n}ski, Laflamme, Nguyen Van Th{\'e}, and Woodrow~\cite{Jasinski2013}) ignited renewed interest in the search for new Ramsey classes.
Towards this goal, multiple authors learned and applied the partite construction~\cite{The2010,Solecki2010,Solecki2012,Sokic2017,Sokic2016,boettcher2013ramsey,foniok2014ramsey,Jasinski2013,sokic2012ramsey2,junge2023categorical}, several transfer
principles became well understood (for example product Ramsey arguments, model-theoretic notions such as interpretations and free superpositions~\cite{bodirsky2014new}, categorical preadjunctions~\cite{masulovic2016pre}, semi-retractions~\cite{Scow2021,BartosovaSemiretractions}, or range-rigid functions~\cite{Mottet2021}), and
various Ramsey classes were identified using different methods~\cite{Bodirsky2010,sokic2013ramsey,sokic2015semilattices,sokic2015directed}.
The at the time new framework of Ramsey expansions in combination with a large supply of homogeneous structures proved to be useful,
making it possible to take smaller steps at a time and see Ramsey classes in a context and not only as isolated examples of hard
theorems.
Bodirsky's~\cite{Bodirsky2015} and Nguyen Van Th\'e's~\cite{NVT14} surveys describe the state of the art of that time.

Recall that a countable structure $\str M$ is \emph{$\omega$-categorical} if $\Aut(\str M)$ has only finitely many orbits on $n$-tuples for every $n\in \mathbb N$.
Motivated by rapid progress in finding new Ramsey classes
it seemed natural to ask for interesting examples of amalgamation classes with no Ramsey expansion.  This question
(in an informal variant also appearing in Ne\v set\v ril's 2011 talk at the Bertinoro meeting),
needs a careful formulation, and two variants were considered:

\begin{question}[Bodirsky--Pinsker--Tsankov~\cite{Bodirsky2013}, See also~\cite{Melleray2015}]\label{q:ramsey_precompact}
	Does the age of every $\omega$-categorical structure have a precompact Ramsey expansion?
\end{question}
\begin{question}[Bodirsky--Pinsker--Tsankov~\cite{Bodirsky2011a}]\label{q:ramsey_finite}
	Does every amalgamation class in a finite relational language have a Ramsey expansion in a finite relational language?
\end{question}
These questions were the starting points for the first
author's adventures in the area of structural Ramsey theory. During the 2013
trimester on homogeneous structures held in Bonn, he together with Ne\v
set\v ril attempted to answer these questions negatively.  Nešetřil's
intuition was that Ramsey classes are a lot more special than
homogeneous structures, and most likely a ``randomly chosen'' homogeneous structure will
be a counterexample to the question.  The lack of a counterexample was
simply attributed to the fact that the classification programme of
homogeneous structures considers only particularly simple examples of
structures.

For this reason, an ``exotic''
example was considered: the class of finite bowtie-free graphs.  This is the class
$\mathcal K$ of all finite graphs $\str{G}$ not containing a bowtie
(see Figure~\ref{fig:bowtie}) as a (not necessarily induced) subgraph.
\begin{figure}
	\centering
	\includegraphics{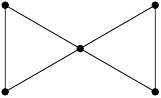}
	\caption{Bowtie graph.}
	\label{fig:bowtie}
\end{figure}
The existence of a countable countably-universal bowtie-free graph (constructed by
Komj\'ath~\cite{Komjath1999}) served as a key example in the more
general result of Cherlin, Shelah, and Shi~\cite{Cherlin1999}.  While
bowtie-free graphs seemed completely out of reach of the
Ramsey constructions at that time, primarily due to the lack of the strong
amalgamation property, three months later (when, as Ne\v set\v ril described it, ``we were both working hard as dogs'') a Ramsey expansion was found using an adjustment of the partite construction~\cite{Hubivcka2014}.
The main idea, still somewhat implicit in~\cite{Hubivcka2014}, was to add unary functions
to represent closures and turn the class into a free amalgamation class in a language with both
relations and unary functions.

Looking back, working hard as dogs was not truly necessary. An easy
on-the-top construction to add unary functions into Ramsey classes in
relational languages was later found~\cite[Theorem 4.29]{Hubicka2016}, \cite[Section
	7]{hubicka2024survey}, and using this trick, the canonical Ramsey expansion of bowtie-free
graphs can be obtained from the standard Ne\v set\v ril--R\"odl theorem.
Independently, Ramsey classes in languages containing unary functions were also
considered by Soki\'c~\cite{Sokic2016} who also approached the problem by adapting
the partite construction.

\subsection{Systematic approach: functions}
The unexpected  failures of the approach of considering ``exotic'' amalgamation classes in a
hope of finding a counterexample for the existence of precompact Ramsey
expansions led to a need to make the proof techniques used to show that a given class
is Ramsey more systematic and effective.  It was clear that writing partite
construction proofs for many amalgamation classes of interest is not only laborious, but also
dangerous, since the construction is very subtle and it is easy to
include mistakes.

The first natural question was thus whether the partite construction method can be
further generalized to structures in languages with function symbols of arity 2
or greater. Answering this was significantly harder and needed a bigger surgery
to the whole approach, discussed in Appendix~\ref{arecursive}.
However, in the end it led to the following natural generalization of irreducible structures and the Ne\v set\v ril--R\"odl theorem.

\begin{definition}[Irreducible Structure~\cite{Evans3}]
	\label{def:irreducible}
	A structure is \emph{irreducible} if it is not the free amalgam of any two of its proper
	substructures.
\end{definition}
Note that for languages without constants (nullary function), the empty structure is always a substructure; taking the free amalgam over the empty structure corresponds to finding joint embeddings.
\begin{observation}
	If $L$ is a relational language and $\str A$ an $L$-structure then $\str A$ is irreducible if and only if it is irreducible as defined in Section~\ref{sec:ah} (\ie{}, its Gaifman graph is complete).
\end{observation}
\begin{proof}
	This easily follows from the fact that the Gaifman graph of the free amalgam of relational structures is the free amalgam of their Gaifman graphs.
\end{proof}
\begin{theorem}[Hubi\v cka--Ne\v set\v ril, 2019~\cite{Hubicka2016}, See also~\cite{Evans3}]
	\label{thm:HN}
	Let $L$ be a language (possibly with both function and relation symbols) containing a binary relation $<$, and let $\mathcal F$ be a set of finite $L$-structures.
	Assume that every $\str{F}\in \mathcal F$ has an irreducible $(L\setminus \{<\})$-reduct.
	Then the class of all finite ordered $\mathcal F$-free structures is a Ramsey class.
\end{theorem}
We will prove this theorem in the appendix. Its consequence is, in particular, that if $\K$ is a free amalgamation class then $\K^<$ is Ramsey. The expansion property for free amalgamation classes was analysed in~\cite{Evans3}.
\begin{remark}\label{rem:allthousy_funkce}
	Technically, the citation~\cite{Hubicka2016} is incorrect as~\cite{Hubicka2016} works with standard non-set-valued model-theoretic functions. However, for linearly ordered locally finite structures it actually makes no difference if one allows set-valued functions: One can equivalently have countably many standard functions for every set-valued one and use the $i$-th one to address the $i$-th smallest vertex in the image. This is, however, not the case for unordered structures where set-valued functions are strictly more general (this will be relevant in Section~\ref{sec:eppa}).
\end{remark}

A related result, using a different adjustment of the partite construction, was at the same time also obtained by Bhat, Nešetřil, Reiher, and Rödl~\cite{bhat2016ramsey},
showing that the class of finite ordered partial Steiner triple system is Ramsey for colouring \emph{strong} subsystems.
This result translates naturally to the setup of free amalgamation classes in the language $L$ with a single binary function $\func{}{}$:
Recall that a \emph{partial Steiner triple
	system} is a pair $(A,\mathcal A)$ where $A$ is set of vertices and $\mathcal A\subseteq \binom{A}{3}$ satisfies $\vert E\cap E'\vert \leq 1$ for every $E,E'\in \mathcal A$ such that $E\neq E'$.
For every such partial Steiner triple system one can  construct an $L$-structure $\str{A}$ by putting $\func{A}{}(a,b)=\{c\}$ if and only if $\{a,b,c\}\in \mathcal A$, and putting $\func{A}{}(a,b)=\emptyset$ otherwise.
This yields a free amalgamation class where substructures correspond precisely to strong subsystems in the sense of Bhat, Nešetřil, Reiher, and Rödl.
This approach also easily generalizes to designs~\cite{Hubicka2017designs}.

\subsection{Systematic approach: homomorphism-embeddings and completions}\label{sec:systematic}
Amalgamation classes without the free amalgamation property present the next interesting problem. We have already seen a Ramsey expansion of finite partial orders (Theorem~\ref{thm:posets}). Prompted by the authors of~\cite{Kechris2005}, Ne\v set\v ril proved the following theorem (see also a related result of Dellamonica and R{\"o}dl~\cite{Dellamonica2012}).
\begin{theorem}[Ne\v set\v ril, 2007~\cite{Nevsetvril2007}]
	\label{thm:nesetril-metric-spaces}
	The class of all finite linearly ordered metric spaces is a Ramsey class.
\end{theorem}
The proof of this theorem was a starting point for Ne\v set\v ril and first author's generalizations towards classes described by
forbidden homomorphisms~\cite{Cherlin1999} which they analysed earlier for homogeneous expansions~\cite{Hubicka2013,hubivcka2015ramsey,Hubicka2009,Hubicka2017graham}. Before we state the main theorems, let us consider the following abstract question: Given an amalgamation class $\mathcal K$, how much additional knowledge is necessary to prove that it is a Ramsey class?

There is a nice intuition why amalgamation alone is not sufficient to build Ramsey objects, which was popularized by Ne\v set\v ril under the name of $\langle \str{A},\str{B},\str{C} \rangle$-hypergraphs:
Given structures $\str{A}$, $\str{B}$,  and $\str{C}$, the \emph{$\langle \str{A},\str{B},\str{C} \rangle$-hypergraph} is the hypergraph
with vertex set $\Emb(\str{A},\str{C})$ where $E\subseteq \Emb(\str{A},\str{C})$ forms a hyperedge if and only
if there exists $e\in \Emb(\str{B},\str{C})$ such that $E=\{e\circ f:f\in \Emb(\str{A},\str{B})\}$.
Intuitively, vertices are copies of $\str{A}$ and hyperedges are copies of $\str{B}$.  It easy to see that $\str{C}\longrightarrow (\str{B})^\str{A}_2$
if and only if the chromatic number of the \emph{$\langle \str{A},\str{B},\str{C} \rangle$-hypergraph} is at least 3.

The reason why, given $\str{A},\str{B}\in \mathcal K$, a structure $\str{C}\in \mathcal K$ satisfying $\str{C}\longrightarrow (\str{B})^\str{A}_2$ cannot be built out of $\str{A}$ and $\str{B}$
using the amalgamation property of $\mathcal K$ only is then easy to explain: Assume that there is such an abstract construction and apply it to the class $\mathcal{K}$ of finite ordered graphs for $\str A$ being a vertex and $\str B$ being an edge, always taking the free amalgam of graphs (and amalgamating the orders arbitrarily). Let $\str C$ be the resulting putative Ramsey witness, and observe that the graph reduct of $\str C$ is acyclic. Consequently, the $\langle \str{A},\str{B},\str{C} \rangle$-hypergraph, which is in this case isomorphic to the graph reduct of $\str C$, has chromatic number 2, a contradiction.

Proofs of Ramseyness thus need to contain a construction which is not amalgamation-based and produces a more complicated $\langle \str{A},\str{B},\str{C} \rangle$-hypergraph
containing cycles as depicted in Figure~\ref{fig:multiamalg}. On the other hand, tree-like $\langle \str{A},\str{B},\str{C} \rangle$-hypergraphs are easy to understand and have many nice properties, so it would be very convenient if one did have them as Ramsey witnesses.

\begin{figure}
	\centering
	\includegraphics{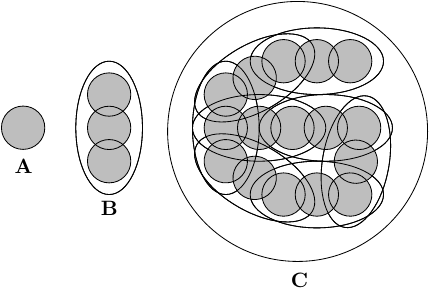}
	\caption{Construction of Ramsey objects as  $\langle \str{A},\str{B},\str{C} \rangle$-hypergraphs.}
	\label{fig:multiamalg}
\end{figure}

In~\cite{Hubicka2016}, Nešetřil and the first author address this dilemma. In order to state their results, we need two more definitions:
\begin{definition}[Tree Amalgam {\cite[Definition 7.1]{Hubicka2018EPPA}}]\label{defn:tree-amalgamation}
	Let $L$ be a language and let $\str A$ be a finite $L$-structure. We inductively define a \emph{tree amalgam of copies of $\str A$}:
	\begin{enumerate}
		\item If $\str D$ is isomorphic to $\str A$ then $\str D$ is a tree amalgam of copies of $\str A$.
		\item If $\str B_1$ and $\str B_2$ are tree amalgams of copies of $\str A$, $\str D$ is an $L$-structure, and $f_1\colon \str D\to \str B_1$ and $f_2\colon \str D\to \str B_2$ are embeddings such that $f_i[D]$ is a subset of some irreducible substructure of $\str B_i$ for $i\in \{1,2\}$, then the free amalgam of $\str B_1$ and $\str B_2$ over $\str D$ with respect to $f_1$ and $f_2$ is also a tree amalgam of copies of $\str A$.
	\end{enumerate}
\end{definition}
\begin{remark}
	Note that since we only amalgamate over substructures of irreducible substructures, it is possible to track the system of ``controlled'' copies of $\str A$ in the tree amalgam and this system is a tree. We remark that~\cite[Definition 7.1]{Hubicka2018EPPA} is slightly less restrictive as it has no requirements on $f_i[D]$. Since it is only used for structures $\str A$ which contain a binary relation forming a complete graph in~\cite{Hubicka2018EPPA}, this difference does not materialize in practice. In retrospect, it seems more practical to give a slightly more restrictive definition and avoid having additional requirements in statements using it, and so we decided to remedy it here.
\end{remark}

\begin{definition}[Homomorphism-embedding~\cite{Hubicka2016}]\label{def:homomorphism-embedding}
	Let $\str{A}$ and $\str{B}$ be structures.
	A homomorphism $f\colon\str{A}\to \str{B}$ is
	a \emph{homomorphism-embedding} if the restriction $f\vert _{\str C}$ is an embedding whenever $\str C$ is an irreducible
	substructure of $\str{A}$.
\end{definition}

While not stated in this form, the constructions in~\cite{Hubicka2016} actually prove the following (which we will prove in Appendix~\ref{sec:iterated}, see also~\cite{hubika2020structural}):
\begin{theorem}
	\label{thm:sparseningRamsey}
	Let $L$ be a language, $n$ an integer, and $\str{A}$, $\str B$, and $\str C_0$ finite $L$-structures such that $\str A$ is irreducible and $\str{C}_0\longrightarrow(\str{B})^\str{A}_2$.
	Then there exists a finite $L$-structure $\str{C}$ such that $\str{C}\longrightarrow (\str{B})^\str{A}_2$ and
	\begin{enumerate}
		\item there exists a homomorphism-embedding $\str{C}\to \str{C}_0$,
		\item\label{thm:sparseningRamsey:2} for every substructure $\str{C}'$ of $\str{C}$ with at most $n$ vertices there exists
		      a structure $\str{T}$ which is a tree amalgam of copies of $\str{B}$, and a homomorphism-embedding $\str{C}'\to\str{T}$, and
		\item\label{thm:sparseningRamsey:3} every irreducible substructure of $\str{C}$ extends to a copy of $\str{B}$. (Formally, if $\str C'\subseteq \str C$ is irreducible then there is an embedding $f\colon \str B\to \str C$ with $C'\subseteq f[B]$.)
	\end{enumerate}
\end{theorem}
Note that $\str C$ need not be linearly ordered (in fact, the second property implies that it almost never is linearly ordered). Nevertheless, $\str C_0$ will typically be linearly ordered, and thus from the existence of a homomorphism-embedding $\str C\to \str C_0$ it follows that $<_\str C$ is irreflexive, antisymmetric, and acyclic (when seen as a directed graph). Consequently, there is a linear order extending $<_\str C$.

\begin{remark}\label{rem:tree_amalgam_ramsey}
	Theorem~\ref{thm:sparseningRamsey} is, in fact, the result of interesting and fruitful interactions between the areas of structural Ramsey theory and EPPA (see Section~\ref{sec:eppa}). The first preprint of~\cite{Hubicka2016} appeared in 2016. In 2017, Aranda, Bradley-Williams, Hubi{\v c}ka, Karamanlis, Kompatscher, Konečný, and Pawliuk~\cite{Aranda2017, Aranda2017a, Aranda2017c} applied~\cite{Hubicka2016} to obtain Ramsey expansions of metric spaces from Cherlin's catalogue of metrically homogeneous graphs~\cite{Cherlin2011b,Cherlin2013}, and they also studied EPPA for these classes. It was very obvious that the arguments needed for applying~\cite{Hubicka2016} (to prove Ramseyness) and the arguments needed for applying the Herwig--Lascar theorem~\cite{herwig2000} (to prove EPPA) were very similar.

	This observation was the motivation for Nešetřil and the two authors to seek a general theorem for EPPA with assumptions similar to the main result of~\cite{Hubicka2016} (Theorem~\ref{thm:hn_completions}). The search was successful, resulting in~\cite{Hubicka2018EPPA} (see Section~\ref{sec:eppa} and Theorem~\ref{thm:hkn_completions}). Curiously, an EPPA analogue of Theorem~\ref{thm:sparseningRamsey} had to be proved first in order to carry out a crucial induction (see Theorem~\ref{thm:hkn}). After isolating this theorem, it was quite clear that the partite construction in~\cite{Hubicka2016} also proves an analogous statement. This was announced in the habilitation thesis of the first author~\cite{hubika2020structural}; Appendix~\ref{sec:iterated} is the first appearance of the full proof.
\end{remark}

For practical applications (showing that a given class is Ramsey) one typically thinks about structures produced by Theorem~\ref{thm:sparseningRamsey} as partial structures
which need to be completed back to a structure in the given class:
\begin{definition}[Completion~\cite{Hubicka2016}]
	Let $\str{C}$ be a structure. A structure $\str{C}'$ is a \emph{completion}
	of $\str{C}$ if there is a homomorphism-embedding $\str{C}\to\str{C}'$. It is a \emph{strong} completion if the homomorphism-embedding is injective. (Typically, $\str C'$ will be irreducible, hence the name completion.)
\end{definition}
Observe that if $\mathcal K$ is a class of irreducible structures, the amalgamation property for $\mathcal K$
can be equivalently formulated as follows: For $\str{A}$, $\str{B}_1$,
$\str{B}_2 \in \mathcal K$, and embeddings $\alpha_1\colon\str{A}\to\str{B}_1$ and
$\alpha_2\colon\str{A}\to\str{B}_2$, there is $\str{C}\in \mathcal K$
which is a  completion of the free amalgam of $\str{B}_1$ and $\str{B}_2$ over $\str{A}$ with respect
to $\alpha_1$ and $\alpha_2$ (which itself need not be in $\mathcal K$). Similarly, strong amalgamation
is just a strong completion of the free amalgamation.

In the following we define a locally finite subclass. Note that this is not related to locally finite substructures defined in Section~\ref{sec:lachlan}.
\begin{definition}[Locally Finite Subclass~\cite{Hubicka2016}]\label{defn:locfin}
	Let $L$ be a language, let $\mathcal R$ be a class of finite $L$-structures, and let $\mathcal K$ be a subclass of $\mathcal R$. We say
	that $\mathcal K$ is a \emph{locally finite subclass of $\mathcal R$} if for every $\str B\in \mathcal K$ and every $\str{C}_0 \in \mathcal R$ there is an integer $n = n(\str B, \str {C}_0)$ such that
	every finite $L$-structure $\str C$ has a completion $\str C'\in \mathcal K$, provided that it satisfies the following:
	\begin{enumerate}
		\item\label{locallyfinite:1} Every irreducible substructure of $\str C$ extends to a copy of $\str B$ in $\str C$ (formally, if $\str D\subseteq \str C$ is irreducible then there is an embedding $f\colon \str B\to \str C$ with $D\subseteq f[B]$),
		\item\label{locallyfinite:2} there is a homomorphism-embedding from $\str{C}$ to $\str{C}_0$, and
		\item\label{locallyfinite:3} every substructure of $\str{C}$ on at most $n$ vertices has a comple\-tion in $\mathcal K$.
	\end{enumerate}
	A locally finite subclass is \emph{strong} if the completion $\str C'$ can always be chosen to be strong.
\end{definition}

The following main result of~\cite{Hubicka2016} (cf. Remark~\ref{rem:allthousy_funkce}) is a direct corollary of Theorem~\ref{thm:sparseningRamsey}.
\begin{theorem}[Hubička--Nešetřil, 2019~\cite{Hubicka2016}]
	\label{thm:hn_completions}
	Let $L$ be a language, let $\mathcal R$ be a Ramsey class of finite $L$-structures and let $\mathcal K$ be a hereditary locally finite subclass of $\mathcal R$ with the amalgamation property consisting of irreducible structures. Then $\mathcal K$ is Ramsey.
\end{theorem}

\begin{remark}
	Note that Definition~\ref{defn:locfin} differs slightly from the original version from~\cite{Hubicka2016} (as well as from the version from~\cite{Hubicka2018EPPA}). In particular:
	\begin{itemize}
		\item Here we do not demand that structures in $\mathcal R$ or $\mathcal K$ are irreducible (instead, we add this requirement to the statement of the corresponding theorems).
		\item The original version asks that every irreducible substructure of $\str C$ is in $\mathcal K$, while here we say that every irreducible substructure of $\str C$ extends to a copy of $\str B$ in $\str C$. If $\mathcal K$ is hereditary (which is the case in the main theorems) then every locally finite subclass in the sense of~\cite{Hubicka2016} is also locally finite in the sense of this survey.
	\end{itemize}
	Similarly, the version of Theorem~\ref{thm:hn_completions} from~\cite{Hubicka2016} requires $\mathcal R$ to consist of irreducible structures, while here we only demand that $\mathcal K$ does. In light of Proposition~\ref{prop:ordernecessary}, this difference is purely cosmetic, but in the proof we give here we only need irreducibility of members of $\mathcal K$.
\end{remark}

In the proof of Theorem~\ref{thm:hn_completions}, we will first use Theorem~\ref{thm:sparseningRamsey} to produce a Ramsey witness and then we will use local finiteness to find its completion in $\mathcal K$. The following observation, which was first proved in~\cite{Hubicka2018EPPA}, tells us that small substructures of the structure $\str C$ produced by Theorem~\ref{thm:sparseningRamsey} have completions in $\K$.
\begin{observation}[Observation~9.4 of~\cite{Hubicka2018EPPA}]\label{obs:tree_amalg_completion}
	Let $\mathcal K$ a class of finite structures with the amalgamation property. Given $\str A\in \mathcal K$ and $\str D$ which is a tree amalgam of copies of $\str A$, if every irreducible substructure of $\str A$ is in $\mathcal K$ then there is $\str D'\in \mathcal K$ and a homomorphism-embedding $\str D\to \str D'$.
\end{observation}
\begin{proof}
	We will proceed by induction on the recursive definition of $\str D$. If $\str D$ is isomorphic to $\str A$, then the statement clearly holds. Otherwise there is a structure $\str E$ such that $\str D$ is the free amalgam of its proper substructures $\str D_1$ and $\str D_2$ over embeddings $\alpha_1\colon \str E\to \str D_1$ and $\alpha_2\colon \str E\to \str D_2$, where $\str D_1$ and $\str D_2$ are tree amalgams of copies of $\str A$, and $\alpha_i[E]$ is a subset of an irreducible substructure of $\str D_i$ for $i\in \{1,2\}$. By the induction hypothesis there are $\str D_1',\str D_2'\in \mathcal K$ and homomorphism-embeddings $f_1\colon \str D_1\to \str D_1'$ and $f_2\colon \str D_2\to \str D_2'$. Since $\alpha_i[E]$ is a subset of an irreducible substructure of $\str D_i$ for $i\in \{1,2\}$ and $f_1$ and $f_2$ are homomorphism-embeddings, it follows that they are embeddings when restricted to $\alpha_1[E]$ and $\alpha_2[E]$ respectively. Consequently, we can put $\str D'$ to be an amalgam of $\str D_1'$ and $\str D_2'$ over the embeddings $f_1\circ \alpha_1$ and $f_2\circ\alpha_2$.
\end{proof}

\begin{proof}[Proof of Theorem~\ref{thm:hn_completions}]
	Fix $\str A,\str B\in \mathcal K$.
	Use the Ramsey property of $\mathcal R$ to obtain $\str C_0\in \mathcal R$ such that $\str C_0\longrightarrow(\str B)^\str A_2$. Put $n=n(\str B,\str C_0)$ (from Definition~\ref{defn:locfin}), and obtain $\str C$ using Theorem~\ref{thm:sparseningRamsey}. In particular, every irreducible substructure of $\str C$ extends to a copy of $\str B$ in $\str C$ and there is a homomorphism-embedding $\str C\to \str C_0$.

	Let $\str D$ be a substructure of $\str C$ on at most $n$ vertices. By Theorem~\ref{thm:sparseningRamsey} it has a homomorphism-embedding to a tree amalgam of copies of $\str B$, and since $\mathcal K$ is hereditary, has the amalgamation property, and consists of irreducible structures, Observation~\ref{obs:tree_amalg_completion} gives us a completion of $\str D$ in $\mathcal K$. Consequently, as $\mathcal K$ is a locally finite subclass of $\mathcal R$, we get $\str C'\in \mathcal K$ and a homomorphism-embedding $g\colon \str C\to \str C'$. Note that whenever $f$ is an embedding $\str A\to \str C$ or $\str B\to \str C$ then $g\circ f$ is an embedding $\str A\to \str C'$ resp. $\str B\to \str C'$, since both $\str A$ and $\str B$ are irreducible. Consequently, a colouring of $\Emb(\str A,\str C')$ gives us a colouring of $\Emb(\str A,\str C)$, and as $\str C\longrightarrow(\str B)^\str A_2$, we obtain a monochromatic embedding $f\colon \str B\to \str C$. The embedding $g\circ f$ then witnesses that, indeed, $\str C'\longrightarrow(\str B)^\str A_2$.
\end{proof}

In applications it would often be much more convenient if Definition~\ref{defn:locfin} promised that every substructure on at most $n$ vertices has a strong completion in $\mathcal K$, not only a completion. Luckily, in strong amalgamation classes in languages where all functions are unary one can obtain injective homomorphism-embeddings from non-injective ones, and thus assume the existence of strong completions in Definition~\ref{defn:locfin}:
\begin{fact}[Proposition~11.3 of~\cite{Hubicka2018EPPA}, see also Proposition~2.6 of~\cite{Hubicka2016}]\label{fact:strong_homemb}
	Let $L$ be a language where all functions are unary, let $\mathcal C$ be a strong amalgamation class of $L$-structures, let $\str A$ be a finite $L$-structure and let $f\colon \str A\to \str B$ be a homomorphism-embedding for some $\str B\in \mathcal C$. Then there is $\str B'\in \mathcal C$ and an injective homomorphism-embedding $f'\colon \str A\to \str B'$.
\end{fact}
The proof with unary functions is a bit technical and we refer the reader to~\cite{Hubicka2018EPPA}. Here we only give the conceptually similar proof for relational languages.
\begin{proof}[Proof of Fact~\ref{fact:strong_homemb} for Relational $L$]
	Assume for a contradiction that there is an $L$-structure $\str A$ with a homo\-morphism-embedding $f\colon \str A\to \str B$ for some $\str B\in \mathcal C$ such that there is no $\str B'\in \mathcal C$ and an injective homomorphism-embedding $f'\colon \str A\to \str B'$. Among all such counterexamples pick one with minimal $\lvert \str A\rvert$.

	It follows (here we are using that $L$ is relational) that there are vertices $x,y\in A$ such that $f(x) = f(y)$, and $f$ is injective on $A\setminus\{y\}$. Note that $f(x) = f(y)$ implies that $x$ and $y$ are in no relations of $\str A$ together. Put $\str A' = f(\str A\setminus \{x,y\})$, let $\str B'\in \mathcal C$ be a strong amalgam of two copies of $\str B$ over $\str A'$, and let $\beta_1\colon \str B\to \str B'$ and $\beta_2\colon \str B\to \str B'$ be the corresponding embeddings. Define a function $f'\colon A\to B'$ such that $f'(v) = \beta_1(f(v))$ if $v\neq y$ and $f'(y) = \beta_2(f(y))$. It is easy to verify that $f'\colon \str A\to \str B'$ is an injective homomorphism-embedding.
\end{proof}

\begin{corollary}
	Let $\mathcal K$ be a locally finite subclass of $\mathcal R$. If $\mathcal K$ is a strong amalgamation class and all functions in the language of $\mathcal K$ are unary then $\mathcal K$ is a strong locally finite subclass of $\mathcal R$.
\end{corollary}

\medskip

Theorem~\ref{thm:hn_completions} could be used to show Ramseyness of all ``primal structural'' (see Section~\ref{sec:category}) Ramsey classes known at that time.
This motivated the title ``All the Ramsey classes'' for Ne\v set\v ril's 2016 talk at Tel Aviv University\footnote{\url{https://www.youtube.com/watch?v=_pfa5bogr8g}}. However, after some consideration, it was turned into the less ambitious and more colloquial form of ``All those Ramsey classes''.
This name still holds: the vast majority of Ramsey classes known today follow from Theorem~\ref{thm:hn_completions}. Nevertheless, there are exceptions, in particular those based on dual Ramsey theorems or the Milliken tree theorem discussed in Section~\ref{exceptions}.

Next, we review some typical applications of Theorem~\ref{thm:hn_completions}.

\subsubsection{Metric spaces}
\label{sec:metric}
We start with a proof of Theorem~\ref{thm:nesetril-metric-spaces}. Given $S\subseteq \mathbb R^{{>}0}$, an \emph{$S$-edge-labelled graph} is a graph with edges labelled by numbers from $S$ (or equivalently, a relational structure with $\lvert S\rvert$ binary relations such that all of them are symmetric and irreflexive and every pair of vertices is in at most one relation). An $\mathbb R^{{>}0}$-edge-labelled triangle is \emph{non-metric} if it has labels $a,b,c$ and $a > b+c$. Complete $S$-edge-labelled graphs which embed no non-metric triangles correspond to metric spaces with distance set $S$.

A \emph{non-metric cycle} is an $\mathbb R^{{>}0}$-edge-labelled cycle with labels $a_0,\ldots,a_{k-1}$ such that $a_0 > \sum_{i\geq 1} a_i$. The following proposition will be our key ingredient.
\begin{prop}\label{prop:metric_completion}
	For every $S\subseteq \mathbb R^{{>}0}$ and every $S$-edge-labelled graph $\str G$, if at least one of $S$ and $\str G$ are finite then the following are equivalent:
	\begin{enumerate}
		\item\label{prop:metric_completion:1} $\str G$ has a completion to a metric space,
		\item\label{prop:metric_completion:2} $\str G$ has a strong completion to a metric space,
		\item\label{prop:metric_completion:3} there is no non-metric cycle $\str C$ with a homomorphism-embedding $\str C\to \str G$,
		\item\label{prop:metric_completion:4} there is no non-metric cycle $\str C$ with an embedding $\str C\to \str G$.
	\end{enumerate}
\end{prop}
\begin{proof}
	Clearly, (\ref{prop:metric_completion:2}) implies (\ref{prop:metric_completion:1}) and (\ref{prop:metric_completion:3}) implies (\ref{prop:metric_completion:4}).

	It is easy to see that non-metric cycles do not have completions to metric spaces, and as being a completion transfers to substructures, we see that (\ref{prop:metric_completion:1}) implies (\ref{prop:metric_completion:3}).

	To see that (\ref{prop:metric_completion:4}) implies (\ref{prop:metric_completion:3}), assume that there is a non-metric cycle $\str C$ with a ho\-mo\-mor\-phism-embedding $f\colon \str C\to \str G$, and among all such cycles, pick one with the least number of vertices. Note that if there are $x,y\in C$ with $f(x)=f(y)$ then one can find a shorter non-metric cycle in $f(\str C)$ (by taking one containing the longest edge of $f(\str C)$), hence $f$ is injective. Similarly, if there are $x,y\in C$ such that $xy$ is not an edge of $\str C$, but it is an edge of $f(\str C)$ then this edge splits $f(\str C)$ into two cycles of which at least one is non-metric and has fewer vertices, hence $f$ is in fact an embedding.

	It remains to prove that (\ref{prop:metric_completion:3}) implies (\ref{prop:metric_completion:2}), let $\str G$ be an $S$-edge-labelled graph such that no non-metric cycle has a homomorphism-embedding to $\str G$, and let $D$ be the maximum distance appearing in $\str G$ (such $D$ exists as either $S$ or $\str G$ are finite).
	Define a function $d\colon G^2\to \mathbb R^{{\geq} 0}$ as follows:
	$$d(x,y) = \inf\left(\{D\} \cup \{\|\str P\| : \str P\text{ is a path between $x$ and $y$ in $\str G$}\}\right),$$
	where $\|\str P\|$ is the sum of all labels of the path. (Note that $d(x,x)=0$ for every $x\in G$.) It is easy to see that $(G,d)$ is a metric space: It is well-defined, symmetry follows from the fact that paths are not directed, and $d(x,y) \leq d(x,z)+d(z,y)$, as both $d(x,z)$ and $d(z,y)$ are witnessed by some paths, and concatenating them leads to a path from $x$ to $y$ of length at most $d(x,z)+d(z,y)$. Moreover, since $\str G$ or $S$ are finite, the infimum is actually a minimum, and hence non-zero for $x\neq y$.

	Clearly, one can define an $\mathbb R^{{>}0}$-edge-labelled graph $\str G'$ from $(G,d)$ by defining the relations according to $d$. It follows that the identity $\str G\to\str G'$ is a (strong) completion if and only if for every edge $xy$ of $\str G$ with label $s$ we have that $d(x,y) = s$.

	As the edge $xy$ is a path between $x$ and $y$, it follows that $d(x,y)\leq s$. But if $d(x,y) < s$, it means that there is a path $\str P$ in $\str G$ between $x$ and $y$ of length $d(x,y)$. Adding the edge $xy$ to $\str P$, we obtain a homomorphism-embedding of a non-metric cycle to $\str G$. As we assumed that there is no such homomorphism-embedding, it follows that $\str G'$ is a strong completion of $\str G$.
\end{proof}

\begin{remark}
	The construction of a completion of $G$ in Proposition~\ref{prop:metric_completion} appears in various contexts. In Section~\ref{sec:eppa} it will be useful that every automorphism of $G$ is also an automorphism of its completion. Another nice fact about it is that it is \emph{canonical}; having canonical completions of free amalgams has turned out to be an important ingredient for understanding normal subgroups of automorphism groups of homogeneous structures, see \eg{}~\cite{Higman54,Truss1985,Droste1989,Glass1993,Droste1999, Droste2000,Macpherson2011b,Tent2013b,Tent2013,Evans2016,Li2018,Li2019,Li2020,LiPhd,calderoni2020simplicity,Evanssimplicity,Kaplan2025,Tent2025,Li2025}.
\end{remark}

\begin{lemma}\label{lem:metric_unordered_locally_finite}
	The class $\mathcal K$ of all finite metric spaces is a strong locally finite
	subclass of the class $\mathcal R$ of all finite $\mathbb R^{{>}0}$-edge-labelled graphs.
\end{lemma}
\begin{proof}
	Given $\str B \in \mathcal K$, let $S$ be the set of distances in $\str B$, and put $n=\left\lceil \frac{\max(S)}{\min(S)}\right\rceil$. Now, for an arbitrary $\str C_0\in \mathcal R$ (we will not need it), let $\str C$ be as in Definition~\ref{defn:locfin}.

	It follows by condition~(\ref{locallyfinite:1}) of Definition~\ref{defn:locfin} that $\str C$ is a finite $S$-edge-labelled graph. Since the largest non-metric cycle using only distances from $S$ has at most $n$ vertices, it follows by Proposition~\ref{prop:metric_completion} that $\str C$ indeed has a strong completion to a metric space, and hence $\mathcal K$ is indeed a strong locally finite subclass of $\mathcal R$.
\end{proof}

With enough optimism one might think that Lemma~\ref{lem:metric_unordered_locally_finite} directly allows us to apply Theorem~\ref{thm:hn_completions}. However, $\mathcal R$ is not Ramsey, which is one of the conditions that one needs to satisfy. In view of Proposition~\ref{prop:ordernecessary}, one could try to expand by linear orders and this turns out to actually work in general under some minor assumptions:
\begin{lemma}\label{lem:locally_finite_add_order}
	Let $L$ be a language not containing the symbol $<$, let $\mathcal R$ be a class of finite $L$-structures and let $\mathcal K$ be a strong locally finite subclass of $\mathcal R$ consisting of irreducible structures. Let $\mathcal R^<$ and $\mathcal K^<$ be classes of free orderings of $\mathcal R$ and $\mathcal K$ respectively (see Definition~\ref{defn:freeorder}). Then $\mathcal K^<$ is a strong locally finite subclass of $\mathcal R^<$.
\end{lemma}
\begin{proof}
	Given $\str B\in \mathcal K^<$ and $\str C_0\in \mathcal R^<$, put $\str B^- = \str B\vert _L$ and $\str C_0^- = \str C_0\vert _L$. Note that $\str B^-\in \mathcal K$ and $\str C_0^-\in \mathcal R$. Since $\mathcal K$ is a strong locally finite subclass of $\mathcal R$, we get $n = n(\str B^-,\str C_0^-)$ witnessing this. Let $\str C$ be an $L^<$-structure with the following properties:
	\begin{enumerate}
		\item Every irreducible substructure of $\str C$ extends to a copy of $\str B$ in $\str C$,
		\item there is a homomorphism-embedding $f\colon \str{C}\to\str{C}_0$, and
		\item every substructure of $\str{C}$ on at most $n$ vertices has a completion in $\mathcal K^<$.
	\end{enumerate}
	Put $\str C^- = \str C\vert _L$. If $\str C^-$ induces an irreducible structure on $D\subseteq C$, then so does $\str C$. Consequently, we have the following:
	\begin{enumerate}
		\item Every irreducible substructure of $\str C^-$ extends to a copy of $\str B^-$ in $\str C^-$,
		\item $f$ a homomorphism-embedding $\str{C}^-\to\str{C}_0^-$, and
		\item every substructure of $\str{C}^-$ on at most $n$ vertices has a completion in $\mathcal K$.
	\end{enumerate}
	Consequently, as $\mathcal K$ is a strong locally finite subclass of $\mathcal R$, we get $\str C'\in \mathcal K$ and an injective homomorphism-embedding $g\colon \str C^-\to\str C'$.

	As every irreducible substructure of $\str C$ has an embedding to $\str B$, we know that the relation $<_\str C$ is irreflexive and antisymmetric. However, $\str C$ theoretically could contain arbitrarily long sequences of vertices $c_0,c_1,\ldots,c_{m-1}$ such that $c_i <_\str C c_{i+1}$ for every $0\leq i < m-1$, and $c_{m-1} <_\str C c_0$, with no other $<_\str C$ relations between these vertices (in other words, oriented cycles in the binary relation $<$). Clearly, such structures have no completion to a linear order.

	However, we also know that there is a homomorphism-embedding $f\colon\str C\to \str C_0$. Note that $\str C_0$ is linearly ordered, and hence contains no oriented cycles in $<_{\str C_0}$. Since the image of an oriented cycle contains an oriented cycle, we find that $<_{\str C}$ is an oriented acyclic graph. Since $g$ is injective, the image $g[<_\str C]$ is an oriented acyclic graph on $\str C'$, and thus there is a linear order $<_{\str C'}$ on $\str C'$ extending $g[<_\str C]$. Clearly, $g$ is an injective homomorphism from $\str C$ to $(\str C',<_{\str C'})$.

	It remains to prove that if $\str D\subseteq \str C$ is irreducible then $g$ restricted to $\str D$ is embedding. Note, however, that in such a case there is $\str B'\subseteq \str C$ isomorphic to $\str B$ such that $\str D\subseteq \str B'$, and so we know that $\str C^-$ induces an irreducible substructure on $B'$. Consequently, $g\vert _{B'}$ is an embedding, and so in particular $g\vert _D$ is also an embedding.
\end{proof}

Theorem~\ref{thm:nesetril-metric-spaces} now follows easily:
\begin{proof}[Proof of Theorem~\ref{thm:nesetril-metric-spaces}]
	Let $\mathcal K^<$ be the class of free orderings of $\mathcal K$, the class of all finite metric spaces. Lemma~\ref{lem:metric_unordered_locally_finite} tells us that $\mathcal K$ is a strong locally finite subclass of $\mathcal R$, the class of all finite $\mathbb R^{{>}0}$-edge-labelled graphs. Since $\mathcal K$ consists of irreducible structures, Lemma~\ref{lem:locally_finite_add_order} implies that $\mathcal K^<$ is a strong locally finite subclass of the class $\mathcal R^<$ of free orderings of $\mathcal R$, which is Ramsey by Theorem~\ref{thm:NR}. It is easy to check that $\mathcal K^<$ is hereditary and has the (strong) amalgamation property (free amalgams of metric spaces contain no non-metric cycles), and hence Theorem~\ref{thm:hn_completions} applies.
\end{proof}

This shows us that, as the linear order is independent from the rest of the structure, we can the linear order separately when discussing the existence of a completion. Note however that we needed $\mathcal K$ to consist of irreducible structures in order to observe that the order in $\str C$ is defined on the exactly same pairs as the distances. This is a small but important detail, as otherwise it would be impossible to separate the treatment of the order from the metric. This is witnessed in the following example:

\begin{example}
	Let $L$ be the language consisting of binary relations $<$ and $E$, and let $\mathcal C$ be the class of all finite $L$-structures in which $<$ is a linear order and $E$ is an equivalence relation. It is easy to see that $\mathcal C$ is a strong amalgamation class consisting of irreducible structures. However, when we forget the order relation, many structures become reducible. This has serious consequences: If $\str C$ is a graph then it always has a completion to an equivalence: Just consider the complete graph on the same vertex set. On the other hand, if $\str C$ is an $L$-structure such that $<$ is irreflexive and antisymmetric and $E$ is reflexive and symmetric then there might be some pairs of vertices which are in the $<$ relation but not in the $E$ relation, hence we can no longer simply complete $\str C$ to a complete graph.
\end{example}

In the following section we will see an example where the order plays a more central role.
\subsubsection{Equivalences}\label{sec:equivalences}
Let $\mathcal K$ consist of all finite $\{E,N\}$-edge-labelled complete graphs where $E$ is transitive. (In other words, $E$ together with the diagonal is an equivalence relation and $N$ consists of the non-equivalent pairs.) Clearly, $\mathcal K$ is a strong amalgamation class which consists of irreducible structures. Let $\str C$ be an $\{E,N\}$-edge-labelled graph. We want to characterize when we can complete $\str C$ to $\mathcal K$.

Let $\ell$ be the labelling of edges of $\str C$. The natural attempt is to emulate the proof of Proposition~\ref{prop:metric_completion} and define a new labelling $\ell'$ on pairs of vertices of $\str C$ such that $\ell'(u,v) = E$ if and only if there exists a path in $\str C$ between $u$ and $v$ such that all edges of the path are labelled by $E$, and $\ell'(u,v) = N$ otherwise. Clearly, whenever $\ell'(u,v)=E$ then it needs to be so in any completion of $\str C$ due to transitivity.

Define $\str C'$ to be the $\{E,N\}$-edge-labelled complete graph on vertex set $C$ where the label of $uv$ is $\ell'(u,v)$. Observe that $\str C'\in \mathcal K$ and that if $\ell\subseteq \ell'$ then the identity is a homomorphism-embedding $\str C\to \str C'$ (and so $\str C'$ is a strong completion of $\str C$). Assume that $\ell'$ does not extend $\ell$, and let $u,v\in C$ be such that $\ell(u,v)$ is defined and different from $\ell'(u,v)$. If $\ell(u,v)=E$ then $\ell'(u,v)=E$ which is witnessed by the path $uv$. So $\ell(u,v)=N$ and $\ell'(u,v)=E$. But this means that there is a path from $u$ to $v$ with all edges labelled by $E$. Consequently, $\str C$ does not have any completion into $\mathcal K$, as cycles with one edge labelled by $N$ and all other edges labelled by $E$ do not have such a completion. Unfortunately, the shortest witnessing path for $u$ and $v$ being equivalent can be arbitrarily long, hence there is no bound on the size of the largest obstacle.

In the proof of Lemma~\ref{lem:locally_finite_add_order} we have seen an analogous situation happen with the linear order and the solution there was to invoke condition~(\ref{locallyfinite:2}) of Definition~\ref{defn:locfin}. This is, however, not applicable in this case. Instead, we need to side-step the issue:

Let $L'$ be the language consisting of one unary function $\func{}{}$ and two unary relations $O$ and $I$. Let $\mathcal D$ be the class of all finite $L'$-structures $\str A$ such that $O_\str A$ and $I_\str A$ form a partition of the vertex set $A$, $\func{A}{}(x) = \emptyset$ if and only if $x\in I_\str A$, and for every $x\in O_\str A$ it holds that $\func{A}{}(x) = \{y\}$ for some $y\in I_\str A$.
We will call the vertices from $O_\str A$ the \emph{original} vertices of $\str A$ and the vertices from $I_\str A$ the \emph{imaginary} vertices of $\str A$. Notice that there is a simple correspondence between structures from $\mathcal D$ and structures from $\mathcal K$: Given $\str A\in \mathcal D$, define an $L$-structure $T(\str A)\in \mathcal K$ with vertex set $O_\str A$, putting $uv\in E_{T(\str A)}$ if and only if $\func{A}{}(u)=\func{A}{}(v)$ and $uv\in N_{T(\str A)}$ otherwise. On the other hand, given $\str A\in \mathcal K$, let $I$ be the set of all equivalence classes of $\str A$ and define an $L'$-structure $U(\str A)\in \mathcal D$ with vertex set $A\cup I$ such that all vertices from $A$ are in the relation $O_{U(\str A)}$, all vertices from $I$ are in the relation $I_{U(\str A)}$ and for every $u\in A$, $\nbfunc{U(\str A)}{}(u) = \{c\}$, where $c\in I$ is the equivalence class whose $u$ is a member in $\str A$.

Note that $\str A = T(U(\str A))$ for every $\str A\in \mathcal K$ and that $T$ and $U$ are functors (\ie{}, they preserve embeddings and their compositions). Note also that the membership conditions of $\mathcal D$ ($O$ and $I$ form a partition of the vertex set, the domain of $\func{}{}$ is $O$ and the image of $\func{}{}$ is a subset of $I$) are conditions on irreducible substructures, and so the class $\mathcal D^<_0$ of free orderings of members of $\mathcal D$ is Ramsey by Theorem~\ref{thm:HN}. However, the map $T$ does not lift to linearly ordered structures as there is no canonical way of defining an ordering on $I$ which would be stable under taking substructures. In fact, it turns out that adding all linear orders does not lead to a Ramsey class:

\begin{observation}\label{obs:equiv_non_ramsey}
	The class $\mathcal K^{<}_0$ of free orderings of members of $\mathcal K$ is not Ramsey.
\end{observation}
\begin{proof}
	Let $\str A\in \mathcal K^{<}_0$ be the structure with $A = \{a_0,a_1\}$ where $a_0 < a_1$ and the edge $a_0a_1$ has label $N$, and let $\str B\in \mathcal K^{<}_0$ be the structure with $B=\{b_0,b_1,b_2\}$ such that $b_0 < b_1 < b_2$, the edge $b_0b_2$ has label $E$ and the edges $b_0b_1$ and $b_1b_2$ have label $N$. Let $\str C\in \mathcal K^{<}_0$ be arbitrary. We will prove that $\str C\not\longrightarrow (\str B)^\str A_2$.

	For $x\in C$, we will denote by $[x]$ its equivalence class in $\str C$. Pick an arbitrary linear order $\ll$ of all equivalence classes in $\str C$, and define a colouring $c\colon \Emb(\str A,\str C) \to 2$, such that, for an embedding $f\colon \str A\to \str C$, we put $c(f) = 0$ if and only if $[f(a_0)] \ll [f(a_1)]$, and $c(f) = 1$ if and only if $[f(a_1)] \ll [f(a_0)]$. Clearly, $c$ attains both colours on every embedding $\str B\to \str C$.
\end{proof}

When presented with a Ramsey problem, it is often a good strategy to look for \emph{bad colourings} such as the one in Observation~\ref{obs:equiv_non_ramsey} as they often point in the direction of the desired expansion. In this concrete case, we know by Proposition~\ref{prop:ordernecessary} that we have to add an order, and consequently we will always have the structure $\str A\in \mathcal K^{<}_0$ from the proof of Observation~\ref{obs:equiv_non_ramsey}. Therefore, one needs to avoid the particular $\str B$ from the proof. Note that the key property of $\str B$ was that no matter how the equivalence classes are ordered, it contains both a copy of $\str A$ where the order of vertices agrees with the order of their equivalence classes, as well as a copy of $\str A$ where they disagree.

We say that $\str B \in \mathcal K^{<}_0$ is \emph{convexly ordered} if every equivalence class forms an interval in the order. Let $\mathcal K^{<}$ be the subclass of $\mathcal K^{<}_0$ consisting of the convexly ordered structures. We will now prove:

\begin{theorem}[\cite{Kechris2005}]
	$\mathcal K^{<}$ is Ramsey.
\end{theorem}
This result can be deduced from the product Ramsey theorem. This is probably simpler than the proof below which uses Theorem~\ref{thm:hn_completions}, however, our proof can be adapted to much more complex structures with equivalences.
\begin{proof}
	We have seen that the class $\mathcal D^<_0$ of all linear orderings of members of $\mathcal D$ is Ramsey. Let $\mathcal D^<$ be the subclass of $\mathcal D^<_0$ such that $\str A\in \mathcal D^<$ if and only if for every $x,y\in \str A$ the following are true:
	\begin{enumerate}
		\item If $x\in O_\str A$ and $y\in I_\str A$ then $x <_\str A y$, and
		\item if $x,y\in O_\str A$ and $\func{A}{}(x) <_\str A \func{A}{}(y)$ then $x <_\str A y$.
	\end{enumerate}
	We now prove that $\mathcal D^<$ is a (strong) locally finite subclass of $\mathcal D^<_0$ and then use Theorem~\ref{thm:hn_completions} to infer that $\mathcal D^<$ is Ramsey. We thus need to verify Definition~\ref{defn:locfin}. It turns out that we will only use conditions~(\ref{locallyfinite:1}) and~(\ref{locallyfinite:2}), and so we can always put $n=0$. Fix arbitrary $\str B\in\mathcal D^<$ and $\str C_0\in \mathcal D^<_0$ and let $\str C$ be a finite $(L'\cup \{<\})$-structure with a homomorphism-embedding $\str C\to \str C_0$ such that every irreducible substructure of $\str C$ extends to a copy of $\str B$ in $\str C$. Observe that $\str C\vert _{L'}$ is in $\mathcal D$ (as $\mathcal D$ is described by forbidden irreducible structures, and every irreducible substructure of $\str C\vert _{L'}$ embeds to $\str B\vert _{L'} \in \mathcal D$), and so we only need to complete the order.

	Note that---as before---the existence of a homomorphism-embedding to $\str C_0$ implies that the relation $<_\str C$ is acyclic. Consequently, we can fix a linear order $\ll_I$ on $I_\str C$ extending $<_\str C$ as well as, for every $y\in I_\str C$, a linear order $\ll_y$ on vertices $\{x\in O_\str C : F_\str C(x) = \{y\}\}$ extending $\leq_\str C$. This allows us to define a binary relation $<_{\str C'}$ on vertex set $C$ such that for every distinct $x,y\in C$ we have $x <_{\str C'} y$ if and only if one of the following holds:
	\begin{enumerate}
		\item $x\in O_\str C$ and $y\in I_\str C$,
		\item $x,y\in I_\str C$ and $x\ll_I y$,
		\item $x,y\in O_\str C$ and $F_\str C(x) \ll_I F_\str C(y)$, or
		\item $x,y\in O_\str C$, $F_\str C(x) = F_\str C(y)$, and $x \ll_{F_\str C(x)} y$.
	\end{enumerate}
	Let $\str C'$ be the structure obtained from $\str C$ by replacing $<_\str C$ with $<_{\str C'}$. It is easy to see that $<_{\str C'}$ is a linear order and that $\str C'\in \mathcal D^<$. We will show that $<_{\str C'}$ extends $<_\str C$. Towards this, observe first that there are no $x\in O_\str C$ and $y\in I_\str C$ with $y <_\str C x$ (for every $x,y\in C$ with $y <_\str C x$ we know that the substructure induced by $\str C$ on $\{x,y\}$ is irreducible, and hence embeds to $\str B\in \mathcal D^<$ where we know that all original vertices come before all imaginary ones).

	This, together with $\ll_I$ extending $<_\str C$, implies that if there are vertices $x,y\in C$ with $x<_\str C y$ and $y <_{\str C'} x$ then $x,y\in O_\str C$. Since the orders $\ll_{z}$ extend $<_\str C$ inside the equivalence classes, it follows that $F_\str C(x) \neq F_\str C(y)$, and thus, by the construction of $<_{\str C'}$, we know that $F_\str C(y) \ll_I F_\str C(x)$. Observe that the substructure of $\str C$ with vertex set $\{x,y,F_\str C(x), F_\str C(y)\}$ is irreducible and hence embeds to $\str B\in \mathcal D^<$. Consequently, since $x \leq_\str C y$, we see that $F_\str C(x) \leq_\str C F_\str C(y)$. This is in contradiction with $\ll_I$ extending $\leq_\str C$.

	So $<_{\str C'}$ is a linear order which extends $<_\str C$ and the identity is a homomorphism embedding $\str C \to \str C'$. Consequently, $\mathcal D^<$ is a (strong) locally finite subclass of $\mathcal D^<_0$. It is easy to see that it has the amalgamation property, consists of irreducible structures, and is hereditary, and so Theorem~\ref{thm:hn_completions} implies that it is Ramsey.

	\medskip

	We will now define a pair of functors $T^<\colon \mathcal D^<\to \mathcal K^{<}$ and $U^<\colon \mathcal K^{<}\to \mathcal D^<$ such that $T^<$ and $U^<$ behave exactly as $T$ and $U$ respectively on the unordered reducts, $T^<$ simply keeps the order on $O$, and $U^<$ keeps the order on the original vertices, puts all members of $O$ before all members of $I$, and orders $\func{}{}(x) < \func{}{}(y)$ if and only if all members of the equivalence class represented by $\func{}{}(x)$ are smaller than all members of the equivalence class of $\func{}{}(y)$. Note that since $\mathcal K^{<}$ consists of convexly ordered structures, the order is defined on every pair of vertices. It is easy to check that $\str A = T^<(U^<(\str A))$ and that $T^<$ and $U^<$ respect embeddings and their compositions.

	Now, pick arbitrary $\str A,\str B\in\mathcal K^{<}$. As $\mathcal D^<$ is Ramsey, there is $\str C\in \mathcal D^<$ such that $\str C\longrightarrow (U^<(\str B))^{U^<(\str A)}_2$. We will prove that $T^<(\str C) \longrightarrow (\str B)^\str A_2$. Towards that, fix an arbitrary colouring $c\colon \Emb(\str A,T^<(\str C))\to 2$. This gives rise to a colouring $c'\colon \Emb(U^<(\str A),\str C)\to 2$, and hence there is $g'\colon U^<(\str B)\to \str C$ which is monochromatic. Consequently, restricting to the original vertices, we get an embedding $g\colon \str B\to T^<(\str C)$ which is monochromatic with respect to $c$.
\end{proof}

\begin{remark}
	Note the similarity between the completion by shortest paths for metric spaces and the completion we defined for equivalences. This is not a coincidence, because if we consider metric spaces with distances $\{1,3\}$ then distance 1 describes an equivalence relation and distance 3 its complement. In general, the \emph{ultrametric spaces} are those where the allowed distances are so far apart that the triangle inequality can be equivalently stated as $\max(d(x,z),d(y,z))\geq d(x,y)$ and they describe families of refining equivalences. Ramsey expansions of ultrametric spaces have been identified by Nguyen Van Th\'e~\cite{The2010}. See~\cite{Konecny2018b} for more context and further generalizations.
\end{remark}

\medskip

We have now seen two examples of classes with obstacles to a completion of unbounded size: One was because of orders, the other because of equivalences. The way to deal with orders was to use the existence of a homomorphism-embedding into an already ordered structure, while the way to deal with equivalences was to use unary functions to explicitly represent the equivalence classes. Both these methods are, by now, standard heuristics used when applying Theorem~\ref{thm:hn_completions}:
\begin{itemize}
	\item Theorem~\ref{thm:posets} says that the class of all finite posets is Ramsey when equipped with a linear extension~\cite{Nevsetvril1984}. This can be proved using Theorem~\ref{thm:hn_completions} by letting $\mathcal K$ be the class of all finite posets with a linear extension, $\mathcal R$ the class of all linearly ordered directed graphs, and putting $n(\str B, \str C_0) = \lvert C_0\rvert$ in Definition~\ref{defn:locfin}.
	\item Presence of definable equivalences is a hindrance for local finiteness. When applying Theorem~\ref{thm:hn_completions}, one needs to first understand these definable equivalences and then \emph{eliminate imaginaries} -- introduce imaginary vertices representing the equivalence classes, link them with the original vertices using functions and, possibly, add more structure on the imaginary vertices to obtain the strong amalgamation property. Finally, one has to only expand by orders convex with respect to the imaginaries. Several examples are shown in~\cite{Hubicka2016}, see also~\cite{Hubicka2017sauerconnant} or~\cite{Konecny2018b,Hubicka2017sauer} for more complicated examples.
\end{itemize}

These heuristics have worked remarkably well: Hubička and Nešetřil~\cite{Hubicka2016} obtained canonical Ramsey expansion for all of Sauer's $S$ metric spaces~\cite{sauer2013b} where $S\cup \{0\}$ is closed, this extra condition was later removed by Hubička, Konečný, Nešetřil, and Sauer~\cite{HubickaSauerSmetric}. Aranda, Bradley-Williams, Hubi{\v c}ka, Karamanlis, Kompatscher, Konečný, and Pawliuk~\cite{Aranda2017, Aranda2017a, Aranda2017c} obtained canonical Ramsey expansion of metric spaces from Cherlin's catalogue of metrically homogeneous graphs~\cite{Cherlin2011b,Cherlin2013}.\footnote{This is how the second author jumped on the bandwagon when he, as an undergraduate, cheekily joined the 2016 Ramsey DocCourse semester organized by Nešetřil and the first author in Prague and, together with some other participants, started looking at metrically homogeneous graphs.} Braunfeld introduced $\Lambda$-ultrametric spaces (motivated by his classification of multidimensional permutation structures~\cite{sam2,braunfeld2018classification}) and found their canonical Ramsey expansion; his proof again follows this heuristic~\cite{Sam}. All these arguments are based on some analogue of Proposition~\ref{prop:metric_completion}, and Hubička, Konečný, and Nešetřil then developed an abstract framework of generalized metric spaces and identified their canonical Ramsey expansions~\cite{Hubicka2017sauerconnant,Hubicka2017sauer,Konecny2018b}. There are many more examples in Section~4 of~\cite{Hubicka2016}. We remark that some partial results have been obtained earlier by Nguyen Van Th\'e~\cite{The2010} and Soki\'c~\cite{Sokic2017}.
\begin{figure}
	\centering
	\includegraphics[scale=0.7]{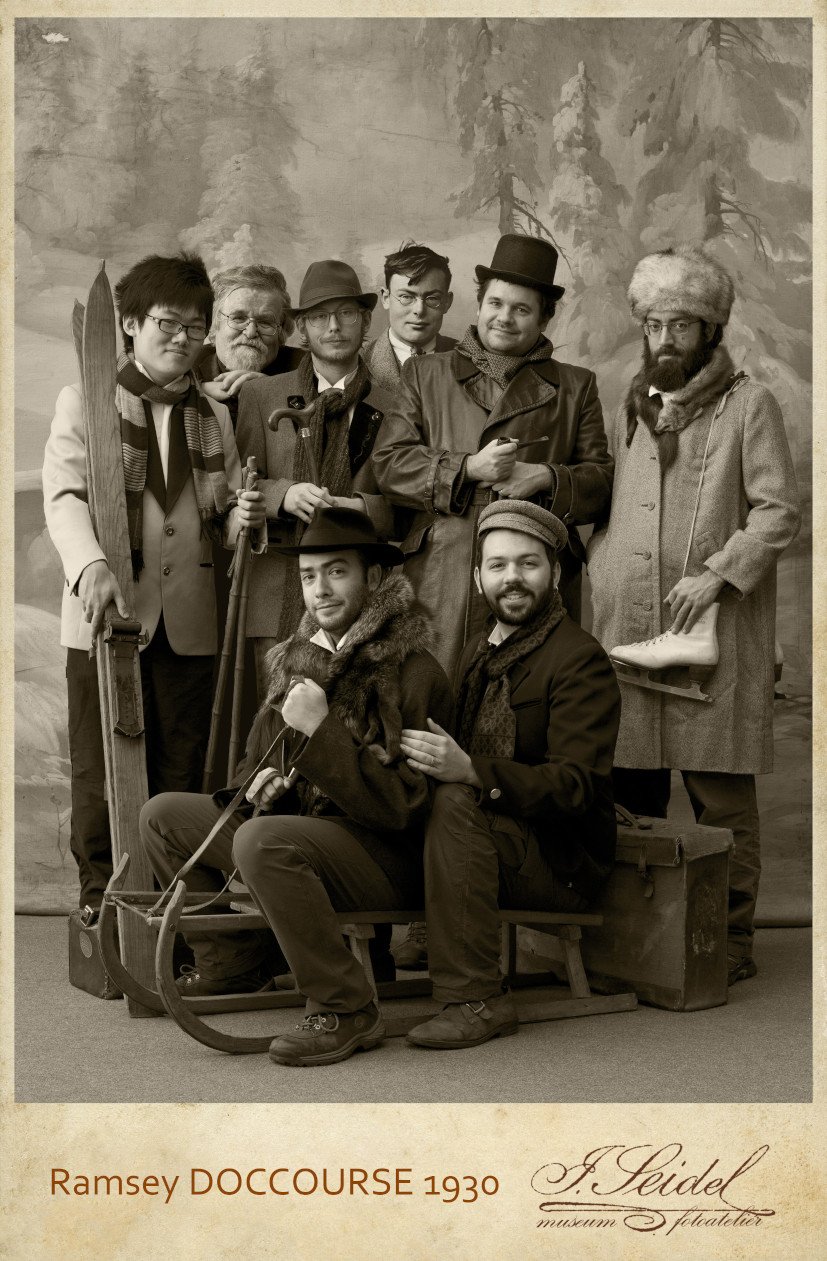}
	\caption{Some participants of the Ramsey DocCourse photographed at the world famous Seidel's studio in \v Cesk\'y Krumlov: Eng Keat Hng, Jaroslav Ne\v set\v ril, David Bradley-Williams, Frank Ramsey, Jan Hubi\v cka, and Miltiadis Karamanlis in the back, Andrés Aranda and Micheal Pawliuk at the front. The second author, being a young undergraduate, had no finances to travel to the photo studio.}
	\label{fig:doccourse}
\end{figure}

\subsubsection{The exceptions}
\label{exceptions}
The current cookbook method for obtaining a Ramsey expansion for a class $\mathcal C$ of $L$-structures is to start by understanding the definable equivalence relations on the structures, then eliminate imaginaries, understand what kind of orders the structures contain and, with all this knowledge, try to understand obstacles to completions. Hopefully, they turn out to have bounded size and then one typically needs to expand the class by convex orders in order to be able to get a linear order on all original vertices as well as imaginaries. Usually, the natural maps between the expansion of the original class, and the expansion of the class eliminating imaginaries turn out to be functorial and Theorem~\ref{thm:hn_completions} then does all the heavy lifting.

\medskip

Of course, every rule has an exception, and so there are multiple Ramsey classes where the proofs do not follow the cookbook.

\begin{example}[Two-graphs]\label{ex:twographs}
	A 3-uniform hypergraph $\str A$ is called a \emph{two-graph} if there are an even number of hyperedges on every quadruple of vertices. Two-graphs are very interesting structures and there is a rich literature about them, see \eg{}~\cite{Cameron1999,Seidel1973}. For our purposes, the important property of two-graphs is that they represent the so-called \emph{switching classes} of graphs: Given a graph $\str G$, one can define its \emph{associated two-graph} $T(\str G)$ on the same vertex set, putting a hyperedge on a triple if and only if there are an odd number of edges spanned by the triple. Conversely, for every (not necessarily finite) two-graph $\str A$, there is a graph $\str G$ such that $\str A = T(\str G)$.

	This gives us a Ramsey expansion for two-graphs by simply considering the class of all ordered graphs equipped also with the two-graph relation. It was proved by Evans, Nešetřil, and the authors (see Proposition~7.1 of~\cite{eppatwographs}) that this expansion has the expansion property, and thus is in fact the canonical one.
\end{example}

Note that the standard cookbook would only suggest adding a linear order, which does not work. It was important to understand the combinatorial properties of the class in order to guess its Ramsey expansion correctly. On the other hand, after guessing it, the Ramsey property follows directly from the Nešetřil--R\"odl theorem (Theorem~\ref{thm:NR}).

\begin{example}[Boolean Algebras]\label{ex:ba}
	Consider the class $\BA$ of all finite Boolean algebras in the language $L=\{0,1,\vee,\wedge,\neg\}$, and put $L^+ = L\cup \{<\}$. Given $\str A\in \BA$ we say that a linear order $<$ on $A$ is \emph{antilexicographic}\footnote{In~\cite{Kechris2005}, they use a slightly different (though first-order interdefinable) definition of an antilexicographic order. Our variant makes the reduction to the dual Ramsey theorem theorem slightly more streamlined.} if the following holds:
	\begin{quote}
		Let $a_0 < \cdots < a_{k-1}$ be the atoms of $\str A$, then for every $x,y\in A$ we have that $x < y$ if and only if there is $0\leq \ell<k$ such that $x\wedge a_\ell = a_\ell$, $y\wedge a_\ell = 0$, and for every $0 \leq i <\ell$ it holds that $x\wedge a_i = y\wedge a_i$.
	\end{quote}
	Let $\BA^+$ be the class of all $L^+$-structures $\str A$ such that $\str A\vert _L \in \BA$ and $<$ is antilexicographic on $\str A$. Then $\BA^+$ is the canonical Ramsey expansion of $\BA$~\cite{Kechris2005}. They used the dual Ramsey theorem which is a special case of the Graham--Rothschild theorem~\cite{Graham1971}. In fact, there is no known proof of the Ramsey property of $\BA^+$ using Theorem~\ref{thm:hn_completions}. In a way, this is not so surprising for two reasons: First, the Ramsey property of $\BA^+$ is in spirit just a reflection of the dual Ramsey theorem in the primal structural setting (see Section~\ref{sec:category}), while most of the other Ramsey classes we have seen are ``truly structural''. Second, there is at most one (isomorphism type of a) Boolean algebra in $\BA$ with $n$ atoms for every $n$, so $\BA$ is actually an unstructured category, even though presented as a class of structures. This is very similar to, say, the class of all finite linear orders, whose Ramseyness also does not follow by Theorem~\ref{thm:hn_completions}, one has to use the Ramsey theorem. Since it is very nice, we will now present a proof that $\BA^+$ is Ramsey using the dual Ramsey theorem and what is basically a straightforward upgrade of Birkhoff's duality.

	We will prove that for $\str A,\str B\in\BA^+$, if we denote by $E_A = \{a_0 < \cdots < a_{k-1}\}$ the atoms of $\str A$ and by $E_B = \{b_0 < \cdots < b_{\ell - 1}\}$ the atoms of $\str B$ then $\Emb(\str A,\str B)$ is in 1-to-1 correspondence with rigid surjections $E_B\to E_A$, where a surjection $f\colon E_B\to E_A$ is \emph{rigid} if for every $0\leq i < i' < k$ we have that $\min\{j : f(b_j) = a_i\} < \min\{j : f(b_j) = a_{i'}\}$. To simplify the arguments, we will assume that elements of $\str A$ and $\str B$ are subsets of $E_A$ and $E_B$ respectively.

	To see one direction, let $g\colon E_B\to E_A$ be a rigid surjection, and define $f\colon \str A\to \str B$ by sending $f(X) = g^{-1}[X] = \{y\in E_B : g(y) \in X\}$. Clearly, $f$ is an embedding $\str A\vert _L \to \str B_L$. To see that it is monotone with respect to $<$, pick arbitrary sets $X,Y\subseteq E_A$ such that $X<Y$. This means that $m = \min(X\triangle Y) \in X$, and so $\min(f(X) \triangle f(Y)) = \min(g^{-1}[X\triangle Y]) = g^{-1}(m) \in f(X)$. Here, $X\triangle Y$ is the symmetric difference of $X$ and $Y$.

	For the other direction, let $f\colon \str A\to \str B$ be an embedding, and define $g\colon E_B\to E_A$ such that $g(y) = x$ if and only if $y\in f(\{x\})$. Clearly, $g$ is a surjection. To see that it is rigid, pick arbitrary $0\leq i < i' < k$, and define $j = \min\{h : g(b_h) = a_i\}$ and $j' = \min\{h : g(b_h) = a_{i'}\}$. We need to prove that $j < j'$. By definition, $b_j = \min(f(\{a_i\}))$ and $b_{j'} = \min(f(\{a_{i'}\}))$. By the definition of antilexicographic order, since $a_i < a_{i'}$, we obtain that, indeed, $j<j'$.

	Observe that our maps from $\Emb(\str A,\str B)$ to rigid surjections and back are mutually inverse: This follows from the fact that each embedding $\str A\to \str B$ is uniquely determined by its restriction to ${E_A\choose 1}$ from which one can derive the corresponding rigid surjection. This implies that the Ramsey property for $\BA^+$ is equivalent to an analogue of the Ramsey theorem for rigid surjections, which happens to be known as the dual Ramsey theorem and is a special case of the Graham--Rothschild theorem~\cite{Graham1971}.
\end{example}
There are a few other examples of this kind such as finite-dimensional vector spaces over a finite field (with an analogue of the antilexicographic ordering) whose Ramseyness was proved by Graham, Leeb, and Rothschild~\cite{Graham1972} and the canonical Ramsey expansion was then discussed in~\cite{Kechris2005}. We remark that Bodirsky and Bodor~\cite{Bodirsky2024} recently obtained a Ramsey expansion of RCC5, a structure closely related to a first-order reduct of the countable atomless Boolean algebra. Although they used other methods, Theorem~\ref{thm:hn_completions} can actually be used to prove this Ramsey property by a variant of the proof that partial orders with a linear extension are Ramsey.

\begin{example}[Trees and $C$-relations]\label{ex:crelations}
	There are various powerful Ramsey theorems for partitions of trees, some of which will play a key role in Section~\ref{sec:big_ramsey}. Similarly as with Boolean algebras, some of them have a reflection in the primal structural setting (see Section~\ref{sec:category}): Let $L=\{C\}$ be a language with a ternary relation $C$, and let $\mathcal C$ be the class of all finite $L$-structures $\str A$ satisfying the following properties for every $a,b,c,d\in A$:
	\begin{enumerate}
		\item $C(a,b,c) = C(a,c,b)$,
		\item if $C(a,b,c)$ then $\neg C(b,a,c)$,
		\item if $C(a,b,c)$ then either $C(a,d,c)$ or $C(d,b,c)$,
		\item if $a\neq b$ then $C(a,b,b)$, and
		\item if $a\neq b\neq c$ then $C(a,b,c)$ or $C(b,a,c)$ or $C(c,a,b)$.
	\end{enumerate}
	One should think of vertices of $\str A$ as of the leaves of a rooted finite binary branching tree, and $C(a,b,c)$ denotes that the meet of $b$ and $c$ is further from the root than the meet of the triple $a,b,c$. Let $L^+$ be an expansion of $L$ by the order symbol $<$. Let $\mathcal C^+$ be the class of all \emph{convex orderings} of members of $\mathcal C$, that is, an $L^+$-structure $\str A$ is in $\mathcal C^+$ if and only if $\str A\vert _L \in \mathcal C$ and $<$ is a linear ordering of $A$ such that if $C(a,b,c)$ and $b<c$ then either $a<b<c$ or $b<c<a$. It was proved by Bodirsky and Piguet that $\mathcal C^+$ is Ramsey~\cite{Bodirsky2010} (they give a direct proof but also observe that it is a consequence of the Milliken theorem~\cite{Milliken1979}), Bodirsky~\cite{Bodirsky2015} later proved that $\mathcal C^+$ is the canonical expansion of $\mathcal C$.

	While $\mathcal C$ contains many non-isomorphic structures of the same cardinality, there are still relatively few of them (compared to ``really random'' structures such as graphs), and so it is perhaps not so surprising that there is no known proof of the Ramsey property of $\mathcal C^+$ using Theorem~\ref{thm:hn_completions}.
\end{example}

Soki\'c proved that semilattices with a linear extension are Ramsey~\cite{sokic2015semilattices}. Similarly as for $C$-relations, there is also no known proof using Theorem~\ref{thm:hn_completions}.

We thus have exceptions of two types: One type are structural reflections of some other Ramsey-type theorems, and it would be very interesting to see if they are bi-interpretable with a class which can be proved to be Ramsey by Theorem~\ref{thm:hn_completions} (see Question~\ref{q:boolean_allthose}). The other type are ``standard'' classes of structures whose Ramsey expansions can be obtained by Theorem~\ref{thm:hn_completions}, but they are, for some reason, more complicated and cannot be discovered using the cookbook methods. In Section~\ref{sec:orientations} we will see another example. In Section~\ref{sec:open} we will present several open problems; some of them seem to be of the first type, while others seem to be of the second type.

\subsection{A counterexample}\label{sec:orientations}
After finding many new examples of Ramsey classes, Question~\ref{q:ramsey_precompact} was eventually answered
negatively by an example found by David Evans~\cite{Evans} using a model-theoretic tool of the Hrushovski construction
(extending his earlier work~\cite{Evans2003}), showing:
\begin{theorem}[\cite{Evans2}]
	There exists a countable $\omega$-categorical structure $\str{M}$ with the property that
	if $H \leq \Aut(\str{M})$ is extremely amenable, then $H$ has infinitely many orbits on $M^2$. Consequently,
	$\str{M}$ has no precompact Ramsey expansion.
\end{theorem}
While the example originates from connections to topological dynamics, it can
be understood in purely combinatorial terms. In fact, it is a nice
example where the connections to topological dynamics aided the development on the combinatorial
side.

Given a graph $\str{G}$ with $n$
vertices and $m$ edges, we define its \emph{predimension} as $\delta(\str{G})=2n-m$
(this name will be justified later, for now, just think of it as a measurement of sparsity of the graph).

Put $\mathcal C_0$ to be the class of all finite graphs $\str{G}$ such that
$\delta(\str{H})\geq 0$ for every subgraph $\str{H}$ of $\str{G}$, and put
$\mathcal C_F$ to be the class of all finite graphs $\str{G}$ such that $\delta(\str{H})\geq \log(\vert\str{H}\vert)$ for every subgraph $\str{H}$ of $\str{G}$.

Given a graph $\str{G}$, a \emph{$2$-orientation} of $\str G$ is an oriented graph created from $\str{G}$ by choosing an orientation of every
edge such that the outdegree of every vertex is at most $2$.  Every vertex $v$ with outdegree $d<2$ is called a \emph{root}, and
the value $2-d$ is its \emph{multiplicity}. Given a vertex $u\in G$, we denote by $\mathrm{roots}(u)$ the set of all roots reachable
by an oriented path from $u$ (that is, all edges of the path are oriented away from $u$, towards the root).  We call a set $A\subseteq G$ \emph{successor-closed} if there are no edges
oriented $u\to v$ such that $u\in A$ and $v\notin A$.
We call $A$ \emph{successor-$d$-closed} if it is
successor-closed and for every vertex $u\in G\setminus A$, the  set $\mathrm{roots}(u)\setminus A$ is non-empty.
We call a graph $\str{G}$ \emph{$2$-orientable} if there exists a $2$-orientation of $\str{G}$.

Given graphs $\str{G}$ and $\str{H}$, we write $\str{G}\leq_s \str{H}$ if $\str{G}$ is an induced subgraph
of $\str{H}$ and for every $\str{G}'$ which is a subgraph of $\str{H}$ and contains $\str{G}$
it holds that $\delta(\str{G})\leq \delta(\str{G}')$.

It is easy to see that every 2-orientable graph $\str{G}$ is in $\mathcal C_0$ and
$\delta(\str{G})$ is equal to the sum of multiplicities of roots of this
orientation.  It is an interesting exercise to use  the marriage theorem to
verify the following statement:
\begin{theorem}[Theorem 3.4 of \cite{Evans2}]
	$\mathcal C_0$ is precisely the class of all finite graphs for which there exists a $2$-orientation.
	Moreover, for every $\str{G}\in \mathcal C_0$ and every $\str{H}\leq_s\str{G}$ there exists a 2-orientation of $\str{G}$
	in which $H$ is successor-closed.
\end{theorem}
From this it follows:
\begin{corollary}
	For every $\str{A}, \str{B}, \str{B}' \in \mathcal C_0$ satisfying $\str{A}\leq_s \str{B}$ and $\str{A}\leq_s \str{B'}$,
	the free amalgam of $\str{B}$ and $\str{B}'$ over $\str{A}$ is in $\mathcal C_0$.
\end{corollary}
To see this, choose $2$-orientations of $\str{B}$ and $\str{B}'$ such that $A$ is successor closed in both.
It is easy to see that they can be chosen such that the orientations of $\str{A}$ agree, then we can use free amalgamation to obtain a $2$-oriented graph.
Consequently, $\mathcal C_0$ is closed for free amalgamations over successor-closed substructures.

Given graphs $\str{G}$ and $\str{H}$, we write $\str{G}\leq_d \str{H}$ if $\str{G}$ is an induced subgraph
of $\str{H}$ and for every $\str{G}'\neq \str{G}$ which is a subgraph of $\str{H}$ and contains $\str{G}$
it holds that $\delta(\str{G})< \delta(\str{G}')$. It can be verified:
\begin{fact}
	For every $\str{G}\in \mathcal C_F$ and every subgraph $\str{H}$ of $\str{G}$ it holds that $\str{H}\leq_d\str{G}$ if and only if there exists a  2-orientation of $\str{G}$ such that $H$ is successor-$d$-closed.
\end{fact}
To see this, notice that $\str{H}\leq_s\str{G}$  and thus $\str{G}$ can be $2$-oriented so that $H$ is successor-closed.
The additional condition about roots comes from the fact that in every $2$-orientation the predimension is equal to the sum of multiplicities
of roots.
\begin{corollary}
	For every $\str{A}, \str{B}, \str{B}' \in \mathcal C_F$ satisfying $\str{A}\leq_d \str{B}$ and $\str{A}\leq_d \str{B'}$
	the free amalgam of $\str{B}$ and $\str{B}'$ over $\str{A}$ is in $\mathcal C_F$.
\end{corollary}

Classes $\mathcal C_0$ and $\mathcal C_F$ can be turned to amalgamation classes
$\mathcal C^+_0$ and $\mathcal C^+_F$ respectively by extending the language $L$ by
function symbols $\func{}{1},\func{}{2},\ldots$, one for each arity.
Given a graph $\str{G}\in \mathcal C_0$, a subset $A\subseteq G$, and an arbitrary enumeration $\vec{a}$ of $A$, we put
$\func{\str{G}}{\vert A\vert}(\vec{a})$ to be the inclusion-smallest subset $H\subseteq G$ such that
$A\subseteq H$ and $\str{H}\leq_s \str{G}$, where $\str{H}$ is the graph induced
by $\str{G}$ on $H$. Similarly, given a graph in $\str{G}\in \mathcal C_F$, a subset
$A\subseteq G$, and an arbitrary enumeration $\vec{a}$ of $A$, we put
$\func{\str{G}}{\vert A\vert}(\vec{a})$ to be the inclusion-smallest subset $H$ of $G$ such that
$A\subseteq H$ and $\str{H}\leq_d \str{G}$ where $\str{H}$ is the graph induced
by $\str{G}$ on $H$.  This yields \Fraisse{} limits $\str{M}_0$ and $\str{M}_F$ respectively.

What are the Ramsey expansions of $\mathcal C^+_0$ and $\mathcal C^+_F$?  By the
same argument as the one we discussed in Proposition~\ref{prop:ordernecessary}
(showing that the automorphism group of every Ramsey structure fixes linear order
of vertices), one can show that all extremely amenable subgroups of $\Aut(\str{M}_0)$ and $\Aut(\str{M}_F)$
must fix a 2-orientation.  Since every vertex in $\str{M}_0$ and $\str{M}_F$ has infinite degree,
there must be oriented paths of unbounded length~\cite[Theorem 3.7]{Evans2}, which means that the
expansions must be non-precompact.

This new example suggests that the traditional setup of KPT-cor\-re\-spon\-dence can be extended.
While neither $\mathcal C^+_0$ nor $\mathcal C^+_F$ have a precompact Ramsey expansion,
it is still possible to obtain an optimal expansion in the sense that the corresponding extremely
amenable subgroups $\Aut(\str{M}_0)$ and $\Aut(\str{M}_F)$ are maximal~\cite[Theorem 6.9]{Evans2}.
See~\cite{sullivan2024sparse,sullivan2024flows} for further analysis of this example.

More recently, Kwiatkowska~\cite{kwiatkowska2018universal} gave a related example, also based on a tree-like structure of the underlying class.

\subsection{Connections to computer science}
The 2010s have also been fruitful for applications of Ramsey classes in computer science, mainly through work of Bodirsky, Pinsker, and their coauthors. For example, it is decidable whether one finite-language first-order reduct of a Ramsey structure is a primitive-positive reduct of another one~\cite{Bodirsky2011a}. Recently, Rydval showed that if every finitely bounded strong amalgamation class has a computable finitely bounded Ramsey expansion (which is open) then the decision problem whether, given a finite set $\mathcal F$ of finite structures, the class of all finite $\mathcal F$-free structures is a \Fraisse{} class, is decidable~\cite{Rydval2024}. (Here, a class $\K$ of finite $L$-structures is \emph{finitely bounded} if there is a finite set $\mathcal F$ of finite $L$-structures such that an $L$-structure $\str A$ is in $\K$ if and only if no member of $\mathcal F$ embeds to $\str A$, and a structure $\str M$ is \emph{finitely bounded} if and only if $\age(\str M)$ is.)

However, the arguably most prominent application is the method of canonical functions used for classifying computational complexities of constraint satisfaction problems of first-order reducts of homogeneous structures. See~\cite{BodirskyBook} for context, details, as well as more applications, here we only review some of the results.

Given a finite relational language $L$ and an $L$-structure $\str M$, the \emph{constraint satisfaction problem} (or \emph{CSP}) over $\str M$ is the computational problem which gets a finite $L$-structure $\str A$ on the input and has to decide if there exists a homomorphism $\str A\to \str M$. Many standard computational problems are, in fact, CSPs: For example 3-SAT, satisfiability of systems of linear equations over a finite field, feasibility of linear programming, or graph 3-colourability to name a few. Following a long progress, the \emph{Finite CSP dichotomy theorem} of Bulatov and Zhuk showed that every CSP over a finite relational structure is either solvable in polynomial time, or NP-complete~\cite{Bulatov2017,Zhuk2017,Zhuk2020}. CSPs over infinite structures are much wilder (in particular, they capture even undecidable problems), but Bodirsky and Pinsker isolated a subclass for which they conjecture that a dichotomy holds:
\begin{conjecture}[Bodirsky--Pinsker, 2011; See~\cite{Bodirsky2019}]\label{conj:csp}
	Let $\str M$ be a first-order reduct of a finitely bounded homogeneous structure in a finite relational language. Then either CSP over $\str M$ is solvable in polynomial time, or it is NP-complete.
\end{conjecture}
Note that this in particular captures all finite structures. See~\cite{BodirskyBook} for more details.

The method of canonical functions (see \eg{}~\cite{Bodirsky2011} for details) was developed with the goal of verifying Conjecture~\ref{conj:csp} for first-order reducts of concrete homogeneous structures, and it has succeeded in several instances~\cite{Bodirsky2017,Bodirsky2019b,Bodirsky2015b,Bodirsky2012,Kompatscher2017}. However, as an intermediate step, one has to classify all first-order reducts of the given structure (see Section~\ref{sec:reducts}), which tends to be very laborious. Recently, Mottet and Pinsker~\cite{mottet2022smooth} introduced a powerful method of \emph{smooth approximations} which is often able to circumvent this laborious step. Nevertheless, it still uses canonical functions and thus also Ramsey expansions. Consequently, the present methods are only able to say something about the following weakening of Conjecture~\ref{conj:csp}:
\begin{quote}
	Let $\str M$ be a first-order reduct of a finitely bounded Ramsey structure in a finite relational language. Then either CSP over $\str M$ is solvable in polynomial time, or it is NP-complete.
\end{quote}
This was in fact the motivation behind Question~\ref{q:ramsey_finite}, as well as the following conjecture:
\begin{conjecture}[{Bodirsky~\cite[Conjecture~11.1.1]{BodirskyBook}}]
	Every finitely bounded homogeneous structure has a finitely bounded Ramsey expansion.
\end{conjecture}

\subsubsection{Reducts of homogeneous structures}\label{sec:reducts}
It is well-known that a permutation group $G$ is a closed subgroup of $\Sym(\mathbb N)$ if and only if it is the automorphism group of some relational structure with vertex set $\mathbb N$. One can show that a countable structure $\str M^-$ is a first-order reduct of a structure $\str M$ if and only if $\Aut(\str M^-)$ is a closed supergroup of $\Aut(\str M)$, both seen as closed subgroups of $\Sym(M)$.

The first result in this direction is by Cameron from 1976~\cite{Cameron1976} who classified all first-order reducts of $(\mathbb Q, {<})$, although he did not yet use this language. In his argument, Ramsey's theorem (or equivalently, the Ramsey property of $\age((\mathbb Q,{<}))$) was used. In 1991, Thomas~\cite{Thomas1991} used the Nešetřil--Rödl theorem to classify first-order reducts of the countable random graph as well as first-order reducts of countable homogeneous $K_m$-free graphs for $m\geq 3$. In the same paper he conjectured that every countable structure in a finite relational language has only finitely many first-order reducts up to first-order interdefinability (or, in other words, its automorphism group has only finitely many closed supergroups in $\Sym(\mathbb N)$ assuming that its vertex set is $\mathbb N$). In 1996, he verified his conjecture for first-order reducts of the random $k$-uniform hypergraph for every $k\geq 3$~\cite{thomas1996reducts}, and Junker and Ziegler~\cite{junker2008116} later verified it for first-order reducts of the order of the rationals with one constant.

Thanks to the method of canonical functions, first-order reducts of many other structures have been classified~\cite{Agarwal2016,Bodirsky2016,bodirsky201542,Linman2015,pach2014reducts,Pongrcz2017}, always relying on structural Ramsey theory as the key ingredient.

\subsection{Some concepts from category theory}
\label{sec:category}
The distinction between ``primal'' (or ``direct'') versus ``dual'', or ``unstructured'' versus ``structural'' Ramsey theorems is somewhat vague, but it is a useful concept which can be motivated by category theory.

First we discuss the difference between primal and dual Ramsey theorems. \emph{Primal} theorems are colouring morphisms from some object $A$ to an object $C$ and looking for a monochromatic morphism from an object $B$ to $C$, while \emph{dual} theorems are colouring morphisms from $C$ to $A$, looking for a monochromatic morphism from $C$ to $B$ (see Figure~\ref{fig:diagram}).

\begin{figure}
	\begin{center}
		\begin{tikzcd}
			& B & & & B\\
			A && C & A && C
			\arrow[""{name=0, anchor=center, inner sep=0}, "", from=2-1, to=1-2]
			\arrow[""{name=0, anchor=center, inner sep=0}, "", from=2-1, to=2-3]
			\arrow[""{name=0, anchor=center, inner sep=0}, "", from=1-2, to=2-3]
			\arrow[""{name=0, anchor=center, inner sep=0}, "", from=1-5, to=2-4]
			\arrow[""{name=0, anchor=center, inner sep=0}, "", from=2-6, to=2-4]
			\arrow[""{name=0, anchor=center, inner sep=0}, "", from=2-6, to=1-5]
		\end{tikzcd}
	\end{center}
	\caption{Diagram of primal (left) and dual (right) Ramsey theorem.}
	\label{fig:diagram}
\end{figure}

An example of a primal Ramsey theorem is Ramsey's theorem itself which talks about finite linear orders with embeddings, and an example of a dual Ramsey theorem is the Graham--Rothschild theorem~\cite{Graham1971} which can (for the empty alphabet) be understood as a theorem about the category of finite linear orders with surjections that are \emph{rigid} in the following sense: for every $x<y$ it holds that $\min(f^{-1}[x]) < \min(f^{-1}[y])$.
(The name rigid surjection was first used by Pr\"omel and Voigt~\cite{promel1986hereditary}, though this view of the Graham--Rothschild theorem is due to Leeb~\cite{Leeb}.)
Note that this is much stronger than the categorical dualization of Ramsey's theorem.

As far as we know, the term ``dual Ramsey theorem'' was first used in the 1980's by Ne\v set\v ril and R\"odl~\cite{Nevetvril1980dual}, and became established a few years later mainly thanks to the work of Carlson and Simpson~\cite{carlson1984}. Recently, Solecki~\cite{solecki2023dual} gave a general method for dualizing Ramsey-type statements using the Galois connection: The objects which Ramsey-type theorems are concerned with typically contain natural partial orders (for the Ramsey theorem itself this is simply the linear order, but it can be a (model theoretic) tree in the case of the Milliken tree theorem~\cite{Milliken1979}, Deuber's theorem for trees~\cite{deuber1975generalization}, or a related result of Leeb~\cite{graham1975some}). Given two partial orders $(S,\sqsubseteq_S)$ and $(T,\sqsubseteq_T)$ a pair $(f,e)$ is called a \emph{Galois connection} if $f\colon T\to S$ and $e\colon S\to T$ are maps such that
\begin{enumerate}
	\item $f\circ e$ is the identity of $S$, and,
	\item for every $a\in T$ it holds that $e\circ f(a)\sqsubseteq_T a$.
\end{enumerate}
As discussed, the classical Ramsey theorem is concerned with the category of finite linear orders with embeddings. Rigid surjections are then precisely those surjections $f$ for which there exists an embedding $e$ forming a Galois connection.  Solecki's dual Ramsey for trees~\cite{solecki2023dual} does exactly the same for the aforementioned result of Leeb where objects are trees (viewed as partial order with the meet operation). Independently, but in a similar fashion, Todor\v{c}evi\'{c} and Tyros dualized the finite variant of Milliken's tree theorem~\cite{todorcevic2022dual}.

\medskip

Second, Ramsey theorems can be \emph{structural} or \emph{unstructured}. Here the distinction is not precise, but unstructured Ramsey theorems typically talk about categories with at most one (or perhaps a few) objects of any given cardinality up to isomorphism. Prototypical examples of unstructured Ramsey theorems are Ramsey's theorem and the Graham--Rothschild theorem, while a structural Ramsey theorem is, for example, any instance of Theorem~\ref{thm:unNR} for $L\neq \{<\}$. (Although a philosophical debate might be conducted for the cases where $L$ only contains unary relations and $<$.)

Finally, given a Ramsey theorem (typically an unstructured one) there is a categorical construction suggesting its possible structural variants: Let $\mathcal C$ be a category and let $L$ be a multiset of objects of $\mathcal C$ (it is possible for an object to appear multiple times in $L$). The \emph{Prömel-indexed category} $\mathcal C^L$ has objects $(a, (e_s : s\in L))$ where $a\in \mathcal C$ and $e_s \subseteq \mathrm{hom}(s,a)$ for every $s\in L$, and given $(a, (e_s : s\in L)), (b, (f_s : s\in L)) \in \mathcal C^L$, the morphisms from $(a, (e_s : s\in L))$ to $(b, (f_s : s\in L))$ are precisely those morphisms $m\in \mathrm{hom}(a,b)$ such that for every $s\in L$ we have that $f_s = me_s$. See Figures~\ref{fig:indcat} and \ref{fig:structural}. This concept can be found in Prömel's paper~\cite{promel1985induced}\footnote{Prömel calls it just \emph{indexed category}. Since this turns out to be an established term in category theory which means something else, we added a prefix to avoid confusion.} which discusses it as a motivation to prove a dual structural Ramsey theorem. This direction was further developed by Frankl, Graham, and R\"odl~\cite{frankl1987}, and Solecki~\cite{Solecki2010} (compare the last citation with its primal form~\cite{Solecki2012}).
\begin{figure}[t]
	\begin{center}
		\begin{tikzcd} & s \\ a && b \arrow["{e_s \ni \varepsilon}"', from=1-2, to=2-1] \arrow["{\varphi \in f_s}", from=1-2, to=2-3] \arrow[""{name=0, anchor=center, inner sep=0}, "m"', from=2-1, to=2-3] \arrow["\circlearrowleft"{description}, draw=none, from=0, to=1-2] \end{tikzcd}
	\end{center}
	\caption{Morphisms in the Prömel-indexed category.}
	\label{fig:indcat}
\end{figure}
\begin{figure}[t]
	\begin{center}
		\begin{tikzcd}
			&& b \\
			\\
			&& s \\
			a &&&& c
			\arrow[""{name=0, anchor=center, inner sep=0}, " ", from=1-3, to=4-5]
			\arrow[""', from=3-3, to=1-3]
			\arrow[""', from=3-3, to=4-1]
			\arrow["", from=3-3, to=4-5]
			\arrow[""{name=1, anchor=center, inner sep=0}, "", from=4-1, to=1-3]
			\arrow[""{name=2, anchor=center, inner sep=0}, " "', from=4-1, to=4-5]
			\arrow["\circlearrowleft"{description}, draw=none, from=3-3, to=2]
			\arrow["\circlearrowleft"{description}, draw=none, from=3-3, to=0]
			\arrow["\circlearrowleft"{description}, draw=none, from=1, to=3-3]
		\end{tikzcd}
	\end{center}
	\caption{Diagram of structural Ramsey theorems.}
	\label{fig:structural}
\end{figure}

For example, let $\mathcal C$ be the category of finite linear orders with embeddings and put $L = \{E\}$ where $E$ is a 2-element linear order. Then an object $\mathcal C^L$ is a finite linear order $a$ together with a collection of embeddings from $E$ to $a$. In other words, objects in $\mathcal C^L$ can be viewed as finite linearly ordered graphs and one can check that morphisms in $\mathcal C^L$ are precisely embeddings of linearly ordered graphs. In general, the Prömel-indexed categories over finite linear orders with embeddings are isomorphic precisely to categories of finite linearly ordered relational structures with embeddings.

The partite construction is actually a method for producing structured Ramsey theorems over unstructured Ramsey categories with certain properties. This was already known to Nešetřil and Rödl, see also a recent paper of Junge~\cite{junge2023categorical} presenting the abstract partite construction in the categorical setting. In this survey we only focus on Prömel-indexed categories over finite linear orders, or in other words, classes of ordered structures, but an abstract framework, phrasing (as much as possible of) the constructions from Appendix~\ref{appendix:partite} in an abstract categorical language, ideally implying analogues of Theorems~\ref{thm:HN} and~\ref{thm:sparseningRamsey} for Ramsey categories other than linear orders with embeddings, is currently being developed.

However, in some cases, the same Ramsey theorem can be viewed through multiple lenses. For example, the Graham--Rothschild theorem for rigid surjections is, as we have seen, an obviously unstructured Ramsey theorem. However, it turns out that it is equivalent to the Ramsey property of the class of all finite Boolean algebras with an antilexicographic order (see Example~\ref{ex:ba}), which is a class of finite linearly ordered structures, and hence in the scope of Nešetřil's classification programme. Despite this, it is visible from the proof of the Graham--Rothschild theorem that the dual unstructured setting is the natural setting for it (and it is thus, perhaps, not surprising that there is no known proof of the Ramsey property of the class of all finite Boolean algebras with an antilexicographic order using Theorem~\ref{thm:hn_completions}, see also Question~\ref{q:boolean_allthose}).

Another example of this is the class of all finite (binary branching) $C$-relations with convex orders (see Example~\ref{ex:crelations}). Again, this is a class of finite relational structures with the Ramsey property. Nevertheless, these structures represent leaves of rooted binary trees with a lexicographic order. Moreover, the Ramsey property for this class is a consequence of the (finite variant of the) Milliken theorem~\cite{Milliken1979} (see~\cite{Bodirsky2010}), which is a generalization of the infinite Ramsey theorem to trees, if the Ramsey theorem is understood as a theorem about trees with no branchings. In particular, Milliken's theorem is a primal unstructured Ramsey theorem, and thus, even though $C$-relations fall into the scope of Nešetřil's classification programme, this explains why there is no known proof of the Ramsey property of the class of all finite (binary branching) $C$-relations with convex orders using Theorem~\ref{thm:hn_completions}.

\medskip

A category $\mathcal C$ is \emph{Ramsey} if for every pair of objects $A,B\in \mathcal C$, there is an object $C\in \mathcal C$ such that for every colouring $c\colon \mathcal C(A,C) \to \{0,1\}$, there is a morphism $f\colon B\to C$ such that $c$ is constant on $f\circ \mathcal C(A,B)$. If $\mathcal C$ is a class of finite structures with embeddings then this exactly corresponds to the Ramsey property. Recently, Hadek found an equivalent characterization of Ramsey categories:

\begin{theorem}[Hadek, 2025~\cite{Hadek2025}]\label{thm:hadek}
	Let $\mathcal C$ be a locally finite essentially small category. Then the following are equivalent:
	\begin{enumerate}
		\item $\mathcal C$ is confluent and Ramsey.
		\item Every $\mathcal C^{\mathrm{op}}$-shaped diagram consisting of finite non-empty sets has a solution.
	\end{enumerate}
\end{theorem}
Here, a category is \emph{locally finite} if there are only finitely many morphisms between every pair of objects, it is \emph{essentially small} if it is equivalent to a category with only set-many objects, and it is \emph{confluent} if it is a disjoint union of categories with the ``joint embedding property'' (that is, for every pair of objects $A,B$, there is an object $C$ and morphisms $A\to C$ and $B\to C$). For example, the age of every structure satisfies all three properties. A $\mathcal C^{\mathrm{op}}$-shaped diagram consisting of finite non-empty sets is a functor $D$ from $\mathcal C^{\mathrm{op}}$ to finite non-empty sets, and a \emph{solution} $s$ of $D$ assigns every object $A\in \mathcal C^{\mathrm{op}}$ an element $s(A) \in D(A)$ such that for every morphism $f\in \mathcal C^{\mathrm{op}}(A,B)$ it holds that $D(f)(s(A)) = s(B)$.

One can check that for $\mathcal C = (\omega,\leq)$ (the category with natural numbers as objects and an arrow from $n$ to $m$ if and only if $n\leq m$) the second property in Theorem~\ref{thm:hadek} is precisely K{\H o}nig's lemma~\cite{Konig1927}.

\medskip

For more interactions of category theory and structural Ramsey theory see \eg{} \cite{Graham1972,Leeb,Nevsetvril1977,nevsetvril1977partitions,nevsetril1979partition,jevzek1983ramsey,Nevsetvril1984,Nevsetril2005,Kubis2014,masulovic2016pre,draganic2019ramsey,barbosa2020categorical,bartovs2021weak,Scow2021,BartosovaSemiretractions,mavsulovic2021kechris,Iyer2022,junge2023categorical,barbosa2023tukey,solecki2024finite,Hadek2025}.

\section{2020s: Extending the classification programme}
Recently, the classification programme has grown in two new directions: to the seemingly unrelated extension property for partial automorphisms (EPPA) and to big Ramsey structures.

\subsection{Classification of EPPA classes}\label{sec:eppa}
A class $\K$ of finite structures has the \emph{extension
	property for partial automorphisms} (\emph{EPPA} for short, sometimes also called the
\emph{Hrushovski property}) if for every $\str{A}\in \K$ there exists
$\str{B}\in \K$ containing $\str{A}$ as a substructure
such that every partial automorphism of $\str{A}$ (an
isomorphism between two substructures of $\str{A}$) extends to an automorphism of
$\str{B}$.  We call such $\str{B}$  an {\em EPPA-witness for
		$\str{A}$}.

While EPPA and the Ramsey property may not seem related at first glance,
there is an analogue of Observation~\ref{obs:ramseyamalg} for EPPA:
\begin{observation}\label{obs:eppaamalgamation}
	Let $\K$ be a class of finite structures with EPPA and the joint embedding property. Then $\K$ has the amalgamation property.
\end{observation}
\begin{proof}
	Let $\K$ be such a class and let $\str{A},\str{B}_1,\str{B}_2\in \K$, $\alpha_1\colon \str A \to \str B_1$, $\alpha_2\colon \str A\to \str B_2$ be as in Definition~\ref{defn:amalg}. Pick $\str B\in \K$ with embeddings $\beta_1'\colon \str B_1\to \str B$ and $\beta_2'\colon \str B_2\to \str B$, and let $\str C\in\K$ be an EPPA-witness for $\str B$ (in particular, we know that $\str B\subseteq \str C$).

	Let $\varphi$ be the partial automorphism of $\str B$ with domain $\beta_1'\circ\alpha_1(\str A)$ such that for every $x\in A$ we have that $\varphi(\beta_1'\circ\alpha_1(x)) = \beta_2'\circ\alpha_2(x)$ and let $\theta$ be its extension to an automorphism of $\str C$ (this means that $\theta\circ\beta_1'\circ\alpha_1 = \beta_2'\circ\alpha_2$). Finally, putting $\beta_1 = \theta\circ\beta_1'$ and $\beta_2 = \beta_2'$ certifies that $\str C$ is an amalgam of $\str B_1$ and $\str B_2$ over $\str A$ with respect to $\alpha_1$ and $\alpha_2$.
\end{proof}

EPPA was originally motivated by the
\emph{small index conjecture} from group theory which states that a subgroup of the automorphism group of the countable random graph is of countable index if and only if it contains a stabilizer
of a finite set. This conjecture was proved by Hodges, Hodkinson, Lascar, and Shelah in 1993~\cite{hodges1993b}.
To complete the argument
they used a result of Hrushovski~\cite{hrushovski1992} who, in 1992, established that the class of all finite graphs has EPPA.  After this, the quest of identifying new classes of structures with EPPA continued. Herwig generalized Hrushovski's result to all finite relational structures~\cite{Herwig1995} (in a fixed finite language) and later to the classes of all finite relational structures omitting certain sets of irreducible substructures~\cite{herwig1998}. He also developed several combinatorial strategies for establishing EPPA for a given class. This early development culminated in the Herwig--Lascar theorem~\cite{herwig2000}, one of the deepest results in the area, which established a structural condition for a class to have EPPA:

\begin{theorem}[Theorem~3.2 of~\cite{herwig2000}]\label{thm:HL}
	Let $L$ be a finite relational language, let $\mathcal F$ be a finite set of finite $L$-structures, and let $\str A$ be a finite $L$-structure such that there is no $\str F\in \mathcal F$ with a homomorphism $\str F\to \str A$. Then there is a finite $L$-structure $\str B$ which is an EPPA-witness for $\str A$ such that there is no $\str F\in \mathcal F$ with a homomorphism $\str F\to \str B$.
\end{theorem}
Solecki~\cite{solecki2005} used Theorem~\ref{thm:HL} to prove EPPA for the class of all
finite metric spaces. This was independently obtained by
Vershik~\cite{vershik2008}, see also~\cite{Pestov2008,rosendal2011,rosendal2011b,sabok2017automatic,Hubicka2018metricEPPA} for other proofs. Metric spaces are a particularly nice example as the proofs have very different flavours: Some are purely combinatorial~\cite{Hubicka2018metricEPPA}, while others use the profinite topology on free groups, the {R}ibes--{Z}alesskii theorem~\cite{Ribes1993}, and Marshall Hall's theorem~\cite{hall1949}. Solecki's argument was refined by Conant~\cite{Conant2015} for
certain classes of generalized metric spaces and metric spaces with
certain forbidden subspaces.

During the 2016 meeting ``Model theory of finite and pseudofinite structures'' in Leeds, the first author noticed that the condition in Theorem~\ref{thm:HL} is a special case of Definition~\ref{defn:locfin}, and hence the proofs of EPPA using Theorem~\ref{thm:HL} and the Ramsey property using Theorem~\ref{thm:hn_completions} are similar. (Notice, though, that while Ramsey classes consist of rigid structures, such classes have EPPA only in the trivial cases.) This initiated the Classification Programme of EPPA Classes.

The first explicit contribution to this programme was given, during the Ramsey DocCourse 2016, by Aranda, Bradley-Williams, Hubi{\v c}ka, Karamanlis, Kompatscher, Konečný, and Pawliuk \cite{Aranda2017,Aranda2017c,Aranda2017a}, who, simultaneously with Ramsey expansions, classified EPPA for classes from Cherlin's catalogue of metrically homogeneous graphs (with the exception of certain antipodal classes which were later handled by the second author~\cite{Konecny2019a}). Many other EPPA classes are known~\cite{hodkinson2003,Ivanov2015,Siniora2,Hubicka2017sauer,Siniora,Evans2,HubickaSemigenericAMUC,eppatwographs,otto2017,Hallback2020,Evans3,Hubicka2018EPPA,BradleyEPPAnumbers,HubickaSemigeneric}.

Inspired by both the statement and the proof strategy of Theorem~\ref{thm:hn_completions}, Hubička, Konečný, and Nešetřil proved an analogous result for EPPA. In order to state it, we need two more definitions:

\begin{definition}
	Let $\str B$ be an EPPA-witness for $\str A$. We say that $\str B$ is \emph{irreducible-structure faithful} if whenever $\str D\subseteq \str B$ is irreducible, there is an automorphism $f\in \Aut(\str B)$ such that $f(\str D) \subseteq \str A$.
\end{definition}
Notice that this implies that every irreducible substructure of $\str B$ extends to a copy of $\str A$ in $\str B$, a property which we are already familiar with thanks to Theorem~\ref{thm:sparseningRamsey} and Definition~\ref{defn:locfin}. This notion originates with Hodkinson and Otto~\cite{hodkinson2003} who introduced it for relational structures under the name \emph{clique faithfulness} and proved that if a relational structure $\str A$ has an EPPA-witness then there is in fact a clique faithful EPPA-witness for $\str A$. The general notion was introduced by Evans, Hubička, and Nešetřil who proved the existence of irreducible-structure faithful EPPA-witnesses in languages where all functions are unary~\cite{Evans3}.

\begin{definition}[Coherent EPPA~\cite{Siniora}]
	\label{defn:coherent}
	Let $\str B$ be an EPPA-witness for $\str A$. We say that $\str B$ is \emph{coherent} if there is a map $\Psi$ from partial automorphisms of $\str A$ to $\Aut(\str B)$ such that if $f,g$ are partial automorphisms of $\str A$ such that $\dom(g) = \range(f)$ then $\Psi(g\circ f) = \Psi(g)\circ\Psi(f)$. A class $\mathcal C$ has coherent EPPA if it has EPPA and the EPPA-witnesses can be chosen to be coherent.
\end{definition}
This dates back to Bhattacharjee and Macpherson~\cite{Bhattacharjee2005} who proved that graphs have coherent EPPA and that this implies that the automorphism group of the random graph contains a dense locally finite subgroup. Solecki and Siniora~\cite{solecki2009,Siniora} isolated this definition and proved its dynamical consequence, as well as a sufficient condition for a class to have coherent EPPA (a generalization of Theorem~\ref{thm:HL}).

We are now ready to state the EPPA analogue of Theorem~\ref{thm:sparseningRamsey}:
\begin{theorem}[\cite{Hubicka2018EPPA}]
	\label{thm:hkn}
	Let $L$ be a language consisting of relations and unary functions, let $\str A$ be a finite irreducible $L$-structure, let $\str{B}_0$ be a finite EPPA-witness for $\str{A}$ and let $n$ be an integer. There is a finite $L$-structure $\str{B}$ satisfying the following.
	\begin{enumerate}
		\item $\str B$ is an irreducible-structure faithful EPPA-witness for $\str A$.
		\item There is a homomorphism-embedding $\str B\to \str B_0$.
		\item For every substructure $\str{C}$ of $\str{B}$ on at most
		      $n$ vertices there is a tree amalgam $\str D$ of copies of $\str A$ and a homomorphism-embedding $f\colon \str C\to\str D$.
		\item If $\str B_0$ is coherent then so is $\str B$.
	\end{enumerate}
\end{theorem}
It is an evidence of the mutually beneficial interaction between the classification programmes of EPPA and Ramsey classes that the statement of Theorem~\ref{thm:sparseningRamsey} was actually inspired by Theorem~\ref{thm:hkn}, as in~\cite{Hubicka2016} only Theorem~\ref{thm:hn_completions} was stated (see Remark~\ref{rem:tree_amalgam_ramsey}). On the other hand, Theorem~\ref{thm:hkn} was first isolated as an intermediate step towards proving an EPPA analogue of Theorem~\ref{thm:hn_completions} which we now state:

\begin{theorem}[\cite{Hubicka2018EPPA}]
	\label{thm:hkn_completions}
	Let $L$ be a language consisting of relations and unary functions and let $\mathcal E$ be a class of finite $L$-structures which has EPPA. Let $\mathcal K$ be a hereditary locally finite automorphism-preserving subclass of $\mathcal E$ with the strong amalgamation property consisting of irreducible structures. Then $\mathcal K$ has EPPA. If $\mathcal E$ has coherent EPPA then so does $\mathcal K$.
\end{theorem}
Here, if $f\colon\str C\to\str C'$ is a homomorphism-embedding, we say that $\str C'$ is an \emph{automorphism-preserving completion of $\str C$} if $f$ is injective (that is, $\str C'$ is a strong completion of $\str C'$) and moreover there is a group homomorphism $f_\Gamma\colon \Aut(\str C)\to \Aut(\str C')$ such that for every $g\in \Aut(\str C)$ it holds that $f\circ g \circ f^{-1}\subseteq f_\Gamma(g)$. One can then generalize Definition~\ref{defn:locfin} to a \emph{locally finite automorphism-preserving subclass} by demanding that the completion $\str C'$ is automorphism-preserving.

\medskip

There is also an EPPA analogue of Theorem~\ref{thm:HN}. Note that unlike Theorem~\ref{thm:HN}, Theorem~\ref{thm:nreppa} does not explicitly talk about forbidden irreducible structures. However, from irreducible-structure faithfulness of $\str B$ it follows that if $\str A$ avoids some irreducible structure then so does $\str B$.
\begin{theorem}[\cite{Evans3}]
	\label{thm:nreppa}
	Let $L$ be a language consisting of relations and unary functions and let $\str A$ be a finite $L$-structure. There is a finite $L$-structure $\str B$, which is an irreducible-structure faithful coherent EPPA-witness for $\str A$. Consequently, the class of all finite $L$-structures has coherent EPPA.
\end{theorem}

Note that in contrast to the Ramsey property, the EPPA results can only handle unary functions. To partially compensate for it, the true statements of the theorems in~\cite{Hubicka2018EPPA} allow for the language to be equipped with a permutation group (see~\cite{Hubicka2018EPPA} for details). Nevertheless, EPPA for, say, the class of all finite structures with a single binary function, remains an important open problem.

\subsubsection{Examples of classes with EPPA}
Theorem~\ref{thm:hkn_completions} can be applied similarly to Theorem~\ref{thm:hn_completions}. We will give an example:
\begin{example}[Metric spaces]
	In 2005, Solecki~\cite{solecki2005} proved that the class of all finite metric spaces has EPPA. The key part of his argument was essentially the construction in the proof of Proposition~\ref{prop:metric_completion} together with the observation that the completion is actually automorphism-preserving. He then used Theorem~\ref{thm:HL}, but with today's methods one can simply say that the class $\mathcal K$ of all finite metric spaces is a locally finite automorphism-preserving subclass of the class $\mathcal R$ of all finite structures in the language with a binary relation for every positive real number, which has coherent EPPA by Theorem~\ref{thm:hkn}, and so we can use Theorem~\ref{thm:hkn_completions}.
\end{example}

Notice that we obtained all the ingredients for a proof of EPPA in an intermediate step for a proof of the Ramsey property. However, sometimes this can also go the other way around:
\begin{example}[Diversities]
	A \emph{diversity} $\str A$ is a set $A$ equipped with a function $d\colon {A\choose {<}\omega}\to \mathbb R^{{\geq}0}$ such that for all finite subsets $X,Y,Z$ of $A$ it holds that $d(X) = 0$ if and only if $\lvert X\rvert \leq 1$, and if $Y\neq \emptyset$ then $d(X\cup Y) + d(Y\cup Z) \geq d(X\cup Z)$. Note that when we restrict $d$ to pairs, we obtain a metric. Diversities are thus generalizations of metric spaces (see for example~\cite{Hallbackphd} for more context), and in an analogous way as with metric spaces, we can see diversities as relational structures.

	Hallbäck used Theorem~\ref{thm:HL} to prove that the class $\mathcal D$ of all finite diversities has EPPA. However, in the course of applying Theorem~\ref{thm:HL}, he actually proved that $\mathcal D$ is a locally finite automorphism-preserving subclass of the class $\mathcal R$ of all finite relational structures in the corresponding language (see the proof of Theorem~5.6 in~\cite{Hallback2020}).

	We thus immediately get the following two theorems as corollaries of his construction:
	\begin{theorem}
		$\mathcal D$ has coherent EPPA.
	\end{theorem}
	\begin{theorem}
		The class of free orderings of $\mathcal D$ is Ramsey.
	\end{theorem}
	\begin{proof}
		Since members of $\mathcal D$ are irreducible and automorphism-preserving completions are strong, we can use Lemma~\ref{lem:locally_finite_add_order} to lift Hallbäck's result to free orderings of $\mathcal D$ and $\mathcal R$, and apply Theorems~\ref{thm:HN} and~\ref{thm:hn_completions}.
	\end{proof}
	This answers a question of Hallbäck~\cite[Question~A.6]{Hallbackphd}.
\end{example}

There are, however, also structures for which EPPA does not follow from Theorem~\ref{thm:hkn_completions} so directly. Such examples are two-graphs and antipodal metric spaces~\cite{eppatwographs,Konecny2019a}, or semigeneric and $n$-partite tournaments~\cite{HubickaSemigeneric,HubickaSemigenericAMUC}:
\begin{example}[$n$-partite tournaments]
	Given $n\geq 2$, an $n$-partite tournament is an oriented graph $\str A$ whose vertices can be partitioned into parts $A_1,A_2,\ldots,A_n$ such that there are no edges between vertices from the same part and there is exactly one oriented edge between every pair of vertices from different parts. In~\cite{HubickaSemigeneric}, Hubička, Jahel, Konečný, and Sabok prove EPPA for $n$-partite tournaments; we show their construction below:

	Fix a finite $n$-partite tournament $\str A$ with parts
	$A_1,A_2,\ldots,A_n$. We will give an explicit construction of an $n$-partite tournament $\str B$ which is an EPPA-witness for $\str A$.
	Assume without loss of generality that $A=\{1,2,\ldots, k\}$, for every $x\in A_i$ and every $y\in A_j$ it holds that $x<y$ whenever $i<j$, and $\vert A_1\vert =\vert A_2\vert =\cdots=\vert A_n\vert$.

	Given vertex $x\in A_i$ for some $1\leq i\leq n$, we put $N(x) = A\setminus A_i$. The vertex set of $\str B$ consists of all pairs $(x,\chi)$ where $x\in A$ and $\chi$ is a function $N(x)\to \mathbb Z_2$. We let $(x,\chi)$ and $(x',\chi')$ be adjacent if and only if $x$ and $x'$ belong to different parts of $\str A$, and orient the edge
	from $(x,\chi)$ to $(x',\chi')$ if and only if either $x > x'$ and $\chi(x') + \chi'(x) = 1$, or $x < x'$ and $\chi(x') + \chi'(x) = 0$ (with addition in $\mathbb Z_2$). It is easy to observe that $\str B$ is an $n$-partite tournament with parts $B_i=\{(x,\chi) \in B : x\in A_i\}$.

	Next, we construct an embedding $\psi\colon \str A\to \str B$ (so that $\str B$ will then extend partial automorphisms of $\psi(\str A)$), putting $\psi(x) = (x,\chi_x)$ where $\chi_x\colon N(x) \to \mathbb Z_2$ satisfies $\chi_x(y)=1$ if and only if $y<x$ and there is an edge directed from $x$ to $y$ in $\str A$. It is easy to verify that $\psi$ is indeed an embedding of $\str A$ into $\str B$. Showing how to extend partial automorphisms of $\str A$ is a bit laborious, see~\cite{HubickaSemigeneric}.
\end{example}

The other constructions in~\cite{eppatwographs} or~\cite{HubickaSemigeneric} also have a similar product structure and do not invoke Theorems~\ref{thm:hkn} and~\ref{thm:nreppa}. However, the structures handled in~\cite{Konecny2019a} needed a combination of the ``ad hoc'' product argument and Theorems~\ref{thm:hkn} and~\ref{thm:nreppa} as they were metric-like. There, one needed the more general statements of Theorems~\ref{thm:hkn} and~\ref{thm:nreppa} (where there is additionally a permutation group on the language and a morphism from $\str A$ to $\str B$ is a pair $(f,g)$ where $g$ is an element of the permutation group and $f$ is an embedding from $\str A$ to $\str B^g$ where the relations are permuted by $g$, see~\cite{Hubicka2018EPPA} for details). In general, it seems that a possible cookbook for proving EPPA might consists of combining such a product structure with Theorems~\ref{thm:hkn} and~\ref{thm:nreppa}.

Of course, like with the Ramsey property, there are classes which are complete exceptions. Here, they come from two directions: First, there are various classes of finite structures which are homogeneous (and, hence, EPPA-witnesses for themselves), such as vector spaces over finite fields, or the class of all finite skew-symmetric bilinear forms~\cite{Cherlin2003}. Second, Siniora~\cite{Siniora2} proved that the class of all finite groups has EPPA.

\subsubsection{Topological dynamics}
Just like the Ramsey property, EPPA also has strong links
to topological dynamics. Let $\str M$ be the \Fraisse{} limit of a \Fraisse{} class, and assume for simplicity that $\str M$ is relational. Kechris and Rosendal~\cite{Kechris2007} observed that $\age(\str M)$ has EPPA if and only if $\Aut(\str M)$ can be written as the closure of a countable chain of compact subgroups. Siniora and Solecki~\cite{solecki2009,Siniora} proved that if $\age(\str M)$ has coherent EPPA then $\Aut(\str M)$ contains a dense locally finite subgroup, and Sági recently isolated a combinatorial property equivalent to the existence of a dense locally finite subgroup~\cite{Sagi2024}.

As we stated at the beginning of this section, the original motivation behind Hrushovski's proof of EPPA for graphs was the proof of the small index property of the random graph by Hodges, Hodkinson, Lascar, and Shelah~\cite{hodges1993b}. In fact, they proved that the random graph has \emph{ample generics} (its automorphism group $G$ has a comeagre orbit in its action on $G^n$ by diagonal conjugation for every $n\geq 1$), the small index property then followed as a corollary. Kechris and Rosendal~\cite{Kechris2007} abstracted their methods and proved various other corollaries of ample generics.

The nowadays standard method for proving ample generics for $\str M$ is proving EPPA and APA for $\age(\str M)$, where a class $\mathcal C$ has \emph{APA}, the \emph{amalgamation property with automorphisms}, if for every $\str{A},\str{B}_1,\str{B}_2\in \mathcal C$ and embeddings $\alpha_1\colon\str{A}\to\str{B}_1$, $\alpha_2\colon\str{A}\to\str{B}_2$ there exists $\str{C}\in \mathcal C$ with embeddings $\beta_1\colon\str{B}_1 \to \str{C}$ and $\beta_2\colon\str{B}_2\to\str{C}$ such that $\beta_1\circ\alpha_1 = \beta_2\circ\alpha_2$ (\ie{} $\str C$ is an amalgam of $\str B_1$ and $\str B_2$ over $\str A$ with respect to $\alpha_1$ and $\alpha_2$) and moreover whenever we have $f\in \Aut(\str B_1)$ and $g\in \Aut(\str B_2)$ such that $f(\alpha_1(A)) = \alpha_1(A)$, $g(\alpha_2(A)) = \alpha_2(A)$ and for every $a\in A$ it holds that $\alpha_1^{-1}(f(\alpha_1(a))) = \alpha_2^{-1}(g(\alpha_2(a)))$ (that is, $f$ and $g$ agree on the copy of $\str A$ we are amalgamating over), then there is $h\in \Aut(\str C)$ which extends $\beta_1 f \beta_1^{-1} \cup\beta_2 g \beta_2^{-1}$. Every free amalgamation class has the amalgamation property with automorphisms. Similarly as a strong completion of the free amalgam is a strong amalgam, automorphism-preserving completions of the free amalgam give rise to strong amalgams with automorphisms. Consequently, metric spaces or diversities do have APA, and hence ample generics.

\subsubsection{Classification Programme of EPPA Classes}

In the light of the recent connections between EPPA and Ramsey classes, we can extend
Nešetřil's classification programme to classes with EPPA. The basic
question for a given locally finite homogeneous $L$-structure $\str{M}$ is simply:

\begin{enumerate}[label=Q\arabic*,resume]
	\item \label{Q3} Does the class $\Age(\str{M})$ have EPPA?
\end{enumerate}
Motivated by the group-theoretic context, we
can additionally consider the following questions:
\begin{enumerate}[label=Q\arabic*,resume]
	\item \label{Q4} If the answer to~\ref{Q3} is positive, does $\Age(\str{M})$ have coherent EPPA?
	\item \label{Q5} If the answer to~\ref{Q3} is positive, does $\Age(\str{M})$ have APA?
	\item \label{Q6} If the answer to~\ref{Q3} is positive but the answer to~\ref{Q5} is negative, does $\Age(\str{M})$ still have ample generics?
\end{enumerate}

There are many open problems in this area, some of which we present in Section~\ref{sec:eppa_questions}.

\subsection{Infinite structural Ramsey theory}\label{sec:big_ramsey}
As discussed in Section~\ref{sec:1970s}, the notion of a Ramsey class was introduced as
a natural generalization of the finite Ramsey theorem.  Now we turn our attention
to the infinite Ramsey theorem:

\begin{theorem}[Infinite Ramsey theorem, 1930~\cite{Ramsey1930}]
	$$(\forall k,r\in \mathbb N)\omega \longrightarrow (\omega)^k_r$$
\end{theorem}
Again, we can re-formulate it using the structural partition arrow. Recall that we
denote by $\mathcal{LO}$ the set of all finite linear orders.  Let $(\omega,{<})$ be the order of the non-negative integers (seen as a structure). Then:
$$(\forall \str{A}\in \mathcal{LO})(\forall r\in \mathbb N)(\omega,{<}) \longrightarrow ((\omega,{<}))^\str{A}_r$$
It is natural to ask if $(\omega,{<})$ can be replaced by, say, $(\mathbb Q,{<})$ (the order of the rationals).
In 1933, Sierpi\'nski showed that this variation is not true~\cite{Sierpinski1933}:  Choose an arbitrary bijection $\varphi\colon\mathbb Q\to \omega$.
Let $\str{A}\in \mathcal{LO}$ be the (unique) linear order with $A=\{0,1\}$ and the natural ordering.
Define $\chi\colon\Emb(\str{A},\mathbb Q)\to 2$ by comparing the order of the rationals and the order of the non-negative integers. Explicitly:
$$\chi(f)=\begin{cases}
		0 & \hbox{if }\varphi\circ f(0)<\varphi\circ f(1)  \\
		1 & \hbox{if }\varphi\circ f(0)>\varphi\circ f(1).
	\end{cases}
$$
Called the \emph{Sierpi\'nski colouring}, $\chi$ has the property that
for every embedding $f\colon (\mathbb Q,{<})\to (\mathbb Q,{<})$ it attains both colours on $\{f\circ e:e\in
	\Emb(\str{A},\mathbb Q)\}$.  Indeed, assume for a contradiction that there is
$f\colon (\mathbb Q,{<})\to (\mathbb Q,{<})$ such that $\chi$ is constant (say, 0) on $\{f\circ e:e\in \Emb(\str{A},\mathbb Q)\}$.
This means that the order of the rationals is isomorphic to the order of $\omega$. Similarly, if $\chi$ is
constant 1 then the order of the rationals would be isomorphic to reversed order of
$\omega$ and neither of these two are true.

In 1968, Galvin showed that this is the ``worst possible counterexample'' in the following sense.
\begin{theorem}[Galvin, 1968~\cite{Galvin68,Galvin69}]
	Let $\str{A}$ be the linear order of size 2.
	Then for every integer $r$ and every colouring $\chi\colon\Emb(\str{A},\mathbb Q)\to r$
	there exists an embedding $f\colon (\mathbb Q,{<})\to (\mathbb Q,{<})$ such that
	$\vert \{\chi(f\circ e):e\in \Emb(\str{A},\mathbb Q)\}\vert \leq 2$.
\end{theorem}

\subsubsection{Big Ramsey degrees}
Galvin's result shows that no matter how many colours a given finite colouring uses, we can always
find a copy that uses at most 2 colours.
In order to be able to write such results using the partition arrow notation, we extend it
as follows.
Given structures $\str{A}$, $\str{B}$, $\str{C}$, and integers $r$ and $t$ the \emph{extended structural Erd\H os--Rado partition arrow}
written as $\str{C}\longrightarrow (\str{B})^\str{A}_{r,t}$, is a shortcut for the following statement:
\begin{quote}
	For every $r$-colouring $\chi\colon\Emb(\str{A},\str{C})\to r$, there exists an embedding
	$f\in \Emb(\str{B},\str{C})$ such that $\vert\{\chi(f\circ e):e\in \Emb(\str{A},\str{B})\}\vert\leq t$.
\end{quote}
For a countably infinite structure $\str{B}$ and its finite substructure
$\str{A}$, the \emph{big Ramsey degree} of $\str{A}$ in $\str{B}$ is the least
number $\ell\in \mathbb N\cup \{\infty\}$ such that $\str{B}\longrightarrow
	(\str{B})^\str{A}_{r,\ell}$ for every $r\in \mathbb N$.
We say that \emph{the big
	Ramsey degrees of $\str{B}$ are finite} if for every finite substructure
$\str{A}$ of $\str{B}$ the big Ramsey degree of $\str{A}$ in $\str{B}$ is
finite.

\begin{remark}
	The name ``big Ramsey degree`` was introduced by Kechris, Pestov, and Todor\v cevi\' c~\cite{Kechris2005},
	motivated by earlier work on (small) Ramsey degrees by Fouch{\'e}~\cite{Fouche1997,fouch1998symmetries} which can be seen as an alternative approach to Ramsey expansions (cf. Theorem~\ref{thm:nr_degrees}).
\end{remark}

Laver proved that Galvin's result can be strengthened significantly:
\begin{theorem}[{Laver, 1969, Unpublished, See e.g.~\cite[Section 6.3]{todorcevic2010introduction}}]
	\label{thm:laver}
	The big Ramsey degrees of $(\mathbb Q,{<})$ are finite. Specifically, if we denote by $\mathcal{LO}$ the class of finite linear orders then the following holds:
	$$(\forall \str{A}\in \mathcal {LO})(\exists t\in \mathbb N)(\forall r\in \mathbb N) (\mathbb Q,{<})\longrightarrow ((\mathbb Q,{<}))^\str{A}_{r,t}.$$
\end{theorem}
Laver originally proved Theorem~\ref{thm:laver} by an argument described as ``long and messy'' by Devlin~\cite{devlin1979}.
However, today it can be proved by a particularly elegant application of Milliken's tree theorem~\cite{Milliken1979}, see~\cite[Section 6]{todorcevic2010introduction}.
Milliken's theorem itself builds on the well-known Halpern--L{\"a}uchli theorem~\cite{Halpern1966}, which was already known in 1969, but Laver was unaware of it.

In 1979, Devlin proved the following curious result fully determining the big Ramsey degrees of $(\mathbb Q,{<})$:
\begin{theorem}[Devlin, 1979~\cite{devlin1979}]
	\label{thm:devlin}
	For every $\str{A}\in \mathcal {LO}$ the big Ramsey degree of $\str{A}$ in $(\mathbb Q,{<})$ is
	precisely the $\vert A\vert $-th \emph{odd tangent number}: the $(2\vert A\vert -1)$-th derivative of $\tan(z)$ evaluated at $0$.
\end{theorem}

The odd tangent numbers are defined by the recurrence $t_1=1$ and $t_k=\sum_{\ell=1}^{k-1}\binom{2k-2}{2\ell-1}t_l\cdot t_{k-\ell}$, which happens to also count the number of certain special trees, called \emph{Devlin types}~\cite[Section 6.3]{todorcevic2010introduction}. It is a well-known sequence (A000182 in the On-line Encyclopedia of Integer Sequences) starting by 1, 2, 16, 272, 7936, 353792, 22368256, 1903757312, 209865342976, \ldots

\subsubsection{Classification of structures with finite big Ramsey degrees}
The study of big Ramsey properties of the order of the rationals ran in parallel with the development described in Section~\ref{sec:1970s}. As probably the first interaction between these two areas, Devlin's thesis~\cite{devlin1979} develops a framework directly inspired by the notation used in the original proof of the Ne\v set\v ril--R\"odl theorem~\cite{nevsetvril1977partitions}.
The connection to Ramsey classes was further developed by Kechris, Pestov, and Todor\v cevi\' c in the concluding remarks of~\cite{Kechris2005}
and is more apparent in the following related result.

We call a structure $\str{A}$ \emph{universal} for a class of structures $\mathcal K$ if every structure in $\mathcal K$
has an embedding to $\str{A}$.  By essentially the same proof as the one for Theorem~\ref{thm:laver}, one gets:
\begin{theorem}[Sauer, 2006~\cite{Sauer2006}]
	\label{thm:biggraphs}
	Let $\str{R}$ be a countable universal graph (\ie{}, a countable graph universal for the class of all countable graphs) and let $\mathcal K$ be the class of all finite graphs. Then
	$$(\forall \str{A}\in \mathcal K)(\exists t\in \mathbb N)(\forall r\in \mathbb N) \str{R}\longrightarrow (\str{R})^\str{A}_{r,t}.$$
	Consequently, the big Ramsey degrees of $\str{R}$ are finite.
\end{theorem}
(A particularly nice choice of such a universal graph $\str{R}$ is the \Fraisse{} limit of the class of all finite graphs, the \emph{Rado graph}.)
Theorem~\ref{thm:biggraphs} can be seen as an infinitary extension of the following easy consequence of the unrestricted Nešetřil--R\"odl theorem (Theorem~\ref{thm:unNR}):
\begin{theorem}\label{thm:nr_degrees}
	Let $\mathcal K$ be the class all finite graphs. Then
	$$(\forall \str{A}\in \mathcal K)(\exists t\in \mathbb N)(\forall \str{B}\in \mathcal K)(\forall r\in \mathbb N)(\exists \str{C}\in \mathcal K) \str{C}\longrightarrow (\str{B})^\str{A}_{r,t}.$$
\end{theorem}
By Theorem~\ref{thm:unNR}, for every given $\str{A}$ it is possible to put $t=\lvert A\rvert!$, which is the number of expansions of the graph $\str{A}$ to an ordered graph.\footnote{Note that if one was colouring copies instead of embeddings, $t=\frac{\lvert A\rvert!}{\lvert \Aut(\str A)\rvert}$ would be sufficient.}
Values of $t$ given by Theorem~\ref{thm:biggraphs} are much bigger, but still finite.
(Note that the proof in~\cite{Sauer2006} gives upper bounds only.  Big Ramsey degrees of finite graphs in $\str{R}$ were subsequently characterized by Laflamme, Sauer, and Vuksanovic~\cite{Laflamme2006}, and computed by Larson~\cite{larson2008counting}. They are directly related to Devlin types. For example, the big Ramsey degree of $\str{K}_k$ in $\str{R}$ is $t_k\cdot k!$ where $t_k$ is the $k$-th odd tangent number.)

In general, one way to prove that a given class $\mathcal K$ has a precompact Ramsey expansion is by showing that its \Fraisse{} limit has finite big Ramsey degrees: By an easy compactness argument, finiteness of big Ramsey degrees implies finiteness of small Ramsey degrees which, by~\cite{zucker2016topological}, implies the existence of a reasonable precompact Ramsey expansion. Every such expansion can then be systematically reduced to one with the expansion property~\cite{NVT14,zucker2016topological}.
We can thus extend Nešetřil's classification programme by the following question:
\begin{enumerate}[label=Q\arabic*,resume]
	\item \label{Q7} Given an amalgamation class $\mathcal K$ with a precompact Ramsey expansion, is it true that its \Fraisse{} limit has finite big Ramsey degrees?
\end{enumerate}
\begin{remark}
	In the context of (finite) structural Ramsey theory it may seem unnatural, but \Fraisse{} limits (or $\omega$-categorical structures) are not the most fitting setup for the study
	of big Ramsey degrees. In fact,~\ref{Q7} should be considered in a broader sense for universal structures (see Section~\ref{sec:types} and also \cite{hubicka2024survey,typeamalg,mavsulovic2020finite} for some first steps in this direction).
\end{remark}
\subsubsection{Recent progress}
Until recently, Question~\ref{Q7} seemed very ambitious primarily due to a complete lack of proof techniques for bounding big Ramsey degrees of structures with forbidden substructures, as well as
for structures in languages containing relations of arity 3 or more.
Extending earlier vertex- and edge-colouring results~\cite{komjath1986,Elzahar1989,sauer1998}, the first big Ramsey result for a restricted class was obtained by Dobrinen:
\begin{theorem}[Dobrinen, 2020~\cite{dobrinen2017universal}]
	\label{thm:natasha}
	Denote by $\mathcal K$ the class of all finite triangle-free graphs and let $\str R_3$ be its \Fraisse{} limit. Then $\str R_3$ has finite big Ramsey degrees.
\end{theorem}
Dobrinen introduced a new proof technique which shows a Ramsey theorem for special trees, called \emph{coding trees}, by an argument inspired by
Harrington's unpublished proof of the Halpern--L\"auchli theorem (the pigeonhole version of Milliken's tree theorem), which uses the language of set-theoretic forcing~\cite{moore2019method,zucker2020,dobrinen2017forcing}.
This method was later extended to graphs omitting larger cliques~\cite{dobrinen2019ramsey}.
Zucker simplified the argument and obtained the first general result in the area:
\begin{theorem}[Zucker, 2022~\cite{zucker2020}]
	\label{thm:zucker}
	Let $L$ be a finite language consisting of unary and binary relation symbols only. Let $\mathcal F$ be a finite set of finite irreducible $L$-structures.
	Denote by $\str{H}$ the \Fraisse{} limit of the class of all finite $\mathcal F$-free structures.
	Then $\str{H}$ has finite big Ramsey degrees.
\end{theorem}
For finite binary languages and finite sets $\mathcal F$, Zucker's result directly implies an existential version of the Ne\v set\v ril--R\"odl theorem (Theorem~\ref{thm:NR}), that is, that there exists a Ramsey expansion, and with a careful analysis, one can even recover the exact expansion. Thus, Theorem~\ref{thm:zucker} can be seen as an important step towards an infinite generalization of the Ne\v set\v ril--R\"odl theorem.
Again, both Dobrinen's and Zucker's proofs give upper bounds on the big Ramsey degrees only. A precise characterization was subsequently obtained by Balko, Chodounský, Dobrinen, Hubička, Konečný, Vena, and Zucker~\cite{Balko2021exact} in a quite technically challenging analysis. See also~\cite{Balko2023,Vodsedalek2025,Vodsedalek2025bc} for a more compact presentation of the special cases of universal $\str{K}_3$-free and $\str{K}_4$-free graphs.

Coding tree methods have been adapted to other structures~\cite{coulson2022SDAP} and extended to Souslin-measurable colourings of infinite substructures~\cite{dobrinen2023infinite}.  See Dobrinen's survey for more historical context and an overview of the usage of coding trees in the area~\cite{dobrinen2021ramsey}, and also the catalogue of known results in~~\cite{hubicka2024survey}.

An easy proof of Theorem~\ref{thm:natasha} using the Carlson--Simpson theorem~\cite{carlson1984} was found by the first author. He also showed using the same method:
\begin{theorem}[Hubi\v cka, 2020~\cite{Hubicka2020CS}]
	\label{thm:hub}
	The \Fraisse{} limit of the class of all finite partial orders has finite big Ramsey degrees.
\end{theorem}
It turns out that this proof method can be further adapted to metric spaces with finitely many distances, and more generally, structures universal for classes described by forbidden cycles:
\begin{theorem}[Balko--Chodounsk\'y--Hubi\v cka--Kone\v cn\'y--Ne\v set\v ril--Vena, 2021~\cite{balko2021big}]\label{thm:big_cycles}
	Let $L$ be a finite language consisting of unary and binary relation symbols only,
	and let $\str{K}$ be a countably-infinite irreducible structure. Assume that every countable
	structure $\str{A}$ has a strong completion to $\str{K}$ provided that every induced cycle in $\str{A}$
	(seen as a substructure) has a strong completion to $\str{K}$ and every irreducible substructure
	of $\str{A}$ of size at most 2 embeds into $\str{K}$. Then $\str{K}$ has finite big Ramsey degrees.
\end{theorem}
This can be seen as a first step towards an infinite generalization of Theorem~\ref{thm:hn_completions}.
It is interesting that Theorem~\ref{thm:big_cycles} does not demand that the number of forbidden cycles is finite. In particular, it also captures equivalences (interpreted as $\{N,E\}$-edge labelled graphs), ultrametric spaces with finitely many distances, $\Lambda$-ultrametric spaces~\cite{sam2}, and more. This is in sharp contrast with Theorem~\ref{thm:hn_completions} where one has to perform some gymnastics to pull these classes through. As a corollary, we get the following result for which we, at this point, do not have a general proof using Theorem~\ref{thm:hn_completions} (although in all concrete cases it works, and we conjecture that there is also such a general proof):

\begin{corollary}\label{cor:big_cycles}
	Let $L$ be a finite language consisting of unary and binary relation symbols only, let $\mathcal F$ be
	a (not necessarily finite) family of finite $L$-structures which are either of size at most 2, or their Gaifman graphs are cycles.
	Let $\mathcal K$ be a strong amalgamation class of finite irreducible $L$-structures such that an $L$-structure $\str A$ has
	a completion to $\mathcal K$ if and only if there is no $\str F\in \mathcal F$ with a homomorphism-embedding
	$\str F\to \str A$. Then $\mathcal K$ has a precompact Ramsey expansion.
\end{corollary}
In Theorem~\ref{thm:big_cycles} we require the structure to be described by forbidden induced cycles, while in Corollary~\ref{cor:big_cycles} we talk about forbidden homomorphism-embeddings. An argument analogous to the one in Proposition~\ref{prop:metric_completion} shows that this is equivalent. This is the first example of big Ramsey methods allowing us to prove new theorems about Ramsey expansions, and there is hope that it is not the last one. Let us, however, remark, that the expansions which Corollary~\ref{cor:big_cycles} provides are not in a finite language, not even in trivial cases where the actual Ramsey expansion is just by free orderings.

\medskip

In~\cite{Hubicka2020uniform}, a new proof technique for showing finiteness of big Ramsey degrees for structures in languages with relations of arity $3$ and greater was introduced, and finiteness of big Ramsey degrees of the universal $3$-uniform hypergraph was proved. Recently, this was generalized to the following:
\begin{theorem}[Braunfeld--Chodounsk{\'y}--de Rancourt--Hubi{\v{c}}ka--Kawach--Kone{\v{c}}n{\'y}, 2024 \cite{braunfeld2023big}]
	\label{thm:infinitelanguages}
	Let $L$ be a countable language such that for every $n>1$ there are only finitely many relations of arity $n$, and let $\mathcal K$ be the class of all
	finite $L$-structures where all tuples in all relations are injective.  Then the \Fraisse{} limit of $\mathcal K$ has finite big Ramsey degrees.
\end{theorem}

Comparing Theorem~\ref{thm:infinitelanguages} to the unrestricted form of the Ne\v set\v
ril--R\"odl theorem (Theorem~\ref{thm:unNR}), Theorem~\ref{thm:zucker} to the
restricted from of the Ne\v set\v ril--R\"odl theorem (Theorem~\ref{thm:NR}), and
Theorem~\ref{thm:hub} to Theorem~\ref{thm:posets} shows that it is realistic to
follow known results about Ramsey classes and try to develop the corresponding results about big Ramsey degrees.

One could say that ``the 2020s of big Ramsey degrees are the 1980s of Ramsey classes''.
Can we close this gap?
Since the partite construction does not naturally generalize to an infinitary form
(due to an essential use of the backward induction),
proving big Ramsey results requires development of completely new techniques (\eg~\cite{Balko2023Sucessor}).
Compared to ones for Ramsey classes, the current proofs giving
bounds on big Ramsey degrees are typically a lot more complicated, but lead to
a better understanding of
the nature of the problem.  It is also for this reason that we believe that there is a high
chance to find more new Ramsey classes using techniques
developed for big Ramsey degrees, such as in Corollary~\ref{cor:big_cycles}.

\subsubsection{Tree of types}
\label{sec:types}
Big Ramsey degrees often exhibit behaviour that is not seen when studying Ramsey classes
and which can at first seem counter-intuitive and surprising.
A lot of it can, however, be explained by the fact that
colourings of a countable structure can be obtained from an enumeration of its vertices,
generalizing the idea of Sierpi\'nski. To explain it, we first introduce the model-theoretic
notion of tree of 1-types.

For us, a \emph{tree} is a (possibly empty) partially ordered set $(T, <_T)$ such
that, for every $t \in T$, the set $\{s \in T : s <_T t \}$ is finite and linearly ordered by $<_T$.
(Note that we do not assume the existence of a unique minimal element--\emph{root}.)

An \emph{enumerated structure} is simply a structure $\str{A}$ with underlying set $A = \vert A\vert  = \{0,1,\ldots \vert A\vert-1\}$. Fix  a countably infinite enumerated structure $\str{A}$.  Given vertices $u,v$ and an integer $n$ satisfying $\min(u,v)\geq n\geq 0$, we write $u\sim^\str{A}_n v$ and say that \emph{$u$ and $v$ have the same (quantifier-free) type over $\{0,1,\ldots,n-1\}$}, if the structures induced by $\str{A}$ on $\{0,1,\ldots, n-1,u\}$ and $\{0,1,\ldots, n-1,v\}$ are isomorphic via the map which is the identity on $\{0,\ldots ,n-1\}$ and sends $u$ to $v$. We write $[u]^\str{A}_n$ for the $\sim^\str{A}_n$-equivalence class of vertex $u$.
\begin{definition}[Tree of 1-types]
	Let $\str{A}$ be a countably infinite (relational) enumerated structure. Given $n< \omega$, write $\mathbb {T}_\str{A}(n) = \omega/\sim^{\str{A}}_n$. A (quantifier-free) \emph{1-type} is any member of the disjoint union $\mathbb {T}_\str{A}\coloneq\bigsqcup_{n<\omega} \mathbb {T}_\str{A}(n)$. We turn $\mathbb {T}_\str{A}$ into a tree as follows. Given $x\in \mathbb {T}_\str{A}(m)$ and $y\in \mathbb {T}_\str{A}(n)$, we declare
	that $x\leq^{\mathbb T}_{\str{A}} y$ if and only if $m\leq n$ and $x\supseteq y$.
\end{definition}

The tree of 1-types of a given structure is useful for giving upper bounds on big Ramsey degrees, since all the existing proofs are in fact structured
as Ramsey-type results about colouring subtrees of the tree of 1-types (or a related object). In the other direction, they are
also useful for constructing unavoidable colourings and for characterizing big Ramsey degrees.

It is interesting to notice that while seeking a Ramsey expansion of a given
class is usually a process that needs a good intuition or a good guess, determining
characterizations of big Ramsey degrees can be done more systematically by
analysing subtrees of the tree of types and comparing them with the upper bound theorems.
This method was implicit in the early characterizations of big Ramsey
degrees~\cite{Laflamme2006,coulson2022SDAP,Balko2021exact} and is explicitly
applied in~\cite{Balko2023}. In fact, one can also use it for discovering
canonical Ramsey expansions of amalgamation classes. As a proof of concept,
Hubička used this method to rediscover that the canonical Ramsey expansion
of triangle-free graphs are free orderings, while the canonical Ramsey expansion
of posets are linear extensions~\cite{Hubicka2020CS}.

\subsubsection{Big Ramsey structures}
As we have shown, even for rather simple structures such as the order of the rationals, the natural formulation of the structural Ramsey theorem fails
and one needs to give ``weaker'' results about big Ramsey degrees.  This is not different from the situation with amalgamation classes, where
most of them are not Ramsey for easy reasons.  Can we introduce a notion of a big Ramsey expansion?
Motivated by the KPT-cor\-re\-spon\-dence, this question was considered by Zucker~\cite{zucker2017} who defined the notion of a big Ramsey structure
as an expansion which precisely codes the big Ramsey degrees. (This notion was originally formulated for homogeneous structures only and generalized in~\cite{ACDMP}; compare it also with the earlier notion of a canonical partition~\cite{Laflamme2006}).
\begin{definition}
	\label{def:BRS}
	Let $\str{K}$ be a countable structure and let $\str{K}^*$ be an expansion of $\str K$.  We call $\str{K}^*$ a \emph{big Ramsey structure} for $\str{K}$ if the following holds:
	\begin{enumerate}
		\item The colouring of $\str{K}$ given by $\str{K}^*$ is \emph{unavoidable}. Namely,
		      for every finite substructure $\str{A}^*$ of $\str{K}^*$ and every embedding $f\colon\str{K}\to \str{K}$ it holds that there is an embedding $e\colon\str{A}^*\to \str{K}^*$ such that $e[A]\subseteq f[K]$.
		\item For every finite substructure $\str{A}$ of $\str{K}$, the number of its mutually non-isomorphic expansions in $\str{K}^*$ is finite and equal to the big Ramsey degree of $\str{A}$ in $\str{K}$.
	\end{enumerate}
\end{definition}
Every big Ramsey structure $\str{K}^*$ satisfies the structural Ramsey theorem:
$$(\forall \str{A}^*\in \Age(\str{K}^*))\str{K}^*\longrightarrow(\str{K}^*)^{\str{A}^*}_2,$$
and thus $\Age(\str{K}^*)$ is a Ramsey expansion of $\Age(\str{K})$, and by the second part of Definition~\ref{def:BRS}, it is also precompact.  It is thus natural to ask.
\begin{enumerate}[label=Q\arabic*,resume]
	\item \label{Q8} Given a countable structure $\str{K}$ with finite big Ramsey degrees, does there exist a big Ramsey structure $\str{K}^*$ for $\str{K}$?
\end{enumerate}
While from the existence of a reasonable precompact Ramsey expansion it follows that there exists a reasonable precompact Ramsey expansion with the expansion property, the corresponding statement for
big Ramsey structures is open. At the moment there are still only relatively few explicit examples of homogeneous structures admitting a big Ramsey structure. These are the most important ones:
\begin{enumerate}
	\item $(\mathbb Q, {<})$ (this is implicit in every proof of Theorem~\ref{thm:devlin}).
	\item The Rado graph, and more generally, unconstrained structures in binary languages (this is an easy consequence of the characterization of big Ramsey degrees by Laflamme, Sauer, and Vuksanovic~\cite{Laflamme2006}).
	\item The generic partial order~\cite{Balko2023}.
	\item \Fraisse{} limits of free amalgamation classes of structures in a finite relational language with finitely many forbidden irreducible substructures~\cite{Balko2021exact}.
\end{enumerate}
Note that we did not explicitly describe the big Ramsey structures here. The reason is that, in essence, they exactly capture special self-embeddings of the trees of 1-types containing as few isomorphism types of subtrees as possible, and hence it tends to take a bit of work to describe them. See \eg{} the recent survey on big Ramsey structures by Zucker and the first author for concrete examples~\cite{hubicka2024survey}.

\subsection{Generalizations of the KPT-cor\-re\-spon\-dence and Ramsey classes}
Besides classes of structures with embeddings, the notion of the Ramsey property and the KPT-cor\-re\-spon\-dence have been successfully adapted to other settings and, thanks to this, extreme amenability of many other groups has been proved.

One example are metric structures where the KPT-cor\-re\-spon\-dence was proved by Melleray and Tsankov~\cite{melleray2014extremely} using the framework of continuous logic by Ben Yaacov \cite{BenYaacov2010}. This makes it possible to show extreme amenability
of groups by proving an approximate form of the Ramsey property. Successful results in this direction are by
Eagle, Farah, Hart, Kadets, Kalashnyk, and Lupini on $C^*$-algebras~\cite{eagle2016fraisse},
Ferenczi, Lop\'ez-Abad,  Mbombo, and Todor\v cevi\' c, showing the Ramsey property of $L_p$-spaces~\cite{ferenczi2020amalgamation},
Barto{\v{s}}ov{\'a}, Lop\'ez-Abad, Lupini, and Mbombo on Banach spaces and Choquet simplices \cite{bartovsova2021ramsey},
operator spaces and noncommutative Choquet simplices~\cite{bartovsova2021ramseyb} and Grassmannians Over $\mathbb R$, $\mathbb C$~\cite{bartovsova2022ramsey}, and
by Kawach and Lop\'ez-Abad on Fréchet spaces~\cite{kawach2022fraisse}. The notion of big Ramsey degrees has also recently been adapted to the metric setting by Bice, de Rancourt, Hubi{\v c}ka, and Kone{\v c}n{\' y}~\cite{Bice2023} in order to generalize results on oscillation stability of the Urysohn sphere~\cite{lopez2008oscillation,NVT2009b,The2010}.

Irwin and Solecki introduced projective \Fraisse{} limits~\cite{irwin2006}. The Ramsey property for the projective amalgamation class used for constructing the Lelek fan was given by Kwiatkowska and Barto{\v{s}}ov{\'a}~\cite{bartovsova2019universal}, which also initiated some very interesting developments on generalizations of the Gowers theorem~\cite{gowers1992lipschitz,bartovsova2017gowers,lupini2017gowers,lupini2019actions,solecki2019monoid} (even though generalizing the Gowers theorem turned out not to be necessary for the results of~\cite{bartovsova2019universal} in the end). In 2022, Iyer found a Ramsey expansion of a projective amalgamation class used for constructing the universal Knaster continuum~\cite{Iyer2022}.

A categorical approach to \Fraisse{} limits was developed by Droste and Gübel~\cite{Droste1993b}, and later further extended by Kubi\'s~\cite{Kubis2014}, leading to a notion of weak Ramsey property and an extension of KPT~\cite{bartovs2021weak}.
See also related results by Mašulović~\cite{mavsulovic2021kechris}.

In 2011, Moore introduced~\cite{moore2013amenability} the \emph{convex Ramsey property} which, informally, generalizes the Ramsey property by asking for a probability distribution on embeddings from $\str B$ to $\str C$ such that the expected colours of embeddings of $\str A$ do not differ much, and proved that it is equivalent to \emph{amenability} of the automorphism group of the \Fraisse{} limit.

The KPT-cor\-re\-spon\-dence is also being extended in an abstract model-theoretic setting~\cite{hrushovski2019,krupinski2019amenability,krupinski2022ramsey,krupinski2023topological}.

\section{Open problems}\label{sec:open}

\subsection{Ramsey classes}
Besides Question~\ref{q:ramsey_finite}, the following has been conjectured:
\begin{conjecture}[Conjecture~1.7 from~\cite{Hubicka2018metricEPPA}]\label{conj:eppa_ramsey}
	Every strong amalgamation class with EPPA has a precompact Ramsey expansion.
\end{conjecture}
This is closely related to a question of Ivanov~\cite{Ivanov2015} asking whether every structure with an amenable automorphism group has a precompact Ramsey expansion (see also Question~7.3 of~\cite{Evans2}), as EPPA implies amenability of the automorphism group~\cite{Kechris2007}. The class of all finite groups is a natural candidate for a counterexample as it has EPPA but the existence of a Ramsey expansion is open (see Section~\ref{sec:groups_ramsey}). If this conjecture turns out to be false due to groups, it is very natural to restrict it to $\omega$-categorical structures only.

\medskip

Despite the rather impressive progress on Nešetřil's classification programme, there are still several concrete classes whose Ramsey expansions are not yet understood. In this section we shall overview some of them.

\subsubsection{$H_4$-free 3-hypertournaments}\label{sec:hypertournaments}
Answering a demand for non-free strong amalgamation classes in languages with a relation of arity at least 3, Cherlin provided some classes of $n$-hyper\-tour\-naments~\cite{Cherlinhypertournaments}, where an \emph{$n$-hypertournament} is a structure with one injective $n$-ary relation such that the automorphism group of every substructure on $n$ vertices is the alternating group $\Alt(n)$. For $n=2$ we get tournaments.

At the first glance it may seem even more natural to demand that for every set of $n$ vertices, exactly one enumeration of the set is in the relation. However, the $\Alt(n)$-variant turns out to be combinatorially much more convenient because one always has only two choices of a relation on a set of $n$ vertices. That is, given distinct vertices $v_1,\ldots,v_n$ it always holds that exactly one of $(v_1,v_2,v_3,\ldots,v_n)$ and $(v_2, v_1,v_3,\ldots,v_n)$ is in the relation, and this choice determines all the other relations on $\{v_1,\ldots,v_n\}$. This means that there is a 1-to-1 correspondence between linearly ordered $n$-hy\-per\-tour\-na\-ments and linearly ordered $n$-uniform hypergraphs: A hyperedge on $v_1 < \cdots < v_n$ means that $(v_1,v_2,\ldots,v_n)$ is in the hypertournament relation. This correspondence is often useful for drawing concrete hypertournaments.

In~\cite{cherlin2021ramsey}, Cherlin, Hubička, Konečný, and Nešetřil discuss Ramsey expansions of 3-hy\-per\-tour\-na\-ments from Cherlin's list. For $n=3$, the hypertournament relation picks, for every triple, one of its two possible cyclic orientations. There are three isomorphism types of $3$-hypertournaments on four vertices: $C_4$, where the cyclic orientations of all four triples agree with one global cyclic orientation of the four vertices, $O_4$ which one gets from $C_4$ by reversing the orientation of one (arbitrary) triple, and $H_4$, which is the unique homogeneous 3-hypertournament on four vertices; its automorphism group is $\Alt(4)$. Equivalently, if we have four vertices $v_1,v_2,v_3,v_4$ ordered cyclically in this order, in $H_4$ either $v_1v_2v_3$ and $v_1v_3v_4$ agree with this orientation, or $v_1v_2v_4$ and $v_2v_3v_4$ agree with this orientation.

One can then ask which subsets of $\{C_4,O_4,H_4\}$ one can forbid in order to obtain a \Fraisse{} class. A simple Ramsey argument shows that every sufficiently large 3-hyper\-tour\-na\-ment contains a copy of $C_4$; and all the four remaining subsets of $\{C_4,O_4,H_4\}$ give rise to a \Fraisse{} class. Particularly interesting are two of them: The $O_4$-free (or \emph{even}) 3-hy\-per\-tour\-na\-ments are, when ordered, first-order interdefinable with ordered two-graphs (see Example~\ref{ex:twographs}). More precisely, restricting the 1-to-1 correspondence between linearly ordered 3-hypertournaments and linearly ordered 3-uniform hypergraphs to the $O_4$-free case, one gets the class of linearly ordered two-graphs. Consequently, their canonical Ramsey expansion also adds a linear order and a graph from the switching class.

The other particularly interesting case are the $H_4$-free 3-hy\-per\-tour\-na\-ments. They are interesting for multiple reasons. For example, Miguel-Gómez proved that they are strictly $\mathrm{NSOP}_4$~\cite{MiguelGomez2024}. One can also analogously define $H_{n+1}$-free $n$-hypertournaments for every $n\geq 2$ and these always form a \Fraisse{} class. In the $n=2$ case, $H_3$ is the cyclic tournament, and so $H_3$-free tournaments are simply linear orders. Therefore, $H_4$-free 3-hy\-per\-tour\-na\-ments can be seen as some higher-order linear orders. And, indeed, they have been discovered independently in this setting by Bergfalk~\cite{Bergfalk2021} in his study of higher-order Todor\v{c}evi\'{c} walks. However, the main reason why the $H_4$-free 3-hypertournaments are interesting for structural Ramsey theory is that they are the only class of 3-hypertournaments from Cherlin's list for which we have been unable to find a Ramsey expansion in~\cite{cherlin2021ramsey}:

\begin{question}[Question~1 in~\cite{cherlin2021ramsey}]\label{q:h4free}
	Does the class of all finite $H_4$-free 3-hy\-per\-tour\-na\-ments have a precompact Ramsey expansion? If yes, find one with the expansion property (is it in a finite language?). If not, can one still find a non-trivial Ramsey expansion which would be in some sense \textit{optimal} (cf.~\cite{Evans2})? What about $H_{n+1}$-free $n$-hypertour\-na\-ments for $n\geq 4$?
\end{question}

\medskip

Both authors have spent a non-trivial amount of time trying to solve Question~\ref{q:h4free} with no success. So, perhaps, $H_4$-free 3-hy\-per\-tour\-na\-ments might turn out to be an important example which may answer Question~\ref{q:ramsey_finite} negatively.

We have also seen that Theorem~\ref{thm:hn_completions} is a very powerful tool for answering Question~\ref{q:ramsey_finite}  positively in concrete instances; in essentially all test cases that Theorem~\ref{thm:hn_completions} faced, it came out victorious. But it might be because of at least two reasons: Either Theorem~\ref{thm:hn_completions} is very strong, or our test cases are very weak.

The issue with resolving this dilemma is that it is hard to come up with novel homogeneous structures which do not mimic old homogeneous structures for which we already know that Theorem~\ref{thm:hn_completions} works. One thus needs to turn to the classification programme of homogeneous structures to produce examples, but as classification is very hard and work-intensive, it does not produce new examples at the same rate as one is able to process them. The other source of interesting examples are the ad hoc homogeneous structures discovered as a side-product of some other research (for example, the study of infinite CSPs seems to produce those from time to time). In any case, very few challenging examples have been appearing recently, and thus one needs to cherish those which do appear.

The class of all $H_4$-free 3-hy\-per\-tour\-na\-ments is, at this point, at the very top of the list of these challenging examples. As opposed to other, longer-standing open concrete classes, this one is relatively recent (and thus has seen much fewer person-hours spent trying to solve it), it is a strong amalgamation class in a finite relational language and has an extremely simple description. See~\cite[Chapter~2]{Konecny2023phd} for a discussion of possible ways how things might resolve based on the answer to Question~\ref{q:h4free}.

While~\cite{cherlin2021ramsey} does not provide a Ramsey expansion of the $H_4$-free 3-hy\-per\-tour\-na\-ments, it is proved there that the class of all finite linearly ordered $H_4$-free 3-hy\-per\-tour\-na\-ments is not a locally finite subclass of the class of all finite linearly ordered structures with one ternary relation which suggests that if one wants to use Theorem~\ref{thm:hn_completions}, one should probably look for a richer expansion.

\subsubsection{Graphs and hypergraphs of large girth}
It is a famous result of Erd\H os that there are graphs of arbitrarily large chromatic number and girth.~\cite{erdos1959graph} (Here, the \emph{girth} of a graph $\str G$ is the length of the shortest cycle of $\str G$.) This has started a whole area rich with interesting and difficult results; in fact,
there is a survey of it in this very volume by Reiher~\cite{reiher2024graphs} to which we refer the reader for details and a historical overview.
\marginpar{MK: TODO, na arXIvovou verzi smazat poznámku o volumu.}

Given $g\geq 3$, let $\mathcal C_g$ be the class of all finite graphs of girth at least $g$. The result of Erd\H os is equivalent to all classes $\mathcal C_g$ being vertex-Ramsey. Note that $\mathcal C_3$ is the class of all finite graphs and $\mathcal C_4$ is the class of all finite triangle-free graphs. It turns out that for $g\geq 5$, $\mathcal C_g$ does not have the amalgamation property: Let $\str B_1$ be the graph $K_{1,3}$, that is, the graph on vertices $\{x_1,x_2,x_3,y\}$ such that $x_1y$, $x_2y$, and $x_3y$ are the only edges of $\str B_1$, and let $\str B_2$ be the graph $K_1+K_{1,2}$, that is, the graph on vertices $\{x_1,x_2,x_3,y'\}$ such that $x_1y'$ and $x_2y'$ are the only edges of $\str B_2$. There is no amalgam of $\str B_1$ and $\str B_2$ over $\{x_1,x_2,x_3\}$, as $x_1y'x_2y$ is a 4-cycle and one cannot identify $y$ with $y'$ since they are connected differently to $x_3$. In fact, F{\"{u}}redi and Komj{\'a}th proved that there is no countable graph of girth 5 universal for all countable graphs of girth 5~\cite{Furedi1997b}.

In general, given a graph $\str G\in \mathcal C_g$ and vertices $x,y\in \str G$ such that $x$ and $y$ are not connected by an edge, for every $d\leq \frac{g-1}{2}$ there is at most one path of length $d$ in $\str G$ from $x$ to $y$. In order to recover the amalgamation property, we can expand the language by a binary function $\func{}{}$ and consider the expansion $\mathcal C_g^+$ of $\mathcal C_g$ where if $\str G\in \mathcal C_g^+$ and $x,y\in G$ are not connected by an edge then $\func{}{}(x,y)$ is the set of all vertices of $\str G\setminus\{x,y\}$ which lie on a path from $x$ to $y$ of length at most $\frac{g-1}{2}$ (note that this always consists of vertices of at most one path). It is easy to verify that $\mathcal C_g^+$ is an amalgamation class (one simply takes the free amalgam of graphs and defines the functions according to it).

The following question is, however, wide open (and even a resolution for $g=5$ would be a major result):
\begin{question}\label{q:girth_graphs}
	For which $g\geq 5$ does $\mathcal C_g^+$ have a precompact Ramsey expansion?
\end{question}
Let us remark that in a recent breakthrough, Reiher and Rödl~\cite{reiher2023girth} apply many variants of complicated partite constructions and prove that free orderings of $\mathcal G_g$ have the Ramsey property for colouring edges for every $g$, extending an earlier result of Nešetřil and Rödl who proved this for $g\leq 8$~\cite{Nevsetvril1987}. Note that one does not need the functional expansion for this result.

In fact, both~\cite{Nevsetvril1987} and~\cite{reiher2023girth} handle not only edge-colourings of graphs with large (or not-so-large) girth, but general edge colourings of linear hypergraphs of large girth. One can analogously define functional expansions of these classes which have the amalgamation property and ask an analogue of Question~\ref{q:girth_graphs}.

Recall Theorem~\ref{thm:sparseningRamsey} and note that it only promises a homomorphism-embedding to a tree amalgam of copies of $\str B$, not an embedding. Consequently, if one applies Theorem~\ref{thm:sparseningRamsey} for some $\str B\in \mathcal C_g^+$, the produced Ramsey witness $\str C$ will likely contain many graph 4-cycles with no functions defined on them (these are produced by the partite lemma, see Appendix~\ref{appendix:partite} for details) which have a (non-injective) homomorphism-embedding to an edge. A stronger version of Theorem~\ref{thm:sparseningRamsey} which promises an embedding to a tree amalgam of copies of $\str B$ (which seems completely out of reach of the current techniques) would, in particular, imply a solution of Question~\ref{q:girth_graphs}.

\subsubsection{Euclidean metric spaces}
A metric space $\str A$ is \emph{Euclidean} if there exists $n$ and a set of points $S\subseteq \mathbb R^n$ such that $\str A$ is isomorphic to $S$ equipped with the Euclidean metric. It is \emph{affinely independent} if $S$ can be chosen to be affinely independent. A variant of following question was already posed on page~150 of~\cite{Kechris2005} and then elaborated on by Nguyen Van Th\'e, 2010~\cite{The2010}:
\begin{question}[\cite{Kechris2005},\cite{The2010}]\label{q:euclid}
	Does the class of all finite affinely independent Euclidean metric spaces have a precompact Ramsey expansion?
\end{question}
(Technically, it was only asked if it is Ramsey with free linear orderings, but this is the natural, modern formulation.) Besides this, Nguyen Van Th\'e asks some other closely related questions, and his paper~\cite{The2010} is also a great reference for the broader context around Question~\ref{q:euclid}. Affinely independent Euclidean metric spaces are very interesting as they form a strong amalgamation class with a complicated membership condition.

\subsubsection{Finite measure algebras and equipartitions}
Giordano and Pestov~\cite{Giordano2007} proved that the automorphism group of every standard probability space is extremely amenable. In~\cite{Kechris2017}, Kechris, Soki\'c, and Todor{\v c}evi{\'c} identified a structural Ramsey statement which would imply this result: A \emph{finite dyadic measure algebra} is a finite structure $\str A = (A, \wedge, \vee, 0, 1, \mu)$, where $(A,\wedge,\vee,0,1)$ is a finite Boolean algebra and $\mu \colon A\to [0,1]$ is a positive measure with values in the dyadic rationals. The structural Ramsey statement identified in~\cite{Kechris2017} says that the class of all \emph{naturally ordered} finite dyadic measure algebras is Ramsey (see~\cite{Kechris2017} for a definition of a natural ordering). This is, in turn, equivalent to the Ramsey property for finite Boolean algebras with antilexicographic orderings (see Example~\ref{ex:ba}) where one only considers maps where every atom is sent to the join of the same number of atoms, which can also be seen as a version of the dual Ramsey theorem for \emph{equipartitions} (see~\cite{Kechris2017}; they call it the \emph{homogeneous dual Ramsey theorem}). To our best knowledge, all these Ramsey-theoretic statements are still open.

\subsubsection{Finite groups}\label{sec:groups_ramsey}
The class of all finite groups is an amalgamation class whose \Fraisse{} limit is called Hall's universal locally finite group (see \eg{}~\cite{Siniora2} for more details). In~\cite{Siniora2} it is proved that this class has EPPA (see Section~\ref{sec:eppa}), and so it is a natural test case for Conjecture~\ref{conj:eppa_ramsey} (another such example is the class of all finite skew-symmetric bilinear forms~\cite{Cherlin2003}). Note that the language contains functions, and thus one only wants to colour closed structures. In particular, the analogue of colouring vertices would be colouring cyclic subgroups of a given order.

\begin{question}\label{q:groups_ramsey}
	Does the class of all finite groups have a precompact Ramsey expansion?
\end{question}
Note that while there are some partition theorems for finite abelian groups~\cite{voigt1980partition}, they do not seem to be structural, and thus not relevant for this question. We weakly conjecture that the answer is negative as we expect Hall's universal locally finite group to have a complex structure of imaginaries. But even if the answer happens to be positive, it feels even harder to prove it using Theorem~\ref{thm:hn_completions} than, for example, Boolean algebras (see Example~\ref{ex:ba}).

\bigskip

Antilexicographically ordered finite Boolean algebras is one of a few known Ramsey classes for which no proof is known using Theorem~\ref{thm:hn_completions}. In Section~\ref{exceptions} we have seen this example. The following question is an attempt to formalize the question whether proving their Ramseyness using Theorem~\ref{thm:hn_completions} is at all possible:
\begin{question}\label{q:boolean_allthose}
	Is there a language $L$ containing $<$, and a homogeneous $L$-structure $\str H$ which is first-order bi-interpretable with the \Fraisse{} limit of $\BA^+$, or the class of all finite convexly ordered $C$-relations, or the class of all finite meet-semilattices with a linear extension, such that $\age(\str H)$ is a locally finite subclass of the class of all finite linearly ordered $L$-structures?
\end{question}
This is not the exact question that one would want to ask (for example, if for some reason one needed to use Theorem~\ref{thm:sparseningRamsey} instead of Theorem~\ref{thm:hn_completions}, it would still count), but it is hard to formalize what it means to use some theorem to prove a statement known to be true.

Note that if a class $\K$ of finite $L$-structures is a locally finite subclass of the class of all finite $L$-structures then this gives us a lot of understanding of the homomorphism-embedding variant of the constraint satisfaction problem over the \Fraisse{} limit of $\K$. It turns out that understanding completions often leads to a polynomial-time algorithm. In fact, we do not know any example of a Ramsey class where this would not be the case (but it is possible that this is due to a lack of serious effort):
\begin{conjecture}
	Every homogeneous Ramsey structure $\str M$ in a finite relational language is first-order bi-interpretable with a homogeneous Ramsey structure $\str H$ in a finite relational language such that the homomorphism-embedding variant of the constraint satisfaction problem of $\str H$ is solvable in polynomial time.
\end{conjecture}
Let us remark that the conjecture is true for the standard, homomorphism variant of CSP without even using Ramseyness: If $\str M$ has no constant tuples in non-unary relations, then we can simply create $\str H$ from $\str M$ by adding all constant tuples to all non-unary relations, making its CSP trivial. If $\str M$ has some constant tuples then one has to introduce extra unary relations recording this information. This, clearly unsatisfactory, construction was suggested to us by Michael Pinsker. However, its existence indicates that one should refine the conjecture to avoid such constructions, for example to the following question (also suggested to us by Pinsker, see \eg{}~\cite{BodirskyBook} for a definition of a model-complete core):
\begin{question}
	Let $\str M$ be a model-complete core in a finite relational language whose age is Ramsey. Is $\str M$ first-order bi-interpretable with a model-complete core $\str H$ in a finite relational language whose age is Ramsey such that the constraint satisfaction problem of $\str H$ is solvable in polynomial time?
\end{question}

\subsection{EPPA}\label{sec:eppa_questions}

In contrast with the Ramsey property, we are not aware of any general questions or conjectures about EPPA or the existence of EPPA expansions. It is easy to see that finite linear orders, for example, do not have EPPA (since they are rigid, but do have non-trivial partial automorphisms). This generalizes to structures where one can locally find arbitrarily long linear orders (formally, structures with the \emph{strict order property}, see~\cite{Siniora2}).

Evans, Hubička, and Nešetřil~\cite{Evans2} showed that the class $\mathcal C_F$ from Section~\ref{sec:orientations} does not have EPPA and that $\mathcal C_F^+$ is its minimal expansion which might have EPPA, and Hubička, Konečný, and Nešetřil later proved that it indeed does have EPPA~\cite{Hubicka2018EPPA}. Nevertheless, EPPA is still wide open for some simple classes of structures:

\begin{question}[Herwig--Lascar~\cite{herwig2000}]\label{q:eppa_tournaments}
	Does the class of all finite tournaments have (coherent) EPPA?
\end{question}
Originally, Herwig and Lascar only ask for EPPA, asking for coherent EPPA is a natural extension if the original question is answered in the affirmative. They proved that this question is in fact equivalent to a problem in profinite group theory (see Section~6 of~\cite{herwig2000}). The reason why the class of all finite tournaments has stood up to all attacks so far is the following: Most of the EPPA constructions construct EPPA-witnesses where a lot of pairs of vertices are not in any relations together (and these witnesses are then, potentially, completed in an automorphism-preserving way). This in particular means that the EPPA-witnesses contain a lot of symmetric pairs. Completing to a tournament means that we need to choose an orientation for each such pair, thereby very likely killing many automorphisms. In fact, if $G$ is the automorphism group of a tournament then $G$ has odd order. For this reason, Question~\ref{q:eppa_tournaments} has a very strong group-theoretic flavour.

Note that Huang, Pawliuk, Sabok, and Wise~\cite{Sabok} disproved EPPA for a certain version of hypertournaments (different from the one from Section~\ref{sec:hypertournaments}) defined specifically to obtain a variant of the profinite-topology equivalence for which the topological statement is false. This naturally motivates the following question.

\begin{question}\label{q:eppa_hypertournaments}
	Does the class of all finite $n$-hy\-per\-tour\-na\-ments (see Section~\ref{sec:hypertournaments} for definitions) have (coherent) EPPA? What about the $H_4$-free 3-hypertour\-na\-ments or the $O_4$-free 3-hypertournaments?
\end{question}
Note that the $\{H_4,O_4\}$-free 3-hy\-per\-tour\-na\-ments do not have EPPA as these structures become just linear orders after naming a least vertex. We expect that the $O_4$-free 3-hypertournaments will behave as a combination of two-graphs (see~\cite{eppatwographs} and Example~\ref{ex:twographs}) and 3-hypertournaments with no restrictions.

\medskip

Besides tournaments, EPPA is also open for the class of all finite oriented graphs with no independent set of size $k$, $k\geq 3$. (Note that for $k=2$ we simply obtain tournaments.) This motivated the following question:

\begin{question}[Hubička--Jahel--Konečný--Sabok, 2024~\cite{HubickaSemigeneric}]
	For which $k\geq 2$ is there a finite $\ell$ such that for every finite oriented graph $\str A$ which contains no independent set of size $k$ there is a finite oriented graph $\str B$ which contains no independent set of size $\ell$ such that $\str A\subseteq \str B$ and every partial automorphism of $\str A$ extends to an automorphism of $\str B$?
\end{question}
For $k=2$ we are looking at tournaments and this question, in contrast with Question~\ref{q:eppa_tournaments}, actually allows the EPPA-witnesses to have some even-order automorphisms. Nevertheless, the existing constructions are not strong enough as the independent sets they produce grow with the number of vertices of $\str A$.

\medskip

Two-graphs (see Example~\ref{ex:twographs}) turned out to be a very interesting example for EPPA \cite{eppatwographs}. They are one of the reducts (in the sense of Section~\ref{sec:reducts}) of the random graph, and one can ask about EPPA of the other reducts:
\begin{question}[\cite{eppatwographs}]
	Which of the reducts of the random graph have (coherent) EPPA?
\end{question}
There are five reducts. The random graph itself and the set with no structure have been solved. For two-graphs, coherence is open. The last two reducts are the \emph{complementing graph} (adding an isomorphism between the random graph and its complement) and the ``complementing two-graph'' (the join of the automorphism group of the generic two-graph and the complementing graph) and for them both EPPA and coherent EPPA are open. In his Bachelor thesis, Beliayeu~\cite{Beliayeu2023bc} proved EPPA for the class of all finite graphs with loops with complementing embeddings (that is, every vertex may, but does not have to, have a loop, and loops also get complemented). This created a particularly interesting situation: It is open whether graphs with complementing embeddings have EPPA, but there are two possible expansions of graphs with complementing embeddings which have EPPA: Graphs with normal embeddings, and graphs with loops with complementing embeddings. Moreover, both of them can be seen as reducts of the countable random graph with a generic unary relation, and they are incomparable. It would be interesting to see if the meet of these two expansions are graphs with complementing embeddings, and whether the meet has EPPA.

In 1996, Thomas classified reducts of the homogeneous $k$-uniform hypergraphs for $k\geq 3$~\cite{thomas1996reducts}. The situation turns out to be similar to the random graph in that all the reducts correspond to some type of switching over some number of vertices. EPPA is completely open for all these classes (and it seems that the method for proving EPPA for two-graphs from~\cite{eppatwographs} does not generalize). We can thus ask:

\begin{question}
	Which reducts of the $k$-uniform hypergraph, $k\geq 3$, have (coherent) EPPA?
\end{question}
An answer for any one of the reducts except for the trivial ones would be very interesting.

\medskip

A major limitation of Theorem~\ref{thm:hkn} compared to Theorem~\ref{thm:sparseningRamsey} is the requirement that all functions are unary. Naturally, before trying to generalize it in its full strength, one may want to study some concrete cases. The following two questions point out to simple classes for which EPPA is open, the reason this time being the presence of non-unary algebraicity:
\begin{question}[Question~13.2 from~\cite{Hubicka2018EPPA}]
	Let $L$ be the language consisting of a single partial binary function and let $\mathcal C$ be the class of all finite $L$-structures. Does $\mathcal C$ have EPPA?
\end{question}

\begin{question}[Question~13.3 from~\cite{Hubicka2018EPPA}]
	Does the class of all finite partial Steiner triple systems have EPPA, where one only wants to extend partial automorphism between closed substructures? (A sub-hypergraph $H$ of a partial Steiner triple system $S$ is \emph{closed} if whenever $\{x,y,z\}$ is a triple of $S$ and $x,y\in H$, then $z\in H$.)
\end{question}

\medskip

Recall that the Classification Programme of EPPA Classes asks, in the presence of EPPA, whether the given class also has coherent EPPA, APA, and ample generics. The following two questions go in this direction:

\begin{question}[Evans--Hubička--Konečný--Nešetřil, 2020~\cite{eppatwographs}]
	Does the class of all finite two-graphs have coherent EPPA?
\end{question}
In~\cite{eppatwographs} it is proved that this class has EPPA. In fact, the automorphism group of the generic two-graph does have a dense locally-finite subgroup (a dynamical consequence of coherent EPPA, see Section~5 of~\cite{eppatwographs}), but the EPPA-witnesses from~\cite{eppatwographs} are not coherent. Two-graphs are thus a concrete candidate for an answer to a question of Siniora whether there is a class of finite structures with EPPA but not coherent EPPA~\cite[Question~3]{Siniora2}.

\begin{question}[Question~6.4 of~\cite{HubickaSemigeneric}]
	Does the class of all finite $n$-partite tournaments for some $2\leq n \leq \omega$, or the class of all finite semigeneric tournaments have coherent EPPA?
\end{question}
These classes do have EPPA by~\cite{HubickaSemigeneric}, and the existence of a dense locally finite subgroup of the automorphism groups of their \Fraisse{} limits is open. Note also that semigeneric tournaments do not have APA, but they do have ample generics, see~\cite{HubickaSemigeneric}.

\subsection{Big Ramsey structures}
In contrast to Ramsey classes and EPPA, there is still too much mystery regarding
big Ramsey degrees and structures and so there are no suggested dividing lines for the
existence of big Ramsey structures or finiteness of big Ramsey degrees. It seems
that if there is an infinitely-branching node above every node of the tree of 1-types
then the big Ramsey degree of some finite substructure is infinite
(see Section~7.1 of~\cite{braunfeld2023big}). However, there are also structures
with infinite big Ramsey degrees which do not have this property such as the
\Fraisse{} limit of all directed graphs omitting an infinite set of tournaments
found by Sauer~\cite{sauer2003canonical}.

Therefore, before one tries to identify meaningful conjectures, more concrete basic examples should be looked at and solved. In fact, the existence of a big Ramsey structure (and the exact characterization of big Ramsey degrees) has not been proved for any structure in a language with a relation of arity greater than 2. In particular, we have the following problem:

\begin{problem}[Problem~7.2 of~\cite{braunfeld2023big}, Problem~6.3.1 of~\cite{Balko2021exact}]\label{prob:3unif}
Characterize the big Ramsey degrees of the generic $3$-uniform hypergraph. Does the generic $3$-uniform hypergraph admit a big Ramsey structure?
\end{problem}
Once Problem~\ref{prob:3unif} is solved, the following problem is its natural generalization.
\begin{problem}[Problem~7.2 of~\cite{braunfeld2023big}, Problem~6.3.2 of~\cite{Balko2021exact}]
\label{Prob:LargeArities}
Given a relational language $L$ with finitely many relational symbols of each arity, characterize the big Ramsey degrees for the class $\mathcal K$ of all finite $L$-structures where tuples in each relation are injective. Does the \Fraisse{} limit of $\mathcal K$  admit a big Ramsey structure?
\end{problem}
An upper bound for these kinds of structures appeared in~\cite{braunfeld2023big}. We believe that a characterization is possible with the present tools and methodology and we expect that the result will naturally generalize the situation for the 3-uniform hypergraph.

\medskip

Curiously, even finiteness of big Ramsey degrees is open for some concrete simple free amalgamation structures in finite languages described by finitely many forbidden substructures, as soon as there is a relation of arity greater than 2. The following two are probably the most striking open problems on the way
towards an infinitary generalization of the Ne\v set\v ril--R\"odl theorem:
\begin{question}\label{q:badclique1}
	Let $\str{F}$ be a 3-uniform hypergraph on 4 vertices with all edges but one.
	Does the \Fraisse{} limit of the class of all finite $\str{F}$-free 3-uniform hypergraphs have finite big Ramsey degrees and a big Ramsey structure?
\end{question}
\begin{question}[Problem~1.4 of~\cite{typeamalg}]\label{q:badclique2}
	Let $L=\{E,H\}$ be a language with one binary relation $E$ and one ternary relation $H$. Let $\str{F}$ be the $L$-structure where $F=\{0,1,2,3\}, E_\str F=\{(1,0),(1,2),(1,3)\}, H_\str F=\{(0,2,3)\}$. Denote by $\mathcal K$ the class of all $L$-structures $\str{A}$ such that there is no monomorphism $\str{F}\to\str{A}$. Does the \Fraisse{} limit of $\mathcal K$ have finite big Ramsey degrees and a big Ramsey structure?
\end{question}

Once Questions~\ref{q:badclique1} and~\ref{q:badclique2} are answered, we hope that we will have enough understanding of big Ramsey degrees of free amalgamation classes in order to be able to prove a big Ramsey analogue of the Nešetřil--Rödl theorem:

\begin{problem}\label{prob:general}
Let $L$ be a finite relational language. Characterize those finite collections of  finite irreducible
$L$-structures $\mathcal F$ for which the \Fraisse{} limit of the class of all finite $\mathcal F$-free structures has finite big Ramsey degrees, resp. admits a big Ramsey structure.
\end{problem}

\medskip

When free amalgamation classes are understood, one can start aiming for a big Ramsey analogue of Theorems~\ref{thm:sparseningRamsey} or~\ref{thm:hn_completions} (or some other variant which can handle similar structures). In the Ramsey classes world, various variants of metric spaces turned out to be useful examples to look at first. Theorem~\ref{thm:big_cycles} implies that homogeneous metric spaces with finitely many distances have finite big Ramsey degrees; the bounds on them produced by Theorem~\ref{thm:big_cycles} are very far from optimal though. One should attempt to characterize the exact big Ramsey degrees for these metric spaces, and if possible, find the corresponding big Ramsey structures, in the hope of uncovering the core of the combinatorial nature of these structures.

\begin{problem}\label{prob:metric_diaries}
Let $S\subseteq \mathbb R^{{>}0}$ be finite and let $\mathbb U_S$ be a countable metric spaces with distances from $S$ universal for all countable metric spaces with distances from $S$. Characterize the big Ramsey degrees of $\mathbb U_S$. Does it admit a big Ramsey structure?
\end{problem}

We believe that in order to solve the above problems, one should consider a more restrictive category of embeddings of enumerated structures, which naturally arise from the current examples of big Ramsey
structures~\cite{typeamalg}.

\medskip

We have seen that the class of finite Boolean algebras with antilexicographic orderings is Ramsey and that it is an easy consequence of the dual Ramsey theorem (see Example~\ref{ex:ba}). In the infinite setting, these two statements diverge. While the Carlson--Simpson theorem~\cite{carlson1984} is a natural infinitary extension of the dual Ramsey theorem, one may also study the big Ramsey degrees of the countable atomless Boolean algebra
(\ie{}, the \Fraisse{} limit of the class of all finite Boolean algebras). Shortly before finishing this survey it was shown that the big Ramsey degrees are in fact infinite for the subalgebra with 3 atoms~\cite{Bartosova2025}, and it turns out that this is related to continuous colourings of 3-elements subsets of $\mathbb Q$ and topological copies~\cite[Theorem 6.33]{todorcevic2010introduction}. However we can still ask:

\begin{problem}[{\cite[Problem 1.4]{Bartosova2025}}]
Determine the big Ramsey degree of the 2-atom Boolean algebra in the countable atomless Boolean algebra.
\end{problem}
A similar situation has also been recently established for homogeneous countable pseudo-trees~\cite{ChEW_pseudotree,chodounsky2025big}. The following questions makes sense in both settings.
\begin{problem}[{\cite[Problem 1.5]{Bartosova2025}}]
\label{probopt}
Prove an infinite Ramsey-theoretic result for the countable atomless Boolean algebra (or the pseudotree) and
show its optimality even in the cases where the big Ramsey degrees are infinite.
\end{problem}
\begin{problem}[{\cite[Problem 1.6]{Bartosova2025}}]
Generalize the notion of a big Ramsey structure to this new setting where big Ramsey degrees are infinite,
and find its topological dynamics analogue.
\end{problem}
Another recent result~\cite{Omegalabelled2025} shows that $\omega$-edge-labeled hypergraphs do not have finite big Ramsey degrees.
However, one can consider the following reduct:
\begin{question}[{\cite[Problem 1.3]{Omegalabelled2025}}]
	Let $L$ be a relational language with a single quaternary relation $R$ and $\mathcal K$ the class of all finite $L$-structures
	where $\str{A}$ such that
	\begin{enumerate}
		\item for every $(a,b,c,d)\in R^\str{A}$ it holds that $a\neq b$, $c\neq d$ and $(c,d,a,b)\in R^\str{A}$,
		\item for every pair of distinct vertices $a,b$ of $\str{A}$ it holds that $(a,b,a,b),(a,b,b,a)\in R^\str{A}$,
		\item whenever $(a,b,c,d)$ and $(c,d,e,f)$ are in $R^\str{A}$ then also $(a,b,e,f)\in  R^\str{A}$.
	\end{enumerate}
	(In other words, $R^\str{A}$ is an equivalence relation on 2-element subsets of vertices of $A$.)
	Does the \Fraisse{} limit of $\mathcal K$ have finite big Ramsey degrees?
\end{question}

\medskip

The area of big Ramsey degrees is currently experiencing very rapid development, and so there are many other open problems in this area, see for example recent surveys~\cite{hubicka2024survey,dobrinen2021ramsey}, Section~7 of~\cite{braunfeld2023big}, Section~6 of~\cite{Balko2021exact}, or Chapter~2 of the second author's dissertation~\cite{Konecny2023phd}.

\section*{Acknowledgments}
We would like to thank Adam Bartoš, Manuel Bodirsky, Gregory Cherlin, Lionel Nguyen Van Th\'e, Vojta Rödl, Rob Sullivan, and Andy Zucker for enlightening discussions and helpful comments on early drafts of this survey. We are grateful to Jisun Baek, Hyoyoon Lee, and Jaehyeon Seo for helping us iron out some important details and for their thoroughness with which they read a preprint of this survey. We would also like to thank both anonymous referees for taking their time to carefully read this survey and for giving us many great suggestions which improved this survey significantly. Most importantly we would like to thank to Jarik Ne\v set\v ril for introducing us to this beautiful area of mathematics, being great supervisor, advisor, co-author and personal friend.

In the earlier stages of this project, J. H. and M.K. were supported by the project 21-10775S of  the  Czech  Science Foundation (GA\v CR). In the later stages, J. H. was supported by a project that has received funding from the European Research Council under the European Union's Horizon 2020 research and innovation programme (grant agreement No 810115), and M. K. was supported by a project that has received funding from the European Union (Project POCOCOP, ERC Synergy Grant 101071674).
Views and opinions expressed are however those of the authors only and do not necessarily reflect those of the
European Union or the European Research Council Executive Agency. Neither the European Union nor the granting
authority can be held responsible for them.
Since January 2025, during a (substantial) revision, J. H. was supported by the project 25-15571S of  the  Czech  Science Foundation (GA\v CR).

\appendix
\gdef\thesection{\Alph{section}}
\makeatletter
\renewcommand\@seccntformat[1]{\appendixname\ \csname the#1\endcsname.\hspace{0.5em}}
\makeatother

\section{Partite construction}
\label{appendix:partite}

The partite construction has been discovered and developed by Nešetřil and Rödl in a series of papers~\cite{Nevsetvril1976,Nevsetvril1977,Nevsetvril1979,Nevsetvril1981,Nevsetvril1982,Nevsetvril1983,Nevsetvril1984,Nevsetvril1984a,Nevsetvril1987,Nevsetvril1989,Nevsetvril1990,Nevsetvril2007}, and further extended by various authors~\cite{bhat2016ramsey,Hubicka2016,reiher2024graphs,reiher2023girth,reiher2024colouring}. It has been applied to many problems and published in
multiple variants.  The purpose of this section is to isolate and present its four variants which are used
today for obtaining structural Ramsey classes.

The main goal of this appendix is to serve as an exposition. For this reason, we choose to only state lemmas in their minimum strength necessary to prove the main result of the given section, even though sometimes we are going to need a slightly stronger form of the lemma (using more definitions) in a later section. Instead, we prove the stronger form separately. We believe that this choice makes it easier to appreciate the core ideas of partite construction for a reader who is not yet familiar with this beautiful concept.
Except for a few basic definitions, this appendix is independent from the rest of the survey.

The (primal) partite construction has two basic forms -- the \emph{non-induced} form and the \emph{induced} form. The non-induced partite construction is used
to prove the unrestricted Ne\v set\v ril--R\"odl theorem (Theorem~\ref{thm:unNR}) and is discussed in Section~\ref{anoninduced}. The induced form
builds on top of existing structural results and is then used for obtaining more restricted theorems from less restricted ones.
We formulate it in a relatively general form in Section~\ref{sec:induced}. This form allows easy proofs of the restricted form of the Ne\v set\v ril--R\"odl theorem (Theorem~\ref{thm:NR})
and a Ramsey theorem for partial orders with linear extensions (Theorem~\ref{thm:posets}).

From the induced partite construction, two other constructions can be derived.  In some cases, performing the induced partite construction once
does not solve the problem, but after performing it multiple times, magic happens and problems disappear. Intuitively, the $\langle \str{A},\str{B},\str{C} \rangle$-hypergraph (the hypergraph of copies in the Ramsey witness, see Section~\ref{sec:systematic}) becomes more and more
sparse with each iteration of the partite construction.  This \emph{iterated} form is discussed in Section~\ref{sec:iterated}
where we prove Theorem~\ref{thm:sparseningRamsey} which then implies, for example, that ordered metric spaces are Ramsey (Theorem~\ref{thm:nesetril-metric-spaces}).
The last discussed extension is done by replacing one of the two key steps of the partite construction by another partite construction. We call this the \emph{recursive} form
and discuss one application of it in Section~\ref{arecursive} to prove an extension of the unrestricted Ne\v set\v ril--R\"odl theorem for structures in languages with functions.
Combining this with the iterated partite construction then yields the main results of~\cite{Hubicka2016} (Theorems~\ref{thm:HN} and \ref{thm:hn_completions}) in their full strength.

We do not discuss variants of the
partite construction used for proving theorems about colourings of special structures only. For colourings of
vertices, see~\cite{Nevsetvril1979}, and for colourings of edges see the survey part of
a recent major result by Reiher and Rödl~\cite{reiher2023girth}.
We choose not to discuss the dual form of the partite construction either. However, what we do here
can be adapted quite naturally in the dual setting (see for example~\cite{junge2023categorical} for a presentation of the partite construction in the abstract categorical setting).

\subsection{Non-induced partite construction}
\label{anoninduced}
In this section, we will prove Theorem~\ref{thm:unNR}. We are inspired primarily by the 1989 presentation from~\cite{Nevsetvril1989}, see also a similar treatment in~\cite{Bodirsky2015}.

We will work with multiple languages. Typically, we will have a ``base'' language $L$ and then we will expand it to $L_P$ by finitely many unary relation symbols. In order to differentiate them from the unary relations in $L$, we will call the elements of $L_P\setminus L$
\emph{predicates} and denote them by letter $\pred{}{}$.
\begin{definition}[Partite System]Let $L_P$ be a language extending $L$ by finitely many predicates.
	An \emph{$L_P$-partite system} is any finite $L_P$-structure $\str{A}$ satisfying the following two conditions:
	\begin{enumerate}
		\item For every $v\in \str{A}$ there exists precisely one predicate $\pred{}{}\in L_P\setminus L$ containing $v$.
		\item For every predicate $\pred{}{}\in L_P\setminus L$ and every relation symbol $\rel{}{}\in L$, each $\vec{x}\in \rel{A}{}$ contains at most one vertex from $\pred{A}{}$. (Such $\vec{x}$ is called \emph{transversal}.)
	\end{enumerate}
\end{definition}
$L_P$-partite systems are a generalization of multipartite graphs. Note that the second condition allows a single vertex to appear multiple times in $\vec{x}$.
Given an $L_P$-partite system $\str{A}$ and a vertex $v\in A$ we denote by $\pi_\str{A}(v)$ the unique predicate $p$ containing $v$. This defines a function $\pi_\str{A}\colon A\to L_P\setminus L$ which we call \emph{projection}.
We denote by $A_{\pred{}{}}$ the set of all vertices of $\str{A}$ with projection $\pred{}{}$ and call
it the \emph{partition} $p$.

An $L_P$-partite system $\str{A}$ is \emph{transversal} if $A=L_P\setminus L$ and
$A_v=\{v\}$ for every $v\in A$ (thus the structure touches every partition by
precisely one vertex, and to simplify some notation below we also unify the names of
partitions and vertices).

Note that $L_P$-partite systems are typically not linearly ordered. In fact, even if $L$ contains a relation $<$, if $\str A$ in an $L_P$-partite system in which $<$ is interpreted as a linear order then $\str A$ is transversal, as vertices from the same partition are in no relations together. In this whole appendix, we will always construct partite systems as Ramsey objects, and in order to obtain an ordered structure (which we typically desire), we will in the very end extend the relation $<$ to a linear order.

The key, and most tricky, part of the proof is the following lemma.
\begin{lemma}[Partite Lemma]
	\label{lem:partite}
	Let $L_P$ be a relational language extending $L$ by finitely many predicates, and let $\str{A}$ and $\str{B}$ be $L_P$-partite systems such that $\str{A}$ is transversal.
	Then there exists an $L_P$-partite system $\str{C}$ such that $\str{C}\longrightarrow (\str{B})^\str{A}_2$.
\end{lemma}
It is interesting to notice that the Partite lemma is a special case of the Nešetřil--Rödl theorem (Theorem~\ref{thm:NR}), and in fact it follows easily from the unrestricted Nešetřil--Rödl theorem (Theorem~\ref{thm:unNR}) by removing all non-transversal tuples from relations. The main point is, however, that the Partite lemma has an easier proof than Theorem~\ref{thm:unNR} itself since the copies of $\str{A}$ overlap in a controlled way. In the early days, the Partite lemma was proved only for the special cases of $\vert A\vert =1$, where it follows directly from the Ramsey theorem,
and $\vert A\vert =2$, where one can adapt the proof of the edge-Ramsey property of bipartite graphs.
For colouring hyperedges, Ne\v set\v ril and R\"odl used a smart application of the Hales--Jewett theorem~\cite{Nevsetvril1982}
and this proof was later generalized to colouring substructures~\cite{Nevsetvril1989}.

We briefly recall the Hales--Jewett theorem~\cite{Hales1963}.
Let $\Sigma$ be a finite set which we will call \emph{alphabet}.
A \emph{word} $w$ of length $N$ in the alphabet $\Sigma$ is a sequence of length $N$ of elements of $\Sigma$. Given $i<n$, we denote by $w_i$ the letter
on index $i$. (Note that the first letter of $w$ has index $0$.)
The \emph{$N$-dimensional combinatorial cube} (in alphabet $\Sigma$) is the set $\Sigma^N$ consisting of all words in the alphabet $\Sigma$ of length $N$.
If $\Sigma$ does not contain symbol $\lambda$, a \emph{parameter word}  of length $N$ in the alphabet $\Sigma$ is a word $W$ in the alphabet $\Sigma\cup \{\lambda\}$ for which there exists $i< N$ satisfying $W_i=\lambda$.
Given letter $a\in \Sigma$, we denote by $W(a)$ the word created from $W$ by replacing all occurrences of $\lambda$ by $a$ (the \emph{substitution} of $a$ to $W$).
If $W$ is a parameter word in alphabet $\Sigma$ of length $N$, we call the set $\{W(a):a\in \Sigma\}$ a \emph{combinatorial line} in $\Sigma^N$.
The Hales--Jewett theorem states that for every finite $\Sigma$ there exists an integer $N=\mathrm{HJ}(\vert\Sigma\vert)$ such that in every $2$-colouring of $\Sigma^N$ there exists a
monochromatic combinatorial line.~\cite{Hales1963}

A proof using the Hales--Jewett theorem was first used for edges only~\cite{Nevsetvril1982}, later in the dual form~\cite{frankl1987},
and eventually in the form presented here~\cite{Nevsetvril1989}.
\begin{proof}[Proof of Lemma~\ref{lem:partite}]
	\begin{figure}
		\centering
		\includegraphics{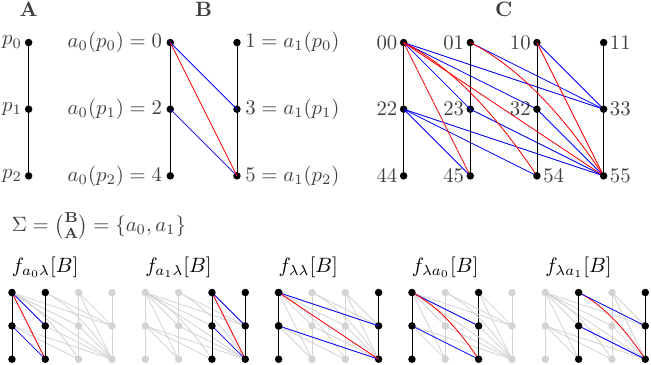}
		\caption{Example of the construction of $\str{C}$ in the proof of Lemma~\ref{lem:partite} for $N=2$ for graphs with three types of edges (black, red and blue).}
		\label{fig:partite1}
	\end{figure}
	Fix $L_P$, $\str{A}$ and $\str{B}$ as in the statement.
	Put $\Sigma=\Emb(\str{A},\str{B})$ and let $N=\mathrm{HJ}(\vert\Sigma\vert)$. For the rest of the proof, every word and parameter word will automatically be of length $N$ in the alphabet $\Sigma$.
	Our ultimate goal is to build an $L_P$-partite system $\str{C}$ such that every word $w$ will correspond to an embedding $e_w\colon \str{A}\to \str{C}$ and every parameter word $W$ will correspond to an embedding $f_W\colon \str{B}\to\str{C}$.
	A specific case of this construction is depicted in Figure~\ref{fig:partite1}. We first give all the necessary constructions, verifications will follow.

	\begin{description}
		\item[Vertex set of $\str{C}$:]~\\ Given $\pred{}{}\in A = L_P\setminus L$, let $C_{\pred{}{}}$ be the set of all functions $f\colon N\to B_{\pred{}{}}$. Put $C=\bigcup_{\pred{}{}\in A} C_{\pred{}{}}$.

		\item[Functions $e_w\colon A\to C$:]~\\ Given a word $w$, denote by $e_w$ the function $A\to C$ where for every $\pred{}{}\in A$ we have $e_w(\pred{}{})\in C_{\pred{}{}}$ defined by $e_w(\pred{}{})(i)=w_i(\pred{}{})$ for every $i<N$.

		\item[Functions $f_W \colon B\to C$:]~\\ Given a parameter word $W$, denote by $f_W$ the function $f_W\colon B\to C$
		      such that for every $p\in A$ and every $v\in B_p$  we have $f_W(v)\in C_p$ defined by putting, for every $i<N$,
		      $$f_W(v)(i)=\begin{cases}W_i(p)& \hbox{if }W_i\in \Sigma, \\ v & \hbox{if }W_i=\lambda.\end{cases}$$
		      Notice that $f_W$ is injective as $W$ contains $\lambda$.

		\item[Relations of $\str{C}$:]~\\
		      Relations of $\str{C}$ will contain precisely those tuples needed to make functions $f_W$ homomorphisms $\str{B}\to\str{C}$. Explicitly:
		      \begin{enumerate}
			      \item For every relation symbol $\rel{}{}\in L$, every tuple $\vec{x}=(x_0,x_1,\ldots, x_{n-1})\in \rel{B}{}$, and every parameter word $W$ we
			            put $f_W(\vec{x})=(f_W(x_0),f_W(x_1),\ldots, f_W(x_{n-1}))\in \rel{C}{}$.
			      \item For every $\pred{}{}\in A$ and every $v\in C_{\pred{}{}}$ we put $(v)\in \pred{C}{}$.
		      \end{enumerate}
	\end{description}
	\begin{claim}
		\label{clm1}
		For every parameter word $W$ it holds that $f_W\in \Emb(\str{B},\str{C})$.
	\end{claim}
	Fix a parameter word $W$. Since $W$ contains $\lambda$, $f_W$ is injective. By the definition of relations of $\str C$,
	we already know that $f_W$ is a homomorphism. Fix a relation symbol $R\in L$ and a tuple $\vec{x}$ of vertices of $\str{B}$ such that $f_W(\vec{x})\in \rel{C}{}$. It
	remains to check that $\vec{x}\in \rel{B}{}$. By the construction there is a parameter word $W'$ and a tuple $\vec{y}\in \rel{B}{}$ satisfying $f_W(\vec{x})=f_{W'}(\vec{y})$.
	Consider two cases (the second case is probably most difficult part of the proof):
	\begin{enumerate}
		\item There exists $i<N$ such that $W_i=W'_i=\lambda$.  In this case for every $j<\vert\vec{x}\vert$ it holds that $x_j=f_W(x_j)(i)=f_{W'}(y_j)(i)=y_j$. Consequently $\vec{x}=\vec{y}$ and we are done.
		\item There exist $i,j<N$ such that $W_i=\lambda$, $W'_i\neq \lambda$, $W_j\neq \lambda$, $W'_j=\lambda$.
		      Since $W'_i\in \Sigma$, it is an embedding $\str{A}\to \str{B}$, and $\vec{x}=W'_i(\pi_\str{B}(\vec{x}))$.
		      Similarly, $W_j\colon \str A\to \str B$ is an embedding and $\vec{y}=W_j(\pi_\str{B}(\vec{y}))$. Since $f_W(\vec{x})=f_{W'}(\vec{y})$, it follows that $\pi_\str{B}(\vec{x})=\pi_\str{B}(\vec{y})$, and as $\str A$ is transversal, we in fact obtain $\vec{x} = W'_i \circ W_j^{-1}(\vec{y})$. Consequently, $\vec{x}\in \rel{B}{}$.
	\end{enumerate}
	This finishes the proof of Claim~\ref{clm1}.
	By the construction we also get:
	\begin{claim}
		\label{clm3}
		For every parameter word $W$ and every embedding $\varphi\in \Sigma=\Emb(\str{A},\str{B})$ we have $f_W\circ \varphi=e_w$ where $w=W(\varphi)$.
		Consequently, for every word $w$ we have that $e_w\in \Emb(\str{A},\str{C})$.
	\end{claim}
	The ``consequently'' part follows from the fact that for every word $w$ there is a parameter word $W'$ and a letter $\varphi'\in \Sigma$ such that $w=W'(\varphi')$.
	By Claim~\ref{clm1} we know that $f_{W'}$ is an embedding $\str{B}\to\str{C}$ and, by choice of the alphabet, we know that $\varphi'$ is an embedding $\str{A}\to\str{B}$. Therefore
	$e_w=f_{W'}\circ \varphi'$ is an embedding $\str{A}\to\str{C}$.

	Using this we can conclude the proof of Lemma~\ref{lem:partite}. Fix a $2$-colouring $\chi$ of $\Emb(\str{A},\str{C})$. Define a colouring $\chi'$ of $\Sigma^N$ by putting $\chi'(w)=\chi(e_w)$.
	By the Hales--Jewett theorem we obtain a parameter word $W$ such that $\{W(\varphi):\varphi \in \Sigma\}$ is monochromatic. By Claim~\ref{clm3}, $f_W\in \Emb(\str{B},\str{C})$ is the desired embedding of $\str B$ to $\str C$ monochromatic with respect to $\chi$.
\end{proof}
Given an $L$-structure $\str{A}$, an $L_P$-partite system $\str{B}$, and an injection $\alpha\colon A\to L_P\setminus L$, we denote by $\Emb(\str{A},\str{B})_\alpha$ the set of all embeddings $e\colon \str{A}\to \str{B}\vert _L$ such that $\pi_\str{B}\circ e=\alpha$. We say that embeddings in $\Emb(\str{A},\str{B})_\alpha$ have \emph{projection} to $\alpha$. Note that one can define an $L_P$-partite system $\str A^\alpha$ from $\str A$ by putting each $v\in A$ into the predicate $\alpha(v)$ and then, clearly, there is a bijection between $\Emb(\str{A},\str{B})_\alpha$ and $\Emb(\str{A}^\alpha,\str{B})$. We use the $\Emb(\str{A},\str{B})_\alpha$ notation since we believe that it is slightly more convenient for our arguments.

\begin{lemma}[Picture lemma]
	\label{lem:picture}
	Let $L$ be a relational language, let $L_P$ extend $L$ by finitely many predicates, let $\str{A}$ be an $L$-structure, let $\str B$ be an $L_P$-partite system,
	and let $\alpha\colon A\to L_P\setminus L$ be an injection.  Then there exists an $L_P$-partite system $\str{C}$ such
	that for every $2$-colouring $\chi$ of $\Emb(\str{A},\str{C})_\alpha$ there exists an embedding $f\colon \str{B}\to\str{C}$ such that $\chi$ is constant on $\{f\circ e : e\in \Emb(\str{A},\str{B})_\alpha\}$.
\end{lemma}
\begin{proof}
	First consider the substructure $\str{B}'$ induced by $\str{B}$ on
	$$B'=\{v\in \str{B}:\pi_\str{B}(v)\in \alpha[A]\}.$$
	Let language $L_{P'}\subseteq L_P$ extend language $L$ by predicates $p_{\alpha(j)}$, $j\in A$.
	Create a transversal $L_{P'}$-partite system $\str{A}'$ such that $\alpha$ is an isomorphism $\str A\to \str A'\vert _L$, and every vertex $p\in A'$ is in the predicate $p_{\str A'}$.
	Consider $\str{B}'$ to be an $L_{P'}$-partite system, and let $\str{C}'$ be given by Lemma~\ref{lem:partite} applied on $\str{A}'$ and $\str{B}'$.

	Now, consider $\str{C}'$ as $L_P$-partite system and create $\str{C}$ by extending  $\str{C}'$ such that every embedding $f\colon \str{B}'\to \str{C}'$ extends to an embedding $f^+\colon \str{B} \to \str{C}$. This can be accomplished by series of free amalgamations (see Figure~\ref{fig:picturei}).
	\begin{figure}
		\centering
		\includegraphics{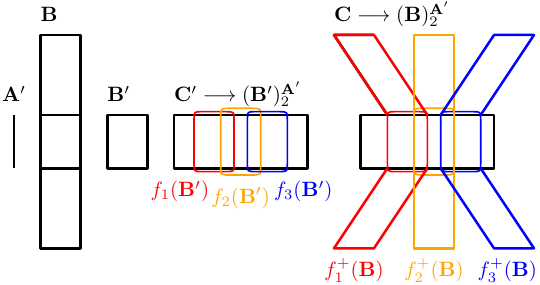}
		\caption{Construction of $\str{C}$ in the proof of Lemma~\ref{lem:picture}.}
		\label{fig:picturei}
	\end{figure}

	Given $\chi\colon \Emb(\str{A},\str{C})_\alpha\to 2$, let $\chi'$ be the colouring of $\Emb(\str{A}',\str{C}')$ defined by $\chi'(e)=\chi(e\circ \alpha)$ for every embedding $e\colon\str{A}'\to \str{B}'$. By Lemma~\ref{lem:partite}, there is an  embedding $f'\colon \str B'\to \str C'$ monochromatic with respect to $\chi'$. By the previous paragraph, there is an embedding $f\colon \str B\to \str C$ extending $f'$ which is the desired embedding.
\end{proof}

Now we are ready to prove Theorem~\ref{thm:unNR}.
\begin{proof}[Proof of Theorem~\ref{thm:unNR} using the Partite Construction]
	Fix a language $L$, finite ordered $L$-struc\-tures $\str{A}$ and $\str{B}$, and put $k=\vert A\vert$ and $n=\vert B\vert $.
	Without loss of generality assume that $A=k=\{0,1,\ldots, k-1\}$ and $B=n=\{0,1,\ldots,n-1\}$ and the orders $<_\str{A}$ and $<_\str{B}$
	correspond to the order of the integers.  Let $N$ be an integer satisfying $N\to (n)^k_2$ (given by the Ramsey theorem).
	Let $L_P$ extend language $L$ by predicates $\pred{}{0}$, $\pred{}{1}$, \ldots, $\pred{}{N-1}$
	and put $P=L\setminus L_P$ to be set of all these predicates. Order $P$ as $\pred{}{0}<\pred{}{1}<\cdots<\pred{}{N-1}$.

	Enumerate all order-preserving injections $\alpha\colon k\to P$ as $\alpha_0$, $\alpha_1$, \ldots, $\alpha_{m-1}$.	We shall define $L_P$-partite systems $\str{P}_0, \str{P}_1, \ldots, \str{P}_m$ (called \emph{pictures}) such that
	for every $0\leq i<m$ and every 2-colouring of $\Emb(\str A,\str{P}_{i+1})_{\alpha_i}$ there is an embedding $f_i\colon \str{P}_{i}\to \str{P}_{i+1}$ such that $\{f_i\circ e:e\in\Emb(\str A,\str{P}_{i})_{\alpha_i}\}$ is constant.
	\begin{figure}
		\centering
		\includegraphics{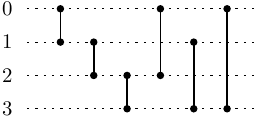}
		\caption{Picture $\str{P}_0$ for $N=4$ and $\str{B}$ being a graph edge.}
		\label{fig:P0}
	\end{figure}

	We start by constructing $\str{P}_0$ with the following property: for every order-preserving injection $\beta\colon n\to P$
	there exists an embedding $f_\beta\colon\str{B}\to \str{P}_0\vert _L$ such that $\pi_{\str{P}_0}\circ f_\beta=\beta$.  This can be easily done by starting with an empty $L_P$-partite system and, for every order-preserving injection $\beta\colon n\to P$, adding a disjoint copy of $\str{B}$ with vertices in the corresponding partitions. See Figure~\ref{fig:P0}.

	Finally, by repeatedly applying Lemma~\ref{lem:picture}, we obtain pictures $\str{P}_1$, $\str{P}_2$, \ldots, $\str{P}_m$.

	\begin{claim}
		$\str{P}_m\vert _L \longrightarrow (\str{B})^\str{A}_2$.
	\end{claim}
	Fix a $2$-colouring $\chi\colon \Emb(\str{A},\str{P}_m\vert _L)\to 2$.
	Using backward induction, obtain an embedding $\varphi\colon\str{P}_0\to \str{P}_m$
	such that colours of embeddings $\varphi\circ e$, $e\in \Emb(\str{A},\str{P}_0)$, depend only on $\pi_{\str{P}_m}\circ\varphi\circ e$.
	Since for every $e\in \Emb(\str{A},\str{P}_0)$ there exists $i$ such that $\pi_{\str{P}_0}\circ e=\alpha_i$, we obtain a $2$-colouring of the set $\{\alpha_0,\alpha_1,\ldots, \alpha_{m-1}\}$. Since for every $\alpha_i$ there is a unique $A_i\in \binom{N}{k}$ such that for every $j\in A$ we have that $\alpha_i(j)=p_{j'}$ where $j'$ is the $j$-th element of $A_i$, we obtain a $2$-colouring of $\binom{N}{k}$. Let $S\in \binom{N}{n}$ be a monochromatic subset.

	Let $\beta\colon\str{B}\to \str{P}_0$ be the embedding satisfying for every $j\in B$ that $\pi_{\str{P}_0}(\beta(j))=p_{j'}$ where $j'$ is the $j$-th element of $S$.  It follows that $\varphi\circ \beta$ is the desired monochromatic embedding $\str{B}\to \str{P}_m\vert _L$.

	Note that for every $i\leq m$ it holds that $\pi_{\str P_i} \colon \str P_i \to P$ is monotone, that is, if $x<_{\str P_i} y$ then $\pi_{\str P_i}(x) < \pi_{\str P_i}(y)$ in the order on $P$. This is true for $\str P_0$ as we consider only order-preserving injections $n\to P$, for $i>0$ this follows by induction using the fact that all $\alpha_i$'s are also order-preserving, as Lemma~\ref{lem:partite} preserves this, and Lemma~\ref{lem:picture} then just freely amalgamates structures with this property.

	Consequently, $<_{\str{P}_m}$ is acyclic, and so we can construct an ordered $L$-structure $\str{C}$ from $\str{P}_m\vert _L$ by extending $<_{\str{P}_m}$ arbitrarily.
\end{proof}

Note that the construction in the proof of Lemma~\ref{lem:partite} may create irreducible substructures which do not belong to any copy of $\str B$. For example, in Figure~\ref{fig:partite1} there is such a clique on
vertices $00$, $32$ and $55$. In order to prove Theorem~\ref{thm:NR}, one needs to address this problem. In the next section, we will see how to do it.

\subsection{Induced partite construction}\label{sec:induced}
Induced partite construction follows the same key steps as the non-induced variant with an important extra invariant
that partite systems project to well-defined structures. We will make this projection an essential part of the partite system as follows.
\begin{definition}[$\str{A}$-partite System]
	Let $L_P$ be a language extending $L$ by finitely many predicates
	and let $\str{A}$ be an $L$-structure with $A\subseteq L_P\setminus L$. We call
	an $L_P$-partite system $\str{B}$ an \emph{$\str{A}$-partite system} if $\pi_\str{B}$ is
	a homomorphism-embedding $\str{B}\vert _L\to\str{A}$.
\end{definition}
\begin{definition}[Powers of $\str{A}$-partite Systems]
	\label{def:power}
	Let $\str{B}$ be an $\str{A}$-partite system and let $N$ be an integer.  Then by $\str{B}^N$ we denote the \emph{$N$-th power of $\str{B}$}. This is the $\str{A}$-partite system $\str{C}$ defined as follows:
	\begin{enumerate}
		\item $C=\bigcup_{\pred{}{}\in A} C_{\pred{}{}}$, where $C_{\pred{}{}}$ is the set of all functions $f\colon N\to B_{\pred{}{}}$ for every $\pred{}{}\in A$.
		\item Given a relation $R\in L_P$ of arity $n$ we put $(x_0,\allowbreak x_1,\allowbreak \ldots,\allowbreak x_{n-1})\in \rel{C}{}$ if and only if for every $i<N$ it holds that $(x_0(i),\allowbreak x_1(i),\allowbreak \ldots,\allowbreak x_{n-1}(i))\in \rel{B}{}$. (Note that $R$ may also be a predicate.)
		\item Given a function $\func{}{}\in L_P$ of arity  $n$ we put $y\in \func{C}{}(x_0,\allowbreak x_1,\allowbreak \ldots,\allowbreak x_{n-1})$ if and only if for every $i<N$ it holds that $y(i)\in\func{B}{}(x_0(i),\allowbreak x_1(i),\allowbreak \ldots,\allowbreak x_{n-1}(i))$.
	\end{enumerate}
\end{definition}
$\str{B}^N$ will later serve as a Ramsey object. However, in most applications we will not need to consider the structure of $\str{B}^N$ and use only the fact that it is
an $\str{A}$-partite system.
The applications of the induced partite construction are often based on maintaining very subtle invariants
about individual pictures constructed, such as the following property:
\begin{definition}
	\label{def:based}
	Let $\str{D}$ be an $L$-structure, let $\str{A}\subseteq \str D$ be its substructure, and let $\str{P}$ be a $\str{D}$-partite system.
	Denote by $\str{P}\vert _{\str{A}}$ the subsystem induced by $\str{P}$ on its vertices with projection to $\str{A}$.
	We say that a $\str{D}$-partite system $\str{Q}$ is \emph{$(\str{P},\str{A})$-based}
	if there exists $N>0$ such that $\str{Q}$ is created from $(\str{P}\vert _\str{A})^N$ by extending,
	using free amalgamation, every embedding $\str{P}\vert _\str{A}\to (\str{P}\vert _\str{A})^N$ to an embedding $\str{P}\to \str{Q}$.
\end{definition}

The following can serve as a black-box for applications of the partite construction.
\begin{theorem}[Induced Partite Construction]
	\label{thm:inducedpartite}
	Let $L$ be a language, and let $\str{A}$, $\str{B}$, and $\str{D}$ be $L$-structures satisfying $\str{D}\longrightarrow (\str{B})^\str{A}_2$. Then there exists a $\str{D}$-partite system $\str{P}$ satisfying $$\str{P}\vert _L\longrightarrow (\str{B})^\str{A}_2.$$
	Moreover, there is a sequence of $\str{D}$-partite systems $\str{P}_0$, $\str{P}_1$, $\str{P}_2$, \ldots, $\str{P}_m=\str{P}$ satisfying:
	\begin{enumerate}
		\item\label{thm:inducedpartite:1} $\str{P}_0\vert _L$ is a disjoint union of multiple copies of $\str{B}$,
		\item\label{thm:inducedpartite:2} for every $i<m$ there exists an embedding $\alpha_i\colon \str{A}\to\str{D}$ such that $\str{P}_{i+1}$ is $(\str{P}_i,\alpha_i(\str{A}))$-based, and
		\item\label{thm:inducedpartite:3} for every $i\leq m$, if $\str E$ is an irreducible substructure of $\str P_i$ then there is $\str B'\subseteq \str D$ isomorphic to $\str B$ such that $\pi_{\str P_i}(\str E)$ is a substructure of $\str B'$.
	\end{enumerate}
\end{theorem}
We devote the rest of this section to a proof of Theorem~\ref{thm:inducedpartite}, which will have the same overall structure as the non-induced version from the previous section.  The reader is welcome to first enjoy the applications in Appendix~\ref{aaplications}.

\begin{lemma}[Induced Partite Lemma]
	\label{lem:indpartite}
	For every $L$-structure $\str A$ and every $\str{A}$-partite system $\str{B}$ there is $N\in \mathbb N$ such
	that $\str{B}^N\longrightarrow (\str{B})^{\str{A}'}_2$ where $\str{A'}$ is the $\str{A}$-partite system created from $A$ by putting every vertex $\pred{}{}\in \str{A}$ into the predicate $\pred{}{}$.
\end{lemma}
\begin{proof}
	Fix $\str{A}$ and $\str{B}$ as in the statement.
	Without loss of generality assume that for every vertex $v\in B$ there exists $f\in \Emb(\str{A},\str{B})$ such that $v\in f[A]$.
	Put $\Sigma=\Emb(\str{A},\str{B})$ and let $N=\mathrm{HJ}(\vert\Sigma\vert,2)$. From now on, all words and parameter words will be in the alphabet $\Sigma$ and have length $N$.  Put $\str{C}=\str{B}^N$. We claim that $$\str{C}\longrightarrow (\str{B})^\str{A}_2.$$
	To show this, we again associate every word $w$ with an embedding $e_w\colon \str{A}\to \str{C}$ and every parameter word $W$ with an embeddings $f_W\colon \str{B}\to\str{C}$ in the exactly same way as in the proof of Lemma~\ref{lem:partite}:
	\begin{description}
		\item[Functions $e_w\colon A\to C$:]~\\ Given a word $w$, denote by $e_w$ the function $A\to C$ where for every $\pred{}{}\in A$ we have $e_w(\pred{}{})\in C_{\pred{}{}}$ defined by $e_w(\pred{}{})(i)=w_i(\pred{}{})$ for every $i<N$.

		\item[Functions $f_W \colon B\to C$:]~\\ Given a parameter word $W$, denote by $f_W$ the function $f_W\colon B\to C$
		      such that for every $p\in A$ and every $v\in B_p$  we have $f_W(v)\in C_p$ defined by putting, for every $i<N$,
		      $$f_W(v)(i)=\begin{cases}W_i(p)& \hbox{if }W_i\in \Sigma, \\ v & \hbox{if }W_i=\lambda.\end{cases}$$
	\end{description}
	\begin{claim}
		For every parameter word $W$ it holds that $f_W\in \Emb(\str{B},\str{C})$.
	\end{claim}
	Fist, notice that that $f_W$ is injective as $W$ contains $\lambda$. It is easy to verify that $f_W$ preserves all predicates. It remains to prove that $f_W$ preserves all relations and functions from $L$.

	Fix a relation symbol $\rel{}{}\in L$, and a tuple $\vec{x}=(x_0,\allowbreak x_1,\allowbreak \ldots,\allowbreak  x_{n-1})$ of vertices of $\str{B}$.
	First assume that $\vec{x}\in \rel{B}{}$. To see that $f_W(\vec{x})\in \rel{C}{}$, by Definition~\ref{def:power} we need to verify that for every $i<N$ it holds that $(f_W(x_0)(i),\allowbreak f_W(x_1)(i),\allowbreak \ldots,\allowbreak f_W(x_{n-1})(i))\in \rel{B}{}$.

	If $W_i=\lambda$ then $f_W(x_j)(i)=x_j$ for every $j<n$ and we know that $\vec{x}\in \rel{B}{}$. If $W_i\in \Sigma$ then we know that $W_i$ is an embedding $\str{A}\to \str{B}$, and we can write $f_W(x_j)(i) = W_i\circ \pi_\str B(x_j)$. Note that since $\str B$ is an $\str A$-partite system, $\pi_\str B$ is a homomorphism-embedding $\str B\vert _L \to \str A$,
	and thus $W_i\circ \pi_\str B$ is a homomorphism-embedding as well, which implies that, indeed
	\begin{align*}
		(f_W & (x_0)(i),\allowbreak f_W(x_1)(i),\allowbreak \ldots,\allowbreak f_W(x_{n-1})(i))                                              \\
		     & = (W_i\circ \pi_\str B(x_0),\allowbreak W_i\circ \pi_\str B(x_1),\allowbreak \ldots,\allowbreak W_i\circ \pi_\str B(x_{n-1})) \\
		     & \in \rel{B}{}.
	\end{align*}

	If $\vec{x}\notin \rel{B}{}$, we can consider some index $i$ satisfying $W_i=\lambda$ and we obtain that $f_W(\vec{x})\notin \rel{C}{}$.
	For function symbols one can proceed analogously.
	Now that we established that functions $f_W$ are embeddings, the rest of the proof proceeds in a complete analogy to the non-induced version.
\end{proof}

\begin{lemma}[Induced Picture Lemma]\label{lem:indpicutre}
	Let $L$ be a language, let $\str{A}$ and $\str{D}$ be $L$-structures, let $\str{B}$ be a $\str{D}$-partite system, and let $\alpha\colon \str{A}\to \str{D}$ be an embedding.
	Then there exists a $(\str{B},\alpha(\str{A}))$-based $\str{D}$-partite system $\str{C}$ such
	that for every $2$-colouring $\chi$ of $\Emb(\str{A},\str{C})_\alpha$ there exists an embedding $f\colon\str{B}\to\str{C}$ such that $\chi$ restricted to $\{f\circ e:e\in \Emb(\str{A},\str B)_\alpha\}$ is constant.
\end{lemma}
\begin{proof}
	Construct $\str{C}$ exactly as in the proof of Lemma~\ref{lem:picture}, but use	Lemma~\ref{lem:indpartite} instead of Lemma~\ref{lem:partite}.
\end{proof}

\begin{proof}[Proof of Theorem~\ref{thm:inducedpartite}]
	We will proceed analogously as in the proof of Theorem~\ref{thm:NR} from the previous section. Fix $L$, $\str A$, $\str B$, and $\str D$ as in the statement. Let $\alpha_0$, $\alpha_1$, \ldots, $\alpha_{m-1}$ be an enumeration of all embeddings $\str A\to \str D$ with the property that there exists and embedding $\beta_i\colon \str B\to \str D$ such that $\alpha_i(\str A)\subseteq \beta_i(\str B)$. We will define $\str D$-partite systems $\str{P}_0, \str{P}_1, \ldots, \str{P}_m$ similarly as in the previous section.

	We start by constructing $\str{P}_0$ with the following property: for every embedding $\beta\colon \str B\to \str D$
	there exists an embedding $f_\beta\colon\str{B}\to \str{P}_0\vert _L$ such that $\pi_{\str{P}_0}\circ f_\beta=\beta$. This can be easily done by starting with an empty $\str{D}$-partite system and, for every $\beta\in \Emb(\str{B},\str{D})$, adding a disjoint copy of $\str{B}$ with vertices in the corresponding partitions.

	Finally, by repeated applications of Lemma~\ref{lem:indpicutre}, we obtain $\str D$-partite systems $\str{P}_1$, $\str{P}_2$, \ldots, $\str{P}_m$. It follows by the same argument as in the non-induced variant that $\str{P}_m\vert _L \longrightarrow (\str{B})^\str{A}_2$. Properties~(\ref{thm:inducedpartite:1}) and~(\ref{thm:inducedpartite:2}) follow directly from the construction.

	We will prove~(\ref{thm:inducedpartite:3}) by induction on $i$. For $i=0$ this trivially follows from the construction. Suppose now that it is true for some $i<m$ and let $\str E$ be an irreducible substructure of $\str P_{i+1}$. There are two possibilities: Either $\str E$ is a substructure of some copy of $\str P_i$ in $\str P_{i+1}$ (and then we can use the induction hypothesis), or $\str E \subseteq (\str{P}_i\vert _{\alpha_i(\str{A})})^N$, but then $\pi_{\str P_{i+1}}(\str E) \subseteq \alpha_i(\str A)$, and so it indeed extends to $\beta_i(\str B)$.
\end{proof}

\subsubsection{Examples of applications}
\label{aaplications}
As a warm-up before proceeding with the next general result we review two proofs which make essential use of the induced partite construction. We start with a proof of the Ne\v set\v ril--R\"odl theorem:
\begin{proof}[Proof of Theorem~\ref{thm:NR}]
	Fix $L$ and $\mathcal F$ as in the statement. Let $\str A$ and $\str B$ be finite ordered $\mathcal F$-free structures.
	Use Theorem~\ref{thm:unNR} to obtain an ordered $L$-structure $\str{D}$ satisfying $\str{D}\longrightarrow (\str{B})^\str{A}_2$,
	and then apply Theorem~\ref{thm:inducedpartite} to get a $\str D$-partite system $\str P$ such that $\str P\vert _L \longrightarrow (\str{B})^\str{A}_2$. Point~(\ref{thm:inducedpartite:3}) of Theorem~\ref{thm:inducedpartite} implies that every irreducible structure of $\str P\vert _L$ embeds to $\str B$, and consequently $\str P\vert _L$ is $\mathcal F$-free, as every $\str F\in \mathcal F$ is irreducible, and hence if there was an embedding $\str F\to \str \str P\vert _L$ then there would also be an embedding $\str F\to \str B$, a contradiction.

	Note that the projection is a homomorphism-embedding $\str{P}\to \str{D}$, and since $\str{D}$ is an ordered structure, we know
	that there is a linear order $<_\str C$ extending $<_\str{P}$. Put $\str C = (\str P\vert _{L\setminus <}, <_\str C)$.
	Note that if $x,y\in C$ are such that they are not ordered by $<_\str{P}$ then they are in no relations of $\str P$ together (as every pair in $\str B$ is ordered). Consequently, every embedding $f\colon \str B\to \str P\vert _L$ is also an embedding $\str B\to \str C$, and thus $\str C\longrightarrow (\str B)^\str A_2$. Moreover, since the $(L\setminus\{<\})$-reducts of members of $\mathcal F$ are irreducible, $\str C$ is also $\mathcal F$-free.
\end{proof}

The main strength of the partite construction is that one can maintain more interesting invariants.
What follows is Ne\v set\v ril and R\"odl's proof that the class of all finite partial orders with a linear extension is Ramsey.
\begin{proof}[Proof of Theorem~\ref{thm:posets}]
	Let $L$ contain two binary relations $<$ and $\ll$, and let $\str{A}$
	and $\str{B}$ be finite $L$-structures such that $(A,\ll_\str{A})$
	and $(B,\ll_\str{B})$ are partial orders and $(A,<_\str{A})$
	and $(B,<_\str{B})$ their respective linear extensions.
	Use Theorem~\ref{thm:unNR} to obtain an ordered $L$-structure $\str{D}$ satisfying
	$\str{D}\longrightarrow (\str{B})^\str{A}_2$.  Notice that while we know that $(\str{D},<_\str{D})$ is a linear
	order, we do not know much about $(\str{D},\ll_\str{D})$.

	We apply Theorem~\ref{thm:inducedpartite}.
	Let $\str{P}_0$, $\str{P}_1$, \ldots, $\str{P}_m=\str{P}$ be the sequence of
	pictures.  We verify by induction the following invariant:
	\begin{invariant}
		\label{inv2}
		Let $\str{P}_i$ be a picture and denote by $\ll'_{\str{P}_i}$ the transitive
		closure of $\ll_{\str{P}_i}$.  Then
		\begin{enumerate}
			\item If $(u,v)\in {\ll_{\str{P}_i}}$ then $(u,v)\in {<_{\str{P}_i}}$. Formally $ {\ll_{\str{P}_i}}\subseteq {<_{\str{P}_i}}$. Notice that neither $\ll_{\str{P}_i}$ nor $<_{\str{P}_i}$ are orders.
			\item If $(u,v)\in {\ll'_{\str{P}_i}} \setminus {\ll_{\str{P}_i}}$ then $(u,v)\notin {<_{\str{P}_i}}$. (Or in other words, ${\ll'_{\str{P}_i}}\cap {<_{\str{P}_i}}\subseteq {\ll_{\str{P}_i}}$.)
		\end{enumerate}
	\end{invariant}

	Invariant~\ref{inv2} is easy to check for $\str{P}_0$, since $\ll_{\str{P}_0}$ is transitive and consequently  $\ll'_{\str{P}_0} = \ll_{\str{P}_0}$.
	Assume that $\str{P}_i$, for $i<m$ satisfies Invariant~\ref{inv2}. Clearly,
	$\str{P}_{i+1}\vert _{\alpha_i(\str{A})}$ also satisfies Invariant~\ref{inv2} as it has homomorphism-embedding to $\str{A}$ where $\ll_\str{A}$ is a partial order and $<_\str{A}$ its linear extension. It remains to notice that Invariant~\ref{inv2} is preserved by free amalgamations.

	Note that $<_\str P$ has a homomorphism-embedding to $<_\str D$ which is a linear order. Consequently, $<_\str P$ is acyclic. The first point of Invariant~\ref{inv2} implies that $\ll_\str P$ is also acyclic. It remains to construct an $L$-structure $\str{C}$ by putting $C=P$, letting $<_\str{C}$ extend $<_\str{P}$ to a linear order, and putting $\ll_\str{C}$ to be the transitive closure of $\ll_\str{P}$.
	By the second point of Invariant~\ref{inv2} we know that every embedding $f\colon\str{B}\to\str{P}\vert _L$ is also an embedding $f\colon\str{B}\to\str{C}$.
\end{proof}
\subsection{Iterating the partite construction}\label{sec:iterated}
In this section we further develop the ideas from the previous section, iterate the induced partite construction and, in the end, prove Theorem~\ref{thm:sparseningRamsey}. We need to start with a definition.
\begin{definition}\label{defn:abnlocallytreelike}
	Let $L$ be a language, $\str A$, $\str B$, $\str C$ finite $L$-structures, and $n\geq 1$ an integer. We say that $\str C$ is \emph{$(\str A, \str B, n)$-locally tree-like} if for every substructure $\str{C}'$ of $\str{C}$ on at most $n$ vertices there exists a structure $\str{T}$ which is a tree amalgam of copies of $\str{B}$, and a homomorphism-embedding $f\colon\str{C}'\to\str{T}$ such that for every embedding $\alpha\colon \str A\to \str C$ there is an embedding $\alpha'\colon \str A\to \str T$ with $f[\alpha[A]\cap C']\subseteq \alpha'[A]$.
\end{definition}
In a typical situation, $\str A$ will be linearly ordered, and so every substructure of $\str A$ will be irreducible. In this case, if $\str{T}_0$ is any tree amalgam of copies of $\str{B}$ with a homomorphism-embedding $f\colon\str{C}'\to\str{T}_0$ then one can extend $\str T_0$ to $\str T$ satisfying Definition~\ref{defn:abnlocallytreelike} by amalgamating a copy of $\str B$ over every embedding of a substructure of $\str A$.

We now prove an extra invariant about the induced partite construction (Theorem~\ref{thm:inducedpartite}):
\begin{theorem}\label{thm:tree_invariant}
	Let $L$ be a language, and let $\str{A}$, $\str{B}$, and $\str{D}$ be $L$-structures such that $\str A$ is irreducible and $\str{D}\longrightarrow (\str{B})^\str{A}_2$. Assume that there is $n\geq 1$ such that $\str D$ is $(\str A, \str B, n-1)$-locally tree-like. Let $\str{P}_0$, $\str{P}_1$, $\str{P}_2$, \ldots, $\str{P}_m$ be the sequence of $\str{D}$-partite systems produced by Theorem~\ref{thm:inducedpartite}. Then for every $i\leq m$ it holds that $\str{P}_i\vert _L$ is $(\str A, \str B, n)$-locally tree-like.
\end{theorem}
The proof is essentially taken from~\cite{Hubicka2016}, even though~\cite{Hubicka2016} does not state the corresponding result in this way.
\begin{proof}
	Fix $\str{A}$, $\str{B}$, $\str{D}$, $n$, and $\str{P}_0$, $\str{P}_1$, \ldots, $\str{P}_m$ as in the statement. We will proceed by induction on $i$. The claim is easy to verify for $i=0$, as $\str{P}_0\vert _L$ is a disjoint union of multiple copies of $\str{B}$, and thus has a homomorphism-embedding to $\str B$, the rest follows from the fact $\str A$ is irreducible.

	Assume that
	$\str{P}_i\vert _L$, for some $i<m$, is $(\str A, \str B, n)$-locally tree-like. We need to verify that $\str{P}_{i+1}\vert _L$ is $(\str A, \str B, n)$-locally tree-like. Put $\str{B}_i=\str{P}_i\vert _{\alpha_i(\str{A})}$ and $\str{C}_{i+1}=\str{B}_i^N$.
	Choose a substructure $\str{C}'$ of $\str{P}_{i+1}\vert _L$ with at most $n$ vertices.
	Let $\str{C}''$ be structure induced by $\str{D}$ on $\pi_{\str{P}_{i+1}}[C']$.
	Consider two cases:
	\begin{enumerate}
		\item $\vert C''\vert<n$. Use the assumption on $\str{D}$ to get a homomorphism embedding $f\colon \str{C}''\to \str{T}$ where $\str{T}$ is a tree amalgam of copies of $\str{B}$. It follows that $f\circ \pi_{\str{P}_{i+1}}$ is a homomorphism-embedding $\str{C}'\to \str{T}$. Now, for an arbitrary embedding $\alpha\colon \str A\to \str{P}_{i+1}\vert _L$ we have that $\pi_{\str{P}_{i+1}}\circ \alpha$ is an embedding $\str A \to \str D$, and so there is an embedding $\alpha'\colon \str A\to \str T$ with
		      $$(f\circ \pi_{\str{P}_{i+1}})[\alpha[A] \cap C] \subseteq f[(\pi_{\str{P}_{i+1}}\circ \alpha)[A] \cap C''] \subseteq \alpha'[A]$$
		      and we are done.
		\item $\vert C''\vert=\vert C'\vert =n$.  Since $\str{P}_{i}$ is $(\str A, \str B, n)$-locally tree-like and $\str{C}_{i+1}$ has a homomorphism-embedding to $\str{A}$,
		      if there is an embedding of $\str{C}'$ to $\str{C}_{i+1}$ or $\str{P}_i$, we are done.
		      It remains to consider the case that $\str{C}'$ was created by amalgamating multiple copies of $\str{P}_i$ over $\str{C}_{i+1}$. In this case, there exist two proper substructures $\str{E}$ and $\str{F}$  of $\str{C}'$ such that $\pi_{\str{P}_{i+1}}[E\cap F]\subseteq \alpha_i[A]$ and $\str{C}'$ is the free amalgam of $\str{E}$ and $\str{F}$ over $\str{E}\cap \str{F}$.

		      Let $\str{E}'$ respectively $\str{F}'$ be the structures induced by $\str{D}$ on  $\pi_{\str{P}_{i+1}}[E]$ and $\pi_{\str{P}_{i+1}}[F]$ respectively.
		      Since $\vert\str{E}\vert=\vert\str{E}'\vert<n$ and $\vert\str{F}\vert=\vert\str{F}'\vert<n$, by our assumption on $\str{D}$ there exist tree amalgams $\str{T}_\str{E}$ and $\str{T}_\str{F}$ of copies of $\str{B}$ and homomorphism-embeddings $f\colon \str{E}\to \str{T}_\str{E}$ and $f'\colon \str{F}\to \str{T}_\str{F}$.
		      Now $\str{G} = \str{E}'\cap \str{F}'$ is a substructure of $\alpha_i(\str{A})$ and thus there are $\alpha'_\str E\colon \str A \to \str T_\str E$ and $\alpha'_\str F\colon \str A \to \str T_\str F$ such that $f[G] \subseteq \alpha'_\str E[A]$ and $f'[G] \subseteq \alpha'_\str F[A]$.  We can therefore construct structure $\str{T}$ as the free amalgam of $\str{T}_\str{E}$ and $\str{T}_\str{F}$
		      unifying $f(\str{G})$ and $f'(\str{G})$.  It follows $\str{T}$ is a tree amalgam of copies of $\str{B}$ and there is a homomorphism-embedding $\str{C}'\to \str{T}$.
		      \begin{figure}
			      \centering
			      \includegraphics{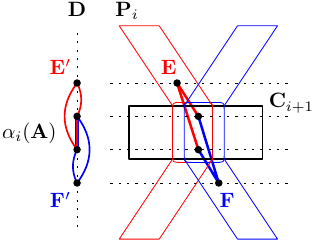}
			      \caption{Structures $\str{E}$, $\str{F}$, $\str{E}'$ and $\str{F'}$ in the proof of Theorem~\ref{thm:tree_invariant}.}
			      \label{fig:P0_iterated}
		      \end{figure}
	\end{enumerate}
\end{proof}

We can now prove Theorem~\ref{thm:sparseningRamsey}.
\begin{proof}[Proof of Theorem~\ref{thm:sparseningRamsey}]
	Without loss of generality we can assume that $n\geq 1$. Iterate Theorem~\ref{thm:inducedpartite} $n$ times to construct a sequence $\str C_0, \str C_1, \ldots, \str C_n$ (where in the $i$-th application we plug in $\str C_{i-1}$ as $\str D$ and we put $\str C_i$ to be the output structure $\str P\vert _L$). Clearly, we have homomorphisms embeddings $\pi_{\str C_i}\colon \str C_{i} \to \str C_{i-1}$ for $1\leq i\leq n$, and composing them we get a homomorphism-embedding $\pi\colon \str C_n \to \str C_0$. By Theorem~\ref{thm:tree_invariant} we know that $\str C_i$ is $(\str A, \str B, i)$-locally tree-like for every $0\leq i\leq n$. In particular, $\str C_n$ is $(\str A, \str B, n)$-locally tree-like which is stronger than property~(\ref{thm:sparseningRamsey:2}) in Theorem~\ref{thm:sparseningRamsey}.

	Finally, Theorem~\ref{thm:inducedpartite} promises us for $1\leq i \leq n$ that if $\str E$ is an irreducible substructure of $\str C_i$ then $\pi_{\str C_i}(\str E)$ extends to a copy of $\str B$ in $\str C_{i-1}$. By composing and using the fact that every $\pi_{\str C_i}$ is a homomorphism-embedding, we get that $\str E$ is an irreducible substructure of $\str C_n$ then $\pi(\str E)$ extends to a copy of $\str B$ in $\str C_{0}$. Consequently, we can obtain $\str C$ from $\str C_n$ satisfying property~(\ref{thm:sparseningRamsey:3}) of Theorem~\ref{thm:sparseningRamsey} by freely amalgamating a copy of $\str B$ over each irreducible substructure while preserving the existence of a homomorphism-embedding to $\str C_0$ and the existence of homomorphism-embeddings to tree amalgams of copies $\str B$.
\end{proof}

\subsection{Recursive partite construction}
\label{arecursive}
It only remains to obtain an analogue of the unrestricted Nešetřil--Rödl theorem for structures in languages with both relation and function symbols.
Our starting point will be the unrestricted Ne\v set\v ril--R\"odl theorem which works with relational languages only.
For this reason we will interpret structures in a language $L=L_R\cup L_F$, where $L_R$ are relation symbols and $L_F$ function symbols, using the relational language $L'=L_R\cup U$ where for every function symbol $\func{}{}\in L_F$ of arity $n$ we put the relation symbol $\rel{}{\func{}{}}$ of arity $n+1$ to $U$.
Given an $L'$-structure $\str{A}'$ we will then reconstruct an $L$-structure $\str{A}$ by defining functions $\func{A}{}$, $\func{}{}\in L_F$ using relations $\nbrel{\str A'}{\func{}{}}$ as follows:
$$\func{A}{}(x_0,x_1,\ldots, x_{n-1})=\{x_n\in A:(x_0,x_1,\ldots, x_{n-1},x_n)\in \nbrel{\str A'}{\func{}{}}\}.$$
For a structure $\str{A}$ with function symbols it is not necessarily true that $\str A$ induces a substructure on every subset $S\subseteq\str{A}$, as substructures need to be closed for function values.
This corresponds to the following notion of closure for $L'$-structures:
\begin{definition}
	Let $L$ be a relational language, let $U\subseteq L$ be a set whose elements have arity at least 2, and let $\str B$ be an $L$-structure. A substructure $\str{A}$ of $\str{B}$ is \emph{$U$-closed} if for every  $\rel{}{}\in U$ of arity $n+1$ and every tuple $(x_0,x_1,\ldots, x_{n-1},x_n)\in \rel{B}{}$ satisfying $x_0,x_1,\ldots,x_{n-1}\in A$ it holds $x_n\in A$.
\end{definition}
An embedding $f\colon\str{A}\to\str{B}$ is \emph{$U$-closed} if $f(\str{A})$ is a $U$-closed substructure of $\str{B}$.
(Embeddings of structures with functions correspond one-to-one to $U$-closed embeddings of their relational interpretations.)

We will denote by $\Emb(\str{A},\str{B})_U$ the set of all $U$-closed embeddings $\str{A}\to\str{B}$.
We will also write $\str{C}\longrightarrow_U (\str{B})^\str{A}_2$ for the following statement:
\begin{quote}
	For every $2$-colouring $\chi\colon\Emb(\str{A},\str{C})_U\to 2$, there exists a
	monochromatic embedding $f\in \Emb(\str{B},\str{C})_U$.
\end{quote}
We will prove:
\begin{theorem}
	\label{thm:models2}
	Let $L$ be a relational language containing a binary relation $<$, let $U\subseteq L\setminus\{<\}$ be a set of relation symbols of arity at least 2, and let $\K$ be the class of all finite ordered $L$-structures. Then
	$$(\forall \str{A},\str{B}\in \K)(\exists \str{C}\in \K)\str{C}\longrightarrow_U (\str{B})^\str{A}_2.$$
\end{theorem}
Before we prove Theorem~\ref{thm:models2}, we show how it implies Theorem~\ref{thm:HN}:
\begin{proof}[Proof of Theorem~\ref{thm:HN}]
	Fix $L$ and $\mathcal F$ as in the statement. Fix also some finite ordered $\mathcal F$-free $L$-structures $\str A$ and $\str B$.

	Let $L' = L_R\cup U$ be the language constructed above for $L$ and use Theorem~\ref{thm:models2} to obtain a finite ordered $L'$-structure $\str C_0'$ such that $\str C_0'\longrightarrow_U (\str{B})^\str{A}_2$, where $\str A$ and $\str B$ are understood as $L'$-structures. Let $\str C_0$ be the $L$-structure obtained from $\str C_0'$ as above. Clearly, $\str C_0\longrightarrow (\str B)^\str A_2$. However, $\str C_0$ need not be $\mathcal F$-free.

	Now we are in the exactly same situation as when we used the induced partite construction to prove Theorem~\ref{thm:NR} from Theorem~\ref{thm:unNR}, and it turns out that none of the arguments in the proof in Section~\ref{aaplications} used the fact that $L$ was relational. We can thus use verbatim the same proof (with the exception of using the previous paragraphs instead of Theorem~\ref{thm:unNR}) as the proof of Theorem~\ref{thm:NR} from Section~\ref{aaplications} to conclude this proof of Theorem~\ref{thm:HN}.
\end{proof}

\medskip

In the rest of this section we will prove Theorem~\ref{thm:models2}. Fix $L$ and $U$ as in the statement of Theorem~\ref{thm:models2}. We, for the last time, will follow the key steps of the partite construction.

We now revisit definitions and constructions from Section~\ref{sec:induced} and introduce their $U$-closed variants.
An $\str{A}$-partite system $\str{P}$ is \emph{$U$-transversal} if for every
$\rel{}{\func{}{}}\in U$ of arity $n+1$ and every $x_0,x_1,\ldots, x_{n-1}\in P$ it holds that the set $S=\{x_n:(x_0,x_1,\ldots, x_{n-1}, x_n) \in \rel{P}{\func{}{}}\}$ is transversal, or in other words, $\vert\pi_{\str{P}}[S]\vert=\vert S\vert$.

\begin{lemma}[Induced Partite Lemma with Closures]
	\label{lem:indpartiteU}
	For every $L$-structure $\str A$ and every $U$-transversal $\str{A}$-partite system $\str{B}$ there is $N\in \mathbb N$ such
	that $\str{B}^N$ is a $U$-transversal $\str{A}$-partite system and $\str{B}^N\longrightarrow_U (\str{B})^{\str{A}'}_2$ where $\str{A'}$ is the $\str{A}$-partite system created from $A$ by putting every vertex $\pred{}{}\in \str{A}$ into the predicate $\pred{}{}$.
\end{lemma}
The proof is an easy extension of the proof of Lemma~\ref{lem:indpartite}: After putting $\Sigma=\Emb(\str{A},\str{B})_U$ one only needs to check that $\str{B}^N$ is $U$-transversal and that the embeddings $e_w$ and $f_W$ are $U$-closed.

Notice the following:
\begin{observation}
	\label{obs:disaster1}
	Let $\str{P}$ be a $U$-transversal $\str{D}$-partite system and let $\str A$ be a substructure of $\str{D}$. If $\str{A}$ is $U$-closed in $\str{D}$ then $\str{P}\vert _\str{A}$ is a $U$-closed substructure of $\str{P}$.
\end{observation}
Note that if $\str{A}$ is not $U$-closed then $\str{P}\vert _\str{A}$ need not be $U$-closed.
\begin{observation}
	\label{obs:disaster2}
	If $\str{C}$ is the free amalgam of $\str{B}$ and $\str{B}'$ over $\str{A}$ then $\str{B}$ and $\str{B}'$ are $U$-closed in $\str{C}$ if and only if $\str{A}$ is $U$-closed in both $\str{B}$ and $\str{B}'$.
\end{observation}
By Observations~\ref{obs:disaster1} and~\ref{obs:disaster2}, the usual construction used to prove the Picture lemma will yield $U$-closed embeddings
only if $\str{P}\vert _\str{A}$ is $U$-closed, and this is typically true only if $\str{A}$ is $U$-closed in $\str{D}$. This is the main difficulty we will need to deal with, and in the following we need to carefully track which embeddings are guaranteed to be $U$-closed.

Given a $\str{D}$-partite system $\str{B}$ and an embedding $\alpha\colon\str{A}\to\str{D}$ we will denote by $\Emb(\str{A},\allowbreak \str{B})_{U,\alpha}$ the set of all $U$-closed
embeddings $e\colon \str{A}\to \str{B}\vert _L$ such that $\pi_\str{B}\circ e=\alpha$.
The following observation will eventually let us ``upgrade'' embeddings to $U$-closed embeddings.
\begin{observation}
	\label{obs:solution}
	If $\str{C}$ is the free amalgam of $\str{B}$ and $\str{B}'$ over $\str{A}$, and $\str{A}'$ is a substructure of $\str{A}$
	then $\str{A}'$ is $U$-closed in $\str{C}$ if and only if $\str{A}'$ is $U$-closed in both $\str{B}$ and $\str{B}'$.
\end{observation}

\begin{lemma}[Induced Picture lemma with Closures]\label{lem:indpicutreU}
	Let $L$ be a language, let $\str{A}$ and $\str{D}$ be $L$-structures, let $\str{B}$ be a $U$-transversal $\str{D}$-partite system, and let $\alpha\colon \str{A}\to \str{D}$ be an embedding.
	Then there exists a $U$-transversal
	$\str{D}$-partite system $\str{C}$ such
	that for every $2$-colouring $\chi$ of $\Emb(\str{A},\str{C})_{U,\alpha}$ there exists an embedding $f\colon\str{B}\to\str{C}$ such that all embeddings $f\circ e$, $e\in \Emb(\str{A},\str{B})_U$ are $U$-closed and have the
	same colour.

	Moreover, if $\alpha$ is $U$-closed then $f$ can also be chosen to be $U$-closed.
\end{lemma}
\begin{proof}
	Do the same construction as always, but use Lemma~\ref{lem:indpartiteU} and then freely amalgamate copies of $\str B$ only over $U$-closed embeddings
	$\str{B}\vert _\str{A}\to (\str{B}\vert _\str{A})^N$.  Notice that Lemma~\ref{lem:indpartiteU} promises us that $f$ can be chosen such that all embedding $f\circ e$ are $U$-closed in $(\str{B}\vert _\str{A})^N$. By Observation~\ref{obs:solution} the process of extending $U$-closed embedding of $\str{B}\vert _\str{A}\to (\str{B}\vert _\str{A})^N$ to $\str{B}$ by free amalgamation preserves this property, and so all embeddings $f\circ e$, $e\in \Emb(\str{A},\str{B})_U$ are $U$-closed in $\str{C}$.

	The ``moreover'' part follows from Observations~\ref{obs:disaster1} and~\ref{obs:disaster2}.
\end{proof}
\begin{lemma}[Induced ``Half $U$-closed'' Partite Construction]
	\label{lem:rpartite}
	Let $L$ be a language, and let $\str{A}$, $\str{B}$, and $\str{D}$ be $L$-structures such that for every 2-colouring $\chi\colon \Emb(\str{A},\str{D})_U\to 2$ there
	exists a (not necessarily $U$-closed) embedding $f\colon\str{B}\to \str{D}$  such that all embeddings $f\circ e$, $e\in \Emb(\str{A},\str{B})_U$ are $U$-closed and have the
	same colour.

	Then there exists a $\str{D}$-partite $U$-transversal system $\str{P}$ satisfying $$\str{P}\vert _L\longrightarrow_U (\str{B})^\str{A}_2.$$
\end{lemma}
\begin{proof}
	We will follow the proof of Theorem~\ref{thm:inducedpartite}, this time using Lemma~\ref{lem:indpicutreU}.
	Fix $L$, $\str A$, $\str B$, and $\str D$ as in the statement. Let $\alpha_0$, $\alpha_1$, \ldots, $\alpha_{m-1}$ be an enumeration of all $U$-closed embeddings $\str A\to \str D$ with the property that there exist embeddings $\beta_i\colon \str B\to \str D$ such that $\alpha_i(\str A)\subseteq \beta_i(\str B)$. We will define $\str D$-partite systems $\str{P}_0, \str{P}_1, \ldots, \str{P}_m$ similarly as in the proof of Theorem~\ref{thm:inducedpartite}.

	The picture $\str{P}_0$ is constructed from a disjoint union of copies of $\str{B}$ as usual.
	Now, to obtain the subsequent pictures, one can apply Lemma~\ref{lem:indpicutreU}, and since we are concerned only with $U$-closed embeddings $\str{A}\to\str{D}$, we can apply the ``moreover'' part of Lemma~\ref{lem:indpicutreU} and maintain the invariant that all embeddings that we consider are $U$-closed.
\end{proof}

\medskip

Now that our revisionism of Section~\ref{sec:induced} is complete, we are ready to finish the proof.
\begin{proof}[Proof of Theorem~\ref{thm:models2}]
	We will proceed analogously to the proofs of Theorems~\ref{thm:NR} and~\ref{thm:inducedpartite}. Fix $L$, $\str A$, $\str B$ as in the statement.
	Apply Theorem~\ref{thm:unNR} to obtain $\str{D}$ satisfying $\str{D}\longrightarrow (\str{B})^\str{A}_2$.
	Let $\alpha_0$, $\alpha_1$, \ldots, $\alpha_{m-1}$ be an enumeration of all (not necessarily $U$-closed) embeddings $\str A\to \str D$.
	We shall define $U$-transversal  $\str{D}$-partite systems $\str{P}_0, \str{P}_1, \ldots, \str{P}_m$ such that
	for every $0\leq i<m$ and every 2-colouring $\chi$ of $\Emb(\str A,\str{P}_{i+1})_{U,\alpha_i}$ there is a $U$-closed embedding $f_i\colon \str{P}_{i}\to \str{P}_{i+1}$ such that $\chi$ is constant on $\{f_i\circ e:\Emb(\str A,\str{P}_{i})_{U,\alpha_i}\}$.

	We start (as usual) by constructing $\str{P}_0$ with the following property: for every embedding $\beta\colon \str B\to \str D$
	there exists a $U$-closed embedding $f_\beta\colon\str{B}\to \str{P}_0\vert _L$ such that $\pi_{\str{P}_0}\circ f_\beta=\beta$. This can be easily done by starting with an empty $\str D$-partite system and, for every $\beta\in \Emb(\str{B},\str{D})$, adding a disjoint copy of $\str{B}$ with vertices in the corresponding partitions.

	Next we build pictures $\str{P}_1$, $\str{P}_2$, \ldots, $\str{P}_n$.
	Assume that $\str{P}_i$ is constructed.
	Apply Lemma~\ref{lem:indpicutreU} for $\str A$, $\str D$, $\str{P}_i$ and $\alpha_i$ to obtain $\str{O}_{i+1}$.
	(The letter $O$ comes from the Czech word ``obrázek'' which means ``little picture''.)  This time, we did not assume that $\alpha_i$ is $U$-closed and thus we only know by Lemma~\ref{lem:indpicutreU} that for every 2-colouring of $\Emb(\str{A},\str{O}_{i+1})_{U,\alpha_i}$ there is a (not necessarily $U$-closed) embedding $f\colon\str{P}_{i}\to \str{O}_{i+1}$ such that all embeddings from $\{f\circ e:e\in \Emb(\str{A},\str{P}_i)_{U,\alpha_i}\}$ are $U$-closed and have the same colour.

	At this point, we remember that $\str{O}_i$ is an $L_P$-structure, and without loss of generality we can assume that
	every vertex in $O_{i+1}$ and every tuple of every relation of $\str{O}_{i+1}$ is contained in some $U$-closed image
	of $\str{P}_i$ in $\str{O}_{i+1}$ (as otherwise we can remove them without losing the Ramsey property of $\str O_{i+1}$).
	Create a $\str{D}$-partite system $\str{A}_i$ from the $L$-structure $\alpha_i(\str{A})$ by putting vertices into the corresponding predicates.
	Notice that $\Emb(\str{A}_i,\str{O}_{i+1})_U=\Emb(\str{A},\str{O}_{i+1})_{U,\alpha_i}$ and we can thus apply Lemma~\ref{lem:rpartite} on $\str{A}_i$, $\str{P}_i$ and $\str{O}_{i+1}$ (seen as $L_P$-structures and not as $\str{D}$-partite systems) to obtain an $\str O_{i+1}$-partite system $\str{P}_{i+1}^+$. Putting $\str P_{i+1} = \str{P}_{i+1}^+\vert _{L_P}$, Lemma~\ref{lem:rpartite} tells us that $\str{P}_{i+1}\longrightarrow_U (\str{P}_i)^{\str{A}_i}_2$.

	Clearly, $\str P_{i+1}$ is a $\str D$-partite system (there are homomorphism-embeddings $\str P_{i+1}\to \str O_{i+1}$ and $\str O_{i+1}\to \str D$). Since every vertex in $O_{i+1}$ and every tuple of every relation of $\str{O}_{i+1}$ is contained in some $U$-closed image
	of $\str{P}_i$ in $\str{O}_{i+1}$, it follows that $\str{P}_{i+1}$ is a $U$-transversal, since $\str P_i$ is $U$-transversal.
	Re-interpreting the properties of $\str{P}_{i+1}$ back to the partite-system setting, we get that for every $2$-colouring $\chi$ of
	$\Emb(\str A,\str{P}_{i+1})_{U,\alpha_i}$ there is a $U$-closed embedding $f_i\colon \str{P}_{i}\to \str{P}_{i+1}$ such that $\{f_i\circ e:e\in \Emb(\str A,\str{P}_{i})_{U,\alpha_i}\}$ is constant.

	The rest of the proof follows in full analogy to the proof of Theorem~\ref{thm:inducedpartite}.
\end{proof}

\bibliographystyle{alpha}
\bibliography{ramsey.bib}

\end{document}